\documentclass[12pt,leqno]{amsart}

\usepackage{amssymb}
\usepackage{amsthm}
\usepackage{amsmath}
\usepackage{amscd}
\usepackage[mathscr]{eucal}
\usepackage{verbatim}
\usepackage{epsfig}

\numberwithin{equation}{section}

\theoremstyle{plain}
\newtheorem{theorem}{Theorem}[section]
\newtheorem{corollary}[theorem]{Corollary}
\newtheorem{lemma}[theorem]{Lemma}
\newtheorem{proposition}[theorem]{Proposition}

\theoremstyle{definition}
\newtheorem{definition}[theorem]{Definition}
\newtheorem{remark}[theorem]{Remark}
\newtheorem{example}[theorem]{Example}

\theoremstyle{remark}

\newcommand{\R}{\mathbb{R}}
\newcommand{\Q}{\mathbb{Q}}
\newcommand{\Z}{\mathbb{Z}}

\newcommand{\C}{\mathbb{C}}
\newcommand{\h}{\mathbb{H}}
\renewcommand{\H}{\mathbb{H}}

\newcommand{\E}{\mathbb{E}}

\newcommand{\G}{\Gamma}
\newcommand{\g}{\gamma}

\newcommand{\La}{\Lambda}
\newcommand{\x}{\mathbf{x}}
\newcommand{\y}{\mathbf{y}}
\newcommand{\w}{\mathbf{w}}

\newcommand{\back}{\backslash}

\newcommand{\V}{V(\Q)}

\newcommand{\Hk}{\mathcal{H}_k}

\newcommand{\zxz}[4]{\begin{pmatrix} #1 & #2 \\ #3 & #4 \end{pmatrix}}

\newcommand{\kzxz}[4]{\left(\begin{smallmatrix} #1 & #2 \\ #3 & #4\end{smallmatrix}\right) }
\newcommand{\kabcd}{\kzxz{a}{b}{c}{d}}

\newcommand{\wwedge}[1]{\sideset{}{^{#1}}\bigwedge}

\newcommand{\calA}{\mathcal{A}}

\newcommand{\calD}{\mathcal{D}}
\newcommand{\calE}{\mathcal{E}}
\newcommand{\calF}{\mathcal{F}}

\newcommand{\calH}{\mathcal{H}}

\newcommand{\calL}{\mathcal{L}}

\newcommand{\calS}{\mathcal{S}}

\newcommand{\frake}{\mathfrak e}

\newcommand{\eps}{\varepsilon}

\newcommand{\vol}{\operatorname{vol}}
\newcommand{\tr}{\operatorname{tr}}

\newcommand{\sgn}{\operatorname{sgn}}

\newcommand{\Span}{\operatorname{span}}

\newcommand{\Sl}{\operatorname{SL}}

\newcommand{\SL}{\operatorname{SL}}

\newcommand{\Spin}{\operatorname{Spin}}
\newcommand{\PSL}{\operatorname{PSL}}
\newcommand{\Symp}{\operatorname{Sp}}
\newcommand{\Mp}{\operatorname{Mp}}
\newcommand{\Orth}{\operatorname{O}}
\newcommand{\Uni}{\operatorname{U}}

\newcommand{\Sym}{\operatorname{Sym}}

\newcommand{\SO}{\operatorname{SO}}

\newcommand{\Iso}{\operatorname{Iso}}

\newcommand{\PD}{\operatorname{PD}}

\newcommand{\mult}{\operatorname{mult}}

\begin{document}

\title[Spectacle cycles and modular forms]{Spectacle cycles with coefficients and modular forms of half-integral weight}

\author[Jens Funke and John Millson]{Jens Funke* and John Millson**}
\thanks{* Partially supported by NSF grant DMS-0710228}
\thanks{** Partially supported by NSF grant DMS-0907446, NSF FRG grant DMS-0554254, and the Simons Foundation}
\address{Department of Mathematical Sciences, University of Durham, Science Laboratories,
South Rd, Durham DH1 3LE, United Kingdom}
\email{jens.funke@durham.ac.uk}
\address{Department of Mathematics, University of Maryland, College Park, MD
20742, USA} \email{jjm@math.umd.edu}

\date{\today}

\dedicatory{To our friend Steve Kudla}

\begin{abstract}
In this paper we present a geometric way to extend the Shintani lift from even weight cusp forms for congruence subgroups to arbitrary modular forms, in particular Eisenstein series. This is part of our efforts to extend in the noncompact situation the results of Kudla-Millson and Funke-Millson relating Fourier coefficients of (Siegel) modular forms with intersection numbers of cycles (with coefficients) on orthogonal locally symmetric spaces. In the present paper, the cycles in question are the classical modular symbols {\it with nontrivial coefficients}. We introduce ``capped'' modular symbols with coefficients which we call ``spectacle cycles'' and show that the generating series of cohomological periods of any modular form over the spectacle cycles is a modular form of half-integral weight. 
In the last section of the paper we develop a new simplicial homology theory with  local
coefficients (that are not locally constant)  that allows us to extend the above results to orbifold quotients of the upper half
plane.
\end{abstract}

\maketitle

\section{Introduction}\label{intro}

In a series of articles from 1979 to 1990 Steve Kudla and the second named author developed a theory to  explain the occurrence of intersection numbers of geometrically defined cycles as Fourier coefficients of automorphic forms from the point of view of Riemannian geometry and the theory of reductive dual pairs and the theta correspondence (see eg \cite{KM90}). Their program was motivated by the work of Hirzebruch-Zagier \cite{HZ} on `Hirzebruch-Zagier' curves in Hilbert modular surfaces and Shintani \cite{Shin} for the `classical' modular symbols inside modular curves. They obtain analogues of the results of \cite{HZ} and \cite{Shin} for orthogonal, unitary, and symplectic groups of arbitrary dimension and signature. In particular, their work gives rise to a lift from the cohomology with compact supports for the associated locally symmetric spaces to spaces of holomorphic Siegel and Hermitian modular forms. Note however that the restriction to cohomology with compact supports implies that their results actually do not include the one obtained by Hirzebruch-Zagier (which deals with a smooth compactification of the Hilbert modular surface).

The typical shape of the results of Kudla-Millson for the dual pair $\Orth(p,q) \times \SL_2$ is as follows. There exists a theta series $\theta(\tau,\varphi)$ (associated to a carefully chosen vector-valued Schwartz function $\varphi$ on $\R^{(p,q)}$) with values in the closed differential forms on $X$, an appropriate, typically non-compact, arithmetic quotient of the orthogonal symmetric space, such that its cohomology class
\[
[\theta(\tau,\varphi)] = \sum_{n \geq 0} \PD(C_n)e^{2\pi i n\tau}
\]
is a holomorphic modular form of weight $(p+q)/2$ for $\SL_2$ (in $\tau \in \h$, the upper half plane) with values in the cohomology of $X$ with trivial coefficients. Here $\PD(C_n)$ are the Poincar\'e dual classes to the geometrically defined, totally geodesic, ``special'' cycles $C_n$  in $X$, parameterized by non-negative integers $n$. These cycles $C_n$ are (usually) non-compact and hence define in general homology classes relative to the (Borel-Serre) boundary. Then the (co)homological pairing of this generating series with the cohomology with compact supports or equivalently with absolute cycles gives rise to a theta lift from these (co)homology groups to classical modular forms.

\bigskip

The authors of present paper have been developing a program, see  \cite{FM1,FMcoeff,FMres,FM-HZ} (for an introductory overview also see \cite{FKyoto}), in which they seek to generalize the original work of Kudla-Millson in various directions. In \cite{FMcoeff}, they extended the lift to include non-trivial local coefficients systems. The main main goal however, is to extend the theta lift to cohomology groups associated to $X$ which capture its boundary.

Among the finite-volume non-compact quotient cases there is a family that appears to be amenable to attack using the techniques we have developed so far.
It is the family such that the theta functions $\theta(\varphi)$ (with potentially non-trivial coefficients) restricted to the Borel-Serre boundary of $X$ are exact. Equivalently, the special cycles that intersect the Borel-Serre boundary $\partial \overline{X}$ have intersections that are boundaries in $\partial \overline{X}$. In this case it appears that one may obtain a correction term given by another theta series $\theta(\phi)$ so that the pair $(\theta(\varphi), \theta(\phi))$ is a cocycle in the mapping cone de Rham complex associated to the pair $(\overline{X}, \partial \overline{X})$. This pair hence corresponds to a cohomology class on $X$ with compact supports whose image in the absolute cohomology coincides with the class of $\theta(\varphi)$. On the cycle level this construction corresponds to capping off the relative cycles $C_n $ at the boundary to form absolute cycles $C_n^c$ {\it which are homologous to $C_n$ as relative cycles}. 
In this way one can then extend the lift to the full cohomology of $X$.

\medskip

We have in fact implemented this procedure in \cite{FM-HZ} to reprove the main result of Hirzebruch and Zagier \cite{HZ} and give a topological interpretation of the remarkable fact of the authors' proof in which the desired generating series is expressed as the difference of two non-holomorphic modular forms. 

\medskip

Our purpose in this paper is to deal with the most basic case of all, namely geodesics {\it with coefficients} in modular curves, that is, we consider and extend Shintani's work \cite{Shin}. In our set-up this is the case of $\SO(2,1)$ whose arithmetic quotients $X$ via the special isomorphism with $\SL_2$ we can interpret as modular curves (in the non-compact case). Furthermore, the special cycles $C_n$ are closed or infinite geodesics (in the latter case these are the classical modular symbols). In \cite{FMcoeff} we explain (in much greater generality) how one can associate coefficient systems to the cycles $C_n$. Namely, we let $E_{2k}$ be the $2k$-th symmetric power of the standard representation of $\SL_2$. We can then construct cycles with coefficients $C_{n,[k]} \in H_1(X, \partial X, \widetilde{E_{2k}})$, where $\widetilde{E_{2k}}$ is the local system associated to $E_{2k}$, see below. Our main result of \cite{FMcoeff} specialized to this case recovers Shintani's result \cite{Shin} and states that 
\begin{equation}\label{Intro-Shintani}
\sum_{n > 0} \\PD(C_{n,[k]})e^{2\pi i n\tau} \in S_{k+3/2}(\G') \otimes H^1(X,\widetilde{E_{2k}})
\end{equation}
is a cusp form of weight $k+3/2$ for a congruence subgroup $\G' \subset \SL_2(\Z)$. Note that Shintani formulates his result in terms of weighted periods of cusp forms $f$ of weight $2k+2$ over the geodesics. Via the map $ f \mapsto \eta_f := f(z)dz \otimes (z\frake_1+\frake_2)^{2k}$, which induces the Eichler-Shimura isomorphism, one can obtain our point of view. Here $\frake_i,i=1,2$ is the standard basis of $E$.

\medskip

In this situation a remarkable phenomenon occurs. The cycles with trivial coefficients (i.e., $k=0$) cannot be capped off or equivalently the theta functions with trivial coefficients cannot be corrected to make them into relative classes. On the other hand {\it for any irreducible non-trivial coefficient system $E_{2k}$ the cycles $C_{n,[k]}$ can be capped off}. Equivalently, for non-trivial coefficients the forms $\theta(\varphi)$ of \cite{FMcoeff} can be corrected to be compactly-supported. The actual procedure of capping off a modular symbol (when it is an infinite geodesic joining two cusps) with coefficients produces a `spectacle' equipped with parallel sections of the coefficent system which we have named a `spectacle cycle (with coefficients)'.

\begin{figure}
\begin{center}
\caption{A spectacle chain}
\label{frequencies}
\epsfig{file=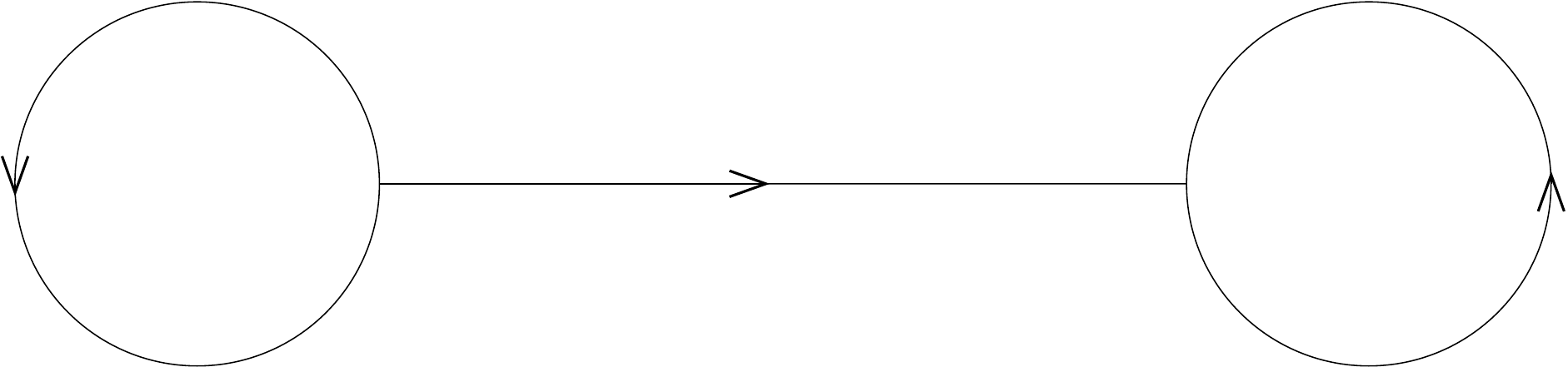,width=5in,height=3cm}
\end{center}
\end{figure}

The oriented graph $S$ of Figure 1 is of course not a cycle. It only becomes a cycle when local coefficients are added in a way that we now describe. Let $p_+$, resp. $p_-$, be the intersection
of the straight part $a$ of the spectacles with the right-hand circle $c_+$, resp. the left hand circle  $c_-$.  Suppose $S$ is embedded in a topological space equipped with a flat bundle $\E$. To promote $S$ to a cycle with coefficients in $\E$ we need a parallel section $s$ on $a$, a parallel section on $c_+$ with a jump at $p_+$ whose value is the negative of $s(p_+)$ and a parallel section on $c_-$ with a jump at $q$ with value equal to $s(p_-)$.  It is a remarkable fact (see the discussion below) that such cycles can be found implicitly in classical complex analysis (with $\E$ a flat complex line bundle). 

In our situation, we consider the Borel-Serre compactification $\overline{X}$ of the modular curve $X = \G \back \h$ which adds a boundary circle to each cusp and $\E=\widetilde{E_{2k}}$. Furthermore, $a$ is a modular symbol given by an infinite geodesic in $X$ joining two cusps, and the section $s=s_v$ on $a$ takes the constant value $v \in E_{2k}$. The circles $c_\pm$ are the boundary circles at the cusps joined by $a$. Then the sections on $c_\pm$ arise from solutions to the equation $(\g_\pm^{-1} -Id)w =v$ (which are not unique, and don't exist if $v$ is a lowest weight vector in $E_{2k}$ with respect to the Borel attached to the given cusp). Here $\g_\pm$ are the (properly oriented) generators of the stabilizer of the cusps $c_\pm$ in $\G$.

\medskip

The idea of adding `spectacles' to convert a locally-finite cycle with coefficients
`joining' two punctures (or equivalently $1$-cycles with coefficients relative to small circles surrounding the punctures) into a closed cycle with coefficients {\it in the same relative cohomology class} comes from Deligne and Mostow  \cite{DM} - see the picture at the top of page 14 (Deligne and Mostow have spectacles with rectangular lenses). 
Deligne and Mostow point out the classical antecedents of their construction in the contour integral formula for the $\G$-function. The gamma and hypergeometric integrands
should be regarded as {\it single-valued} differential one-forms with values in a one-dimensional
flat complex line bundle. The integral of such a form over a contour requires
a section of the dual bundle over that contour. In this paper we are dealing with differential $1$-forms and $1$-cycles with values in local systems of dimension higher than one that are locally homogeneous for the group $\SL(2,\R)$.

The spectacle construction in our situation appeared earlier in the work of Harder and his school, see eg Kaiser's Diplomarbeit \cite{Kaiser}, in their investigation of (the denominator of) Eisenstein cohomology. They discuss the cases $\SL_2(\Z)$ \cite{Harder} and $\G_1(p)$ \cite{Kaiser} in detail and integrate Eisenstein cohomology classes over the spectacle cycles.

\medskip

Throughout the paper we use the language and set-up of a rational quadratic space $V$ of signature $(2,1)$ whose arithmetic quotient gives rise to a modular curve $X = \G \back \h$, see Section~\ref{prelim}. This is necessary for our theta series construction, and is of course also the set-up in Shintani's paper. In Section~\ref{Spec-Section}, we construct the special cycles in our context and the associated spectacle cycles. The special cycles $C_\x$ arise from a rational vector of positive length $\x \in V$, and there is a natural choice to `promote' this cycle to one with coefficients $C_{\x,[k]}$. Then $C_{n,[k]}$ is obtained by summing over a set of representatives of $\G$-equivalence classes of vectors of length $n$ in a coset of an even lattice in $V$. 

The main result for the construction of the spectacle cycles is Lemma~\ref{v_x-lemma}. A rough summary of our considerations in Section~\ref{Spec-Section} is 

\begin{proposition}
Let $\overline{X}$ be the Borel-Serre compactification of $X$. Let $C_\x \otimes s_v$ be a modular symbol with coefficients, where $v$ is a rational vector in $E_{2k}$. Then if $\partial(C_\x \otimes s_v) =0$ in the homology of $\partial{\overline{X}}$, we can associate to $C_\x \otimes s_v$ a spectacle cycle in $\overline{X}$ which defines an absolute cycle with rational coefficients. In particular, there are natural choices to obtain spectacle cycles $C^c_{\x,[k]}$ and $C^c_{n,[k]}$ by requiring that the integral of the canonical generators of the cohomology at the boundary over the ``lenses'' vanishes. 
\end{proposition}

We did not seriously consider integral structures in this paper since our focus is different. However, Proposition \ref{Propv_x} expresses the coefficient vectors of the spectacle explicitly as linear combinations of (integral) weight vectors in $E_{2k}$, where the coefficients are integral multiples of Bernoulli numbers and special values of Bernoulli polynomials. 

\medskip

It is well-known that if $f$ is a holomorphic cusp form of weight $2k+2$ then the critical values of the associated $L$-function has a cohomological interpretation as the weighted periods of the closed holomorphic $1$-form $\eta_f$ along the $y$-axis which may be considered as a locally finite $1$-cycle or as a relative cycle (relative to the Borel-Serre boundary). In fact, the $y$-axis can be realized as a modular symbol $C_\x$ for a certain vector $\x \in V$. Of course it is immediate that we get the same result by integrating over the corresponding spectacle cycle assuming $k>0$. However, if $f$ is not cuspidal then we can no longer take the period of $\eta_f$ over the $y$-axis - the integral diverges. For the definition of the $L$-function this is usually dealt with by subtracting the constant Fourier coefficient from $f$. 
However this looses the homological interpretation of this value as a period of $\eta_f$. We show that for $k>0$ the second interpretation of the $L$-value as the period of $\eta_f$ over the spectacle cycle stills makes sense (we push the spectacles in, take the period and show that this period is independent of how much we pushed in). Thus we obtain a uniform description of the critical $L$-values for all holomorphic modular forms of weight $2k+2$ with $k>0$. We discuss these issues in Section~\ref{L-Section}, and obtain (Theorem~\ref{Lvalue})

\begin{theorem}\label{Intro-Lvalue}
Let $f = \sum_{n=0}^{\infty} a_n e^{2\pi i n z/N} \in M_{2k+2}(\G(N))$, not necessarily a cusp form. Assume that $C_\x$ is the imaginary axis in $\h$. Then the (co)homological pairing $\langle [\eta_f],{C^c_{\x,[k]}} \rangle$ is the central value of the $L$-function of $f$. We have
\[
\langle [\eta_f], {C^c_{\x,[k]}} \rangle = (-2)^k i^{k+1} \Lambda(f,k+1),
\]
where $\Lambda(f,s)$ is the completed Hecke L-function associated to $f$ which (for $Re(s) \gg 0$) is given by $\int_0^{\infty} (f(iy)-a_0) y^s \tfrac{dy}{y}$. For the pairing, we interpreted ${C^c_{\x,[k]}}$ as an absolute cycle in $X$. In fact, all critical values of $f$ arise from the pairing of $[\eta_f]$ with spectacle cycles associated to $C_\x$ (with different coefficients).
\end{theorem}

Note that Kohnen-Zagier \cite{KZ}, \S4 give a formal definition of the periods of Eisenstein series by extending the formulas for cusp forms. From this perspective we give a geometric interpretation for their procedure. 

There is also a different approach to consider the periods of non-cuspidal modular forms over infinite modular symbols. Namely, one can replace $\eta_f$ by a cohomologous $1$-form which extends to $\overline{X}$. For periods of Eisenstein series this is for example the approach of Harder and Kaiser\cite{Harder,Kaiser}, and also of Stevens \cite{Stevens} (in a slightly different context). In fact, the explicit description of a spectacle cycle given in Proposition \ref{Propv_x} together with Theorem~\ref{Intro-Lvalue} can be used to a slightly different approach to the arithmetic properties of Eisenstein cohomology given in \cite{Harder,Kaiser}.

In Section~\ref{Schwartz-Section} we introduce the Schwartz forms needed to construct the theta series. We show in Section~\ref{Main-Section} that the theta series $\theta(\tau,\varphi^V_{1,[k]})$ for $V$ underlying our realization of the Shintani lift extends to a form on $\overline{X}$ (this is in much greater generality the main result of \cite{FMres}). The crucial point for us is that the restriction of $\theta(\tau,\varphi^V_{1,[k]})$ to the boundary face of a given cusp $\ell$ is an exact differential form, and there is a (natural) primitive $\theta(\tau,\phi^{N_{\ell}}_{[k]})$, a theta series for a positive definite $1$-dimensional subspace of $V$. Hence we have found an element $\left[\theta(\tau,\varphi^V_{1,[k]}), \sum_{[\ell]}  \theta(\tau,\phi^{N_{\ell}}_{[k]}) \right]$ in the mapping cone associated to the pair $(\overline{X}, \partial \overline{X})$. In the appendix we discuss the general mapping cone construction associated to the pair $(\overline{X}, \partial \overline{X})$ and also its relationship to the cohomology of compact supports for $X$. One reason for our detailed discussion is the need to have explicit integral formulas for the Kronecker pairings in the different realizations of the cohomology. For future reference we actually carry this out in greater generality needed for this paper, namely for smooth manifolds with boundary. 

The main result is the extension of \eqref{Intro-Shintani} and is discussed in Sections~\ref{Main-Section} and \ref{Modular-Section}.

\begin{theorem}\label{MainIntro}
The mapping cone element $\left[\theta(\tau,\varphi^V_{1,[k]}), \sum_{[\ell]}  \theta(\tau,\phi^{N_{\ell}}_{[k]}) \right]$ representing a class $H^1_c({X},\widetilde{E_{2k}})$ defines a non-cuspidal holomorphic modular form of weight $k+3/2$ and is equal to the (Poincar\'e dual of the) generating series of the spectacle cycles with coefficients
\[
\sum_{n \geq0} [C^c_{n,[k]}] e^{2\pi i n \tau}.
\]
In particular, this generating series can be paired with Eisenstein series/cohomology classes. This gives a geometric way of extending the Shintani lift. 
\end{theorem}

By \eqref{Intro-Shintani} the theorem already holds for the pairing of elements in $H_1(\overline{X},\widetilde{E_{2k}})$ with $\left[\theta_{\calL_V}(\tau,\varphi^V_{1,[k]}), \sum_{[\ell]}  \theta_{\hat{\calL}_{W_\ell}}(\tau,\phi^{N_{\ell}}_{[k]}) \right]$. Hence it suffices to consider the lift for representatives of the cokernel of the natural map $H_1(\overline{X},\partial{\overline{X}},\widetilde{E_{2k}}) \to H_1(\overline{X},\widetilde{E_{2k}})$. For these representatives, we take infinite modular symbols $C_\x$ with {\it different} coefficients  for which the boundary map does not vanish. Then we prove the theorem by explicitly comparing the Fourier coefficients of the lift of these infinite modular symbols with the generating series of the intersection numbers with the spectacle cycles $C_{n,[k]}$. In a sense, we are proving the  theorem by a Hirzebruch-Zagier \cite{HZ} method. 

It turns out that this approach reduces the theorem to a corresponding theta lift for a split space of signature $(1,1)$, which we discuss in Section~\ref{1-1-Section}. For $k=0$, the trivial coefficient case, this lift for signature $(1,1)$ is a special case of a theta lift for the dual pair $\Orth(n,1) \times \Symp_n$ studied by Kudla \cite{KAnnalen,KShintani}. For this $(1,1)$-lift we also establish a regularized Siegel-Weil formula in the spirit of \cite{KR} making explicit the very general results of Kudla-Rallis. In particular, we realize classical integral weight Eisenstein series of  $k+1$ as a theta lift for $\Orth(1,1)$. For example, for the standard $\SL_2(\Z)$-Eisenstein series $E_{k+1}(\tau)$ of even weight we have 
\[
- \frac{B_{k+1}}{k+1} E_{k+1}(\tau)=  -\frac{B_{k+1}}{k+1} + 2 \sum_{\substack{x,y \in \Z_+}} x^k e^{2\pi i xy \tau},
\]
which we now interpret as a theta series of signature $(1,1)$ (where the summation is restricted to the positive cone). Moreover, from our perspective the Fourier coefficients are a weighted sum over certain $0$-cycles in the hyperbolic line. 

It is well-known that one can extend the Shimura-Shintani correspondence to Eisenstein series in a somewhat formal way by considering Hecke-eigenvalues. Our point is that one can give a geometric extension of the correspondence. In section~\ref{Eisenstein-Section} we show (for $\SL_2(\Z)$) that the lift of Eisenstein series of weight $2k+2$ indeed gives Eisenstein series of weight $k+3/2$. 

While the theta lift is defined for all congruence subgroups $\G$, the definition of the cycles themselves and the topology interpretation a priori require $\G$ to be torsion-free. Since, in many ways the most interesting case is the modular curve itself we conclude in Section~\ref{lastsection} by explaining how the results of our paper \cite{FMlocal} allow the results of to be extended to the case of quotients of orthogonal symmetric spaces by arithmetic subgroups {\it that are not torsion free}.

\medskip

It is a great pleasure to dedicate this article to Steve Kudla. His influence in our work is evident. We would like to thank him for by now decades of encouragement, collaboration, mathematical discussions, and friendship.

\section{Preliminaries}\label{prelim}

\subsection{The modular curve associated to the orthogonal group}

Let $V$ be a rational vector space of dimension $3$ with a
non-degenerate symmetric bilinear form  $(\,,\,)$ of signature
$(2,1)$. We write $q(\x) = \tfrac12(\x,\x)$ for the associated
quadratic form. Throughout we assume that $V$ is isotropic, and in fact we realize $V$ as the rational traceless $2\times2$ matrices. For simplicity we assume that the discriminant of $V$ is $1$. Then $q(\x) = -\det(\x)$ and $(\x,\y)= \tr(\x\y)$. 

In this model, $\SL_2$ acts on $V$ by conjugation, $g(\x)= g\x g^{-1}$, as isometries and gives rise to the isomorphism $\underline{G} := \Spin(V) \simeq \SL_2$ viewed as an algebraic group over $\Q$. We write $\bar{G}\simeq \PSL_2 \simeq \SO(V)$, and we set $G = \underline{G}(\R)$ for the real points of $\underline{G}$. We pick an orthogonal basis $e_1,e_2,e_3$ of $V_{\R}$ such that $(e_1,e_1)=(e_2,e_2)=1$ and $(e_3,e_3)=-1$. This also gives rise to an orientation of $V$. Explicitly, we set $e_1=\tfrac{1}{\sqrt{2}}\kzxz{}{1}{1}{}$, $e_2 = \tfrac{1}{\sqrt{2}}\kzxz{1}{}{}{-1}$, and $e_3 = \tfrac{1}{\sqrt{2}}\kzxz{}{1}{-1}{}$. We also define $u = \kzxz{}{1}{}{}$ and $u'=\kzxz{}{}{-1}{}$ so that $(u,u')=-1$. Note that $u$ and $u'$ are defined over $\Q$. 

We let $K\simeq \SO(2)$ be the stabilizer of $e_3$ in $G$, and we let $D= G/K $ be the associated symmetric space. It can be identified with the hyperboloid
\[
D \simeq \{ \x \in V(\R): \, (\x,\x) =-1, \, (\x,e_3)<0 \}.
\]
Hence $e_3$ represents the base point $z_0$ of $D$. The tangent space $T_{z_0}(D)$ at the base point is canonically isomorphic to $e_3^{\perp}$. We orient $D$ by stipulating that $e_1,e_2$ is an oriented basis of $T_{z_0}(D)$ and propagate this orientation continuously around $D$.

Of course we have $\H \simeq D$, and the isomorphism is given explicitly by
\[
z = x+iy \mapsto \x(z):=\frac1{\sqrt{2}y} \zxz{-x}{z\bar{z}}{-1}{x}.
\]
This intertwines the natural action of $G$ on $V$ and on $\h$: $\x(gz) = g(\x(z))$ and also preserves the canonical orientation of $\h$ given by its complex structure. 
 
Let $L \subset \V$ be an even lattice of full rank and write $L^\#$ for the dual lattice of $L$. We fix an element $h \in L^\#$ and let $\G$ be a torsion-free congruence subgroup of $\SL_2(\Z)$ which takes $L+h$ to itself. In the last section we remove this restriction. We let $X= X_{\G}=\G \back D$ be the associated arithmetic quotient. It is a modular curve. 

The set $\Iso(V)$ of all isotropic lines in $V(\Q)$ can be identified with $P^1(\Q)=\Q\cup\infty$, the set of cusps of $G(\Q)$, by means of the map $[\alpha:\beta] \mapsto \Span\kzxz{-\alpha\beta}{\alpha^2}{-\beta^2}{\alpha\beta}\in \Iso(V)$. This maps commutes with the $\underline{G}(\Q)$-actions. So the cusps of $X$ can be identified with the $\Gamma$- equivalence classes of $\Iso(V)$. The cusp $\infty\in P^1(\Q)$ corresponds to the isotropic line $\ell_{\infty}$ spanned by $u=u_{\infty}= \left( \begin{smallmatrix}  0&1 \\0&0  \end{smallmatrix}\right)$. For $\ell \in \Iso(V)$, we pick $\sigma_{\ell} \in \SL_2(\Z)$ such that $\sigma_{\ell} \ell_{\infty} = \ell$. We orient all lines $\ell \in \Iso(V)$ by requiring that $\sigma_{\ell} u_{\infty}=: u_{\ell}$ is a positively oriented basis vector of $\ell$. Hence a positively oriented basis vector of $\ell_{0}$, the cusp $0$, is given by $u'=u_{\ell_0}=\kzxz{0}{0}{-1}{0}$. We let $\G_{\ell}$ be the stabilizer of the line $\ell$ and write $M_{\ell}$ for the width of the associated cusp.

We let $\overline{X}$ be the Borel-Serre compactification of $X$. It is obtained by adding to each cusp $\ell$ of $X$ the circle $X_{\ell}=N_{\ell}/\G_{\ell} \simeq \R/M_{\ell}\Z$, where $N_{\ell} = \underline{N_{\ell}}(\R)$ are the real points of the nilpotent subgroup of $\underline{G}$ corresponding to $\ell$. For the topology of $\overline{X}$ it suffices to note that a sequence $z_n=x_n+iy_n$ in (a nice fundamental domain of) $X$ converges to the point $x$ in $X_{\infty}$ if $\lim x_n =x$ and $\lim y_n = \infty$. We can also view $\overline{X}$ as the $\G$-quotient of $\overline{D}$, the Borel-Serre enlargement of $D\simeq \h$, which is obtained by replacing each (rational) boundary point in $ P^1(\Q)$ by the corresponding nilpotent $N_{\ell} \simeq \R$. We orient $X_{\infty}$ (and then any $X_{\ell}$) by giving it the orientation of $N:=N_{\infty} = \left\{n(x):= \kzxz{1}{x}{}{1}; \, x \in \R \right\} \simeq \R$ (which is the same by stipulating that a tangent vector at a boundary point followed by its outer normal is properly oriented). This gives rise to a basepoint $z_{\infty}$ of $X_{\infty}$, and  any point in $X_{\infty}$ (or 
$D_{\infty}$) can be written as $n(c) z_{\infty}$. By slight abuse we identify this point with the scalar $c \in \R$. Finally, we obtain for each boundary component $X_{\ell}$ a basepoint $z_{\ell}= \sigma_{\ell} z_{\infty}$.

Note that $\overline{X}$ is homotopically equivalent to $X$.

\section{Spectacle cycles}\label{Spec-Section}

\subsection{Special cycles/modular symbols}

A vector $\x \in V(\Q)$ of positive length defines
a geodesic $D_\x$ in $D$ via
\[
D_\x = \{ z \in D; \; z \perp \x \}.
\]
We let $\G_\x$ be the stabilizer of $\x$ in $\G$. We denote the image of the quotient $\G_\x \back D_\x$ in $X$ by $C_\x$. The stabilizer $\bar{\G}_\x$ is either trivial (if the orthogonal
complement $\x^{\perp} \subset V$ is isotropic over $\Q$) or infinite cyclic (if $\x^{\perp}$ is non-split over $\Q$). It is well-known (see eg \cite{FCompo}, Lemma~3.6) that the first case occurs if and only if $q(\x) \in \left(\Q^{\times}\right)^2$. If $\G_\x$ is infinite, then $C_\x$ is a closed geodesic in $X$, while $C_\x$ is infinite if $\bar{\G}_\x$ is trivial. In the latter case these are exactly the classical modular symbols. 

In the upper half plane model, the cycle $D_\x$ is given for $\x = \kzxz{b}{2c}{-2a}{-b}$ by
\[
D_\x = \{ z \in \H; \, a|z|^2+bRe(z)+c =0\}.
\]
We orient $D_\x$ by requiring that a tangent vector $v \in T_{z}(D_x) \simeq z^{\perp} \cap \x^{\perp}\subset V$ followed by $\x$ gives a properly oriented basis of $T_z(D)$. Now $\x^{\perp}$, the orthogonal complement of $\x$ in $V_\R$, is spanned by two (not necessarily rational) cusps corresponding to isotropic lines $\ell_\x$ and $\ell'_\x$ with positive oriented generators $u_{\ell_\x}$ and $u_{\ell'_\x}$. Then the geodesic $D_\x$ joins these two cusps.  These isotropic lines are uniquely determined by the condition $\ell_\x$ and $ \ell'_\x$ both lie in $\x^{\perp}$. Thus
\[ 
(\x,u_{\ell_\x})=(\x,u_{\ell'_\x})=0.
\]
More precisely, $D_\x$ joins two points in boundary components of the Borel-Serre enlargement of $D$, and we denote the boundary points of $C_\x$ in $\overline{X}$ by $c_\x \in X_{\ell_\x}$ and $c'_\x \in X_{\ell'_\x}$. We distinguish $\ell_\x$ and $\ell'_\x$ by requiring that $u_{\ell_\x},\x,u_{\ell'_\x}$ gives a properly oriented basis of $V$ which also gives a different way of characterizing the orientation of $D_\x$. Of course, $C_\x$ is an infinite geodesic if and only if $\ell_\x$ and $\ell'_\x$ are rational. 

For $n \in \Q_{>0}$, the discrete group $\G$ acts on ${L}_{h,n} =
\{\x \in L +h ;\; q(x) =n\}$
 with finitely many orbits. We define the (decomposable) \emph{special cycle}  of discriminant $n $ on $X$ by
\begin{equation*}
C_n= \sum_{\x \in \G \back L_{h,n} } C_\x.
\end{equation*}
(We suppress the dependence on $L$, $h$ and $\G$ in the notation). 
\begin{remark}
In the above we are assuming $n>0$, we will later define $C_0$ (actually $C_{0,[k]}). $.
\end{remark}

Here the sum occurs in $H_1(X,\partial X,\Z)$ (or in $H_1(X,\Z)$ if $n \notin  \left(\Q^{\times}\right)^2$ when the cycles are absolute cycles). Note that by slight abuse of notation we use the same symbol $C_\x$ for the geodesic, the cycle, and the homology class. If we want to emphasize the homology class we write $[C_\x]$. 

\subsection{The local intersection multiplicity of two special cycles}\label{intermult}
Let $\x$ and $\y$ be two independent vectors such that $(\x,\x) >0$ and $(\y,\y)>0$ and such that
the restriction of $(\ ,\ )$ to the planes they span is positive definite. In this
case the corresponding geodesics $D_{\x}$ and $D_{\y}$ intersect in a single point $z \in D$.
It is the goal of this subsection to compute the intersection multiplicity $mult_z( D_{\x}, D_{\y})$
in terms of vector algebra in Minkowski three space $V, (\ ,\ )$. Since $\x$ and $\y$ determine $z$
we will usually drop the $z$ in what follows.
The point in doing this is that all transverse intersection of special cycles will be locally of the
above form. In this subsection we will use the standard terminology from special relativity, namely
a vector such that $(v,v)<0$ will be called ``timelike'' and a vector such that $(v,v)>0$ will be called
``spacelike''. 
\subsubsection{The Minkowski cross-product and scalar triple product}
First we have to choose a volume element for the dual $V^{\ast}$, that is a nonzero element of $\bigwedge^3(V^{\ast})$.
Note that $(\ ,\ )$ induces a negative definite form on the line $\bigwedge^3(V^{\ast})$. If we require the
volume element to have inner product $ -1 $ this does not determine the orientation (and {\it of course
the knowledge of an inner product cannot determine an orientation of a vector space}).  Let $e_1,e_2,e_3$
be the basis we have chosen earlier and let $\alpha_1, \alpha_2, \alpha_3$ be the dual basis.  We choose the
$e_1,e_2,e_3$ as a positively oriented basis and hence the orientation form $\vol \in \bigwedge^3(V^{\ast})$ 
is given by 
$$\vol = \alpha_1 \wedge \alpha_2 \wedge \alpha_3.$$   
The metric $(\ ,\ )$ induces an isomorphism $g: V \to V^{\ast}$ given by $ g(u)(v) = (u,v)$
whence $(g^{-1}(\alpha),v) = \alpha(v)$. We can now define the Minkowski cross-product $\times$ by
$$u \times v = g^{-1}( \iota_{u \wedge v} \vol).$$
Here $\iota_{u \wedge v} \vol$ is interior multiplication whence
$$\iota_{u \wedge v} \vol)(w) = \vol(u,v,w).$$

Following the usual notation from elementary vector algebra we define the Minkowski scalar product 
$s(u,v,w)$ of three vectors $u,v,w$ by
$$s(u,v,w) = ( u \times v, w).$$

We now compute the Minkowski cross-product in coordinates. 
\begin{lemma}(Minkowski cross-product)
 Let $u = x_1 e_1 + x_2 e_2 + x_3 e_3$
and $v = y_1 e_1 + y_2 e_2 + y_3 e_3$. 
Then we have
\begin{equation*}
u \times v = (x_2 y_3 - x_3 y_2) e_1 - (x_1 y_3 - x_3 y_1) e_2 - (x_1 y_2 - x_2 y_1)e_3
\end{equation*}
\end{lemma}

We leave the direct  computation to the reader.  However we can give another proof (under the assumption the above
formula gives the usual formula in the Euclidean case) by  noting there is only one place in the above formula where the Minkowski and
Euclidean cases differ.  
Both cases use the usual volume form for $(\R^3)^{\ast}$ and for both metrics $g$ we have  $g^{-1} \alpha _i =  e_i, i =1,2$
( the same as the Euclidean computation). But for the Euclidean metric we have  $g^{-1} (\alpha_3) = e_3$ whereas for
the Minkowski metric we have $g^{-1}( \alpha_3) = -e_3$.   Thus the formulas for the two cross-products will differ only in the sign of the last term, and we see $e_1 \times e_2 = - e_3$. 
Thus we get the Minkowski cross-product formula above.

\subsubsection{The formula for $\mult(D_\x, D_\y)$ }
We now recall that the formula for the intersection multiplicity.  Let $t_{\x}$ be a positively oriented  tangent
vector to $D_{\x}$ and  $t_{\y}$ be a positively oriented  tangent
vector to $D_{\y}$. Let $T_z(D)$ be the tangent space to the hyperbolic plane $D$ at $z$. Then the orientation multiplicty $\mult(D_\x, D_\y)$ is given by

\[ \mult(D_\x, D_\y)= \begin{cases} +1 & \text{ if $t_{\x}$ and $t_{\y}$ are a positively oriented basis for $T_z(D)$} \\
-1 & \text{otherwise}
\end{cases}
\]

Let $J =J_z$ be the almost complex structure of $D$ acting on $T_x(D)$ (rotation {\it counterclockwise} by $90$ degrees).
Recall we have oriented $D_{\x}$ so that $t_{\x}$ followed by $\x$ is the orientation of $T_x(D)$. It follows
then that 
$$ J t_{\x} = c \x \ , c >0 \qquad \text{and} \qquad  \x = - \frac{1}{c} J t_{_\x}.$$
To prove the following lemma it suffices to check the case $u = e_1,v = e_2,w = e_3$ which is trival.

\begin{lemma}
$$\vol(-Ju, -Jv,w) = \vol(u,v,w).$$
\end{lemma}

We can now give a vector-algebra formula for the required intersection multiplicity.

\begin{lemma} 
$$\mult_z(D_{\x}, D_\y) = \sgn~(\x \times \y,z) = \sgn~s(\x, \y,z) .$$
\end{lemma}
\begin{proof}
We have 
$$\mult_z(D_{\x}, D_\y) = \sgn~\vol( - J \x, - J \y, z) = \sgn~\vol( \x,\y,z) =  \sgn~( \x \times \y ,z).$$
\end{proof}
Noting that the inner product between any two timelike vectors  depends only whether or not their third
components agree (negative inner product) or disagree (positive inner product). Hence, since $z$ is timelike
with third-component positive (like $e_3$)
 we may replace $z$ by $e_3$ in the last expression and obtain

\begin{proposition} \label{intmult}
\[ \mult(D_\x, D_\y)= \begin{cases} +1 & \text{ if  $ \x \times \y$ ``points down'', that is has third component negative}\\ 
-1 & \text{ if  $ \x \times \y$ ``points up'', that is has third component positive.} 
\end{cases}
\]
\end{proposition}

\begin{definition}\label{epsilon}
Let $\x$ and $\y$ be a pair of spacelike vectors spanning a positive two-plane. 
Then we define $\epsilon(\x,\y)$ by
$$ \epsilon(\x,\y) = \begin{cases} +1 & \text{ if  $ \x \times \y$ ``points down''}\\
-1 & \text{ if  $ \x \times \y$ ``points up''.}   
\end{cases} $$
\end{definition}

Thus from  Proposition \ref{intmult}, we see that  $\epsilon(\x,\y)$ is the multiplicity of the intersection of $D_\x$ and
$D_\y$ at their unique point of intersection.

\subsection{A bit of the representation theory of $G$}\label{reptheory}

We let $\Sym^{k}(V)$ be the $k$-th symmetric power of $V$. We write $v^k$ for the vector $v^{\otimes k}$ in $\Sym^k(V)$. Of course, $\Sym^{k}(V)$ is not an irreducible irreducible representation of $\underline{G}(\Q)$, but contains the (rational) harmonic tensors $\calH_{k}(V)$ (with respect to the indefinite Laplacian defined by $(\,,\,)$) as irreducible representation of highest weight $2k$. We have the orthogonal decomposition
\begin{equation}\label{fundinv}
\Sym^{k}(V_{\C}) = \calH_{k}(V_\C) \oplus r^2 \Sym^{k-2}(V_{\C}),
\end{equation}
where $r^2 = e_1^2+e_2^2-e_3^2$. We let $\pi_{k}$ be the orthogonal projection map from $\Sym^{k}(V)$ to $\calH_{k}(V)$. The bilinear form $(\,,\,)$ extends to $\Sym^k(V)$ such that $(\x^k,\y^k) =(\x,\y)^k$.
Viewed differently, we have $V \simeq \Sym^2(\Q^2)$ as $\underline{G}(\Q)$-modules (For the standard symplectic basis $\frake_1,\frake_2$ of $\Q^2$, we have $\frake_1^2 \leftrightarrow u, \frake_2^2 \leftrightarrow u', \frake_1\frake_2 \leftrightarrow -\tfrac{1}{\sqrt{2}} e_2$, the standard symplectic form becomes a symmetric form on $\Sym^2(\Q^2)$, which coincides up to a multiple with the symmetric form on $V$), so that $\Sym^{k}(V) \simeq \Sym^{k}(\Sym^2(\Q^2))$. Then $\calH_{k}(V)$ as representation of $\underline{G}(\Q)$ can be identified with $\Sym^{2k}(\Q^2)$ (which has highest weight $2k$ for $\SL_2$), contained in $\Sym^{k}(\Sym^2(\Q^2))$ with multiplicity one.

Note that the stabilizer of a vector $\x \in V$ of positive length defines a torus $A_\x$ in $G$ (which is rational if and only if $\x^{\perp}$ splits over $\Q$). We obtain an element $\x^k \in \Sym^k(V)$ and then by the projection $\pi_k(\x^k)$ a vector of weight $0$ for $A_\x$ in $\Hk(V)$. We therefore can choose in the weight decomposition of $\calH^{k}(V_{\C})$ with respect to $A_\x$-weight vectors $v_{-2k}, v_{-2k+2},\dots, v_{-2},v_0,v_2,\dots, v_{2k}$ in a way such that 
 $v_0 = \pi_k(\x^k)$ and $v_{-2k+2i} = \tfrac{1}{i!} R^{i} v_{-2k}$, that is, $v_{2i+2} = \tfrac1{i+k+1}Rv_{2i}$, for the raising operator $R$. Note that for $ \x = \kzxz{1}{0}{0}{-1}$, we have $A = \{ \kzxz{t}{}{}{t^{-1}}\}$ and $R =\kzxz{0}{1}{0}{0}$.
 
\begin{lemma}\label{repformulas}
Let $ \x = \kzxz{1}{0}{0}{-1} \in V$. Then 
\[
\pi_k (\x^k) = \frac{(-2)^{k}k!}{(2k)!} R^k u'^k.
\]
More generally,
\[
v_{2i} = \frac{(-2)^{k}(k!)^2}{(2k)!(k+i)!}  R^{k+i} u'^k.
\]
In particular, $v_{-2k}= c_k u'^k$ and $v_{2k} = c_k u^k$, where $c_k = \tfrac{(-2)^k(k!)^2}{(2k)!}$. Moreover, the weight vectors are all rational. Finally,
\[
(v_{2i},v_{-2i}) = \frac{(-1)^ic_k^2 (2k)!}{(k+i)!(k-i)!}
\]

\end{lemma}

\begin{proof}
Of course $R^{k+i} u'^k$ are weight vectors of weight $2i$. Hence we only need to determine the constant $c_k$. We first compute  $(R^k u'^k, v_0) = (R^k u'^k, (\sqrt{2} e_2)^k)$. Writing $n(1) = \sum_{n=0}^{\infty} \tfrac1{n!} R^n$, we see immediately $(R^k u'^k, (\sqrt{2} e_2)^k) =  (2)^{k/2} k!(n(1) u'^k, e_2^k) = (-2)^{k} k!$. But now $(R^k u'^k, R^k u'^k) = (-1)^k (R^{2k} u'^k, u'^k) = (-1)^k (2k)! (n(1)u'^k, u'^k)  = (2k)!$. The lemma follows.
\end{proof}

\subsection{Cycles with coefficients}

We first describe how one can equip the special cycles with coefficients. We let $E$ be a rational (finite-dimensional) representation of $\G(\Q)$ which factors through $\overline{G}$. Hence $E$ is self-dual, i.e., $E \simeq E^{\ast}$, and by slight abuse we won't distinguish between $E$ and $E^{\ast}$. We write $\widetilde{E}$ for the associated local system of $E$. This gives rise to (simplicial) homology and cohomology groups of $X$ and $\overline{X}$ with local coefficients in $\widetilde{E}$. We refer the reader to \cite{JM,FMcoeff} for more details.

\begin{lemma}
Let $C_\x = \G_\x \back D_\x$ be a special cycle.

\begin{itemize}

\item[(i)] Assume $C_\x$ is closed and let $v \in E^{\G_\x}$, a $\G_\x$-invariant vector in $E$. Then the pair $(C_\x,v)$ defines a class in $H_{1}(X,\widetilde{E})$ (and also in $H_{1}(X,\partial X,\widetilde{E})$). 

\item[(ii)]
Assume $C_\x$ is infinite (so that $\overline{\G}_\x$ is trivial) and let $v$ be any vector in $E$. Then the pair $(C_\x,v)$ defines a class in $H_{1}(X,\partial X,\widetilde{E})$. 
\end{itemize}

In both cases we denote the resulting cycle by $C_\x \otimes v$ or  $[C_\x \otimes v]$ if we want to emphasize its class in homology.

\end{lemma}

\begin{proof}
Since $D_\x$ is simply connected, we have
\[
E^{\G_\x} \simeq H^0(\G_\x,E) \simeq H^0(C_\x,\widetilde{E}) \simeq H_{1}(C_\x, \partial C_\x, \widetilde{E}).
\]
($\partial C_\x$ can be empty). So $v \in E^{\G_\x}$ gives rise to an element in $H_{1}(C_\x, \partial C_\x, \widetilde{E})$ which can be pushed over to define an element in $H_{1}(X,\partial X,\widetilde{E})$ (or $H_{1}(X,\widetilde{E})$ if $C_\x$ is closed).
\end{proof}

\begin{remark}
These cycles are special cases of `decomposable cycles', see \cite{FMcoeff}. In there we explain how a (simplicial) $p$-chain with values in $\widetilde{E}$ is a formal sum $\sum_{i=1}^m  \sigma_i \otimes s_i$, where $\sigma_i$ is an  oriented $p$-simplex and $s_i$ is a flat section over $\sigma_i$. In this setting, the $\G_\x$-fixed vector $v \in E$ 
gives rise to a parallel section $s_v$ of $\widetilde{E}$. Namely, for $z \in C_\x$, the section $s_v$ for the bundle $C_\x \times_{\G_\x} E \to C_\x$ is given by $s_v(z) = (z,v)$. Thus $s_v$ is constant, hence
parallel. So in this setting the notation $C_\x \otimes s_v$ is more precise. 

\end{remark}

We now consider $E=\calH_{k}(V)$. In that case $\pi_{k}(\x^{k}) \in \calH_{k}(V)$ is ${\G_\x}$-invariant. In fact, if $\overline{\G}_\x$ is non-trivial, then $\pi_{k}(\x^{k})$ is up to a constant the only such vector. We obtain a cycle
\begin{equation*}
C_{\x,[k]} := C_\x \otimes \pi_{k}(\x^{k}) 
\end{equation*}
with values in $\calH_{k}(V)$. We then define for $n>0$ the composite cycle $C_{n,[k]} = \sum_{\x \in \G \back L_{h,n}} C_{\x,[k]}$ as before.

We write $\langle\,,\,\rangle$ for the pairing between $H_{1}(X, \partial X, \widetilde{\calH_{k}(V)})$ and $H_c^{1}(X,\widetilde{\calH_{k}(V)})$ and also between $H_{1}(X, \widetilde{\calH_{k}(V)})$ and $H^{1}(X,\widetilde{\calH_{k}(V)})$. Then in \cite{FMcoeff} we explain that for $\eta$ a closed (compactly supported) differential $1$-form on $X$ with values in $\calH_{k}(V)$ representing a class $[\eta]$ in $H^{1}(X,\widetilde{\calH_{k}(V)})$ (or $H_c^{1}(X,\widetilde{\calH_{k}(V)})$), one has
\begin{equation}\label{coeffpairing}
\langle {[C_\x \otimes v]}, [\eta] \rangle = \int_{C_\x} \left(\eta,v\right).
\end{equation}
Here $(\eta,v)$ is the scalar-valued differential form obtained by taking the pairing $(\,,\,)$ in the fiber. 

\subsection{Spectacle cycles}

We let $\x \in V$ with $q(\x)=n>0$ such that $C_\x$ is an infinite geodesic connecting the cusps $\ell'_\x$  and $\ell_\x$. In that case $\overline{\G}_\x$ is trivial, so any rational vector $v \in \calH_{k}(V)$ gives rise to a cycle
\[
C_\x \otimes v.
\]
We obviously have

\begin{lemma}
\[
\partial \left(C_{\x} \otimes v\right) =c_\x \otimes  v  - c'_\x \otimes  v .
\]
\end{lemma}

For $c \in X_\ell$, we will now study $c \otimes  v$ as a $0$-cycle at a boundary component $X_{\ell}$ with coefficients in $\Hk(V)$.

\begin{proposition}\label{exactness}
The $0$-cycle $c\otimes  v $ is trivial in $H_0(X_{\ell}, \widetilde{\calH_{k}(V)})$ if and only if $v$ is perpendicular to the highest weight vector, i.e., if and only if the component of the lowest weight space in the weight decomposition of $v$ is zero. So $c \otimes  v $ is a boundary of a $1$-chain with values in $\widetilde{\calH_{k}(V)}$  if and only if  $v \in (u_{\ell}^k)^{\perp}$. In this case there exists $w \in \calH_{k}(V)$ such that 
\[
v = \gamma^{-1} w -w.
\]
Here $\gamma$ is the positively oriented generator of $\Gamma_{\ell}$.
\end{proposition}

\begin{proof}
We may assume $\ell = \ell_{\infty}$ whence $N_{\ell}= N_{\infty}$. Since the circle is a space of type $K(\Z,1)$ and $\Gamma_{\infty}$ is Zariski dense in $N_{\infty}$, we have
\[
H^0(X_{\ell_{\infty}},\widetilde{\Hk(V)}) \simeq H^0(\Gamma_{\infty},\Hk(V)) = H^0(N_{\infty}, \Hk(V)) \cong \Hk(V)^{N_{\infty}}.
\]
But $u^{ k}$ is a highest weight vector of $\Hk(V)$ so that
\[
\Hk(V)^{N_{\infty}} = \Q u^{ k}
\]
Hence (dually using the inner product $(\ ,\ )$ on the coefficients),
\[
H^0(X_{\ell_{\infty}},\widetilde{\Hk(V)}) \simeq H_0(\Gamma_{\infty},\Hk(V)) = H_0(N_{\infty}, \Hk(V))
\]
is also $1$-dimensional, and we conclude that the vector $v$ is zero in the space of coinvariants $H_0(N_{\infty}, \Hk(V))$ if and only if $v  \in (u^k)^{\perp}$. Since $\Gamma_{\ell}$ is infinite cyclic it follows from the standard
resolution of the $\Z$ over the integral group ring, see \cite{Brown}, Example 1, page 58, that $v$ is the boundary of a one chain if and
only if it may be written $v = \gamma^{-1} w-w$ for some $w$ as above.
\end{proof}

\begin{corollary}
As a special case of the above, $c_\x \otimes \pi_{k}(\x^{k})$ is trivial in $H_0(X_{\ell_\x}, \widetilde{\calH_{k}(V)})$.
\end{corollary}

\subsection{$1$-chains at the boundary and an explicit primitive for $c \otimes  v$}

We will need an actual simplicial $1$-chain on $X_{\ell}$ with coefficients in $\widetilde{\calH_{k}(V))}$ whose boundary is $c_\x \otimes  v$. 
We have previously defined $1$-cycles with coefficients $C_\x \otimes s_{v}$ where $s_v$ was a parallel
section {\it globally defined on $C_\x$}.  We will now extend this notation to define certain one-chains
$X_{\ell,c} \otimes s_w$, where $s_{w}$ is a possibly multi-valued parallel-section on the horocircle $X_{\ell}$ obtained by parallel translating a vector $w \in\Hk(V)$ in the fiber over the point $c \in X_{\ell}$ around $X_{\ell}$.  This produces a possible jump at $c$. We construct the chain $X_{\ell,c} \otimes s_{w}$ by triangulating $X_{\ell}$ by using three vertices $u_0, u_1,u_2$ with $u_0 = c$.  We then define single-valued parallel sections on each of the three one simplices $(u_0,u_1)$, $(u_1,u_2)$ and $(u_2,u_0)$ by parallel translation of $w$. More precisely, we start at $u_0$ and parallel translate
along $(u_0,u_1)$ to get the required coefficient parallel section on the $1$-simplex $(u_0,u_1)$, we take the resulting
value at $u_1$ and parallel translate along $(u_1,u_2)$ to get the section we attach to $(u_1,u_2)$.
Finally  we take the resulting value at $u_2$ and parallel translate along $(u_2,u_0)$ to get
the required section on $(u_2,u_0)$.
We obtain a $1$-chain with coefficients to be denoted $X_{\ell,c} \otimes s_{w}=:X_{\ell,c} \otimes w $ as the sum of the three resulting one-simplices with coefficients.
We leave the proof of the following lemma to the reader.

\begin{lemma}\label{1-boundary}
Let $\g$ be the (positively oriented) generator of ${\G_{\ell}}$. Then 
\[
\partial (X_{\ell,c} \otimes {w}) = c \otimes  (\gamma^{-1} - Id)w.
\]
\end{lemma}

In particular, for $v = u_{\ell}^k$, a highest weight vector for $\Hk(V)$, we obtain a $1$-cycle which we denote by 
$X_{\ell} \otimes u_{\ell}^k$. We let $\omega_{\ell,k}$ be the unique $N$-invariant $1$-form on $X_{\ell}$ (generating $H^1(X_{\ell}, \widetilde{\Hk(V_{\C})})$) which is a Poincar\'e dual form of $X_{\ell} \otimes u_{\ell}^k$. That is, $\int_{X_{\ell} \otimes u_{\ell}^k} \omega_{\ell,k} =1$. Note that at the cusp $\infty$ we have $\omega_{\infty,k} =\tfrac{(-1)^k}{M_{\infty}} dx \otimes n(x)u'^k$. Here $M_{\infty}$ is the width of the cusp $\infty$.

\begin{lemma}\label{v_x-lemma}

Let $v \in \Hk(V)$ be a rational vector such that $v \in (u_{\ell}^k)^{\perp}$. Then there exists a unique rational vector $w$ in $\Hk(V)$ such that
\[
\partial (X_{\ell,c} \otimes w) =  c \otimes  v
\]
and
\[
 \int_{X_{\ell,c} \otimes w} \omega_{\ell,k} = 0.
 \]
 For the vector $v = \pi_{k}(\x^{k})$ (and  $\ell=\ell_\x$, $c =c_\x$) we write $w=w_\x$ for this vector.

\end{lemma}

\begin{proof}
 Considering Lemma~\ref{1-boundary}, we first find any rational vector $w$ in $\Hk(V)$ such that
\begin{equation}\label{v_x-eq}
(\gamma^{-1} - Id)w = v.
\end{equation}
But the endomorphism $\gamma^{-1} - Id$ on $\Hk(V)$ has corank $1$ and takes values in $(u^k)^{\perp}$ (since $((\gamma^{-1} - Id)w, u_{\ell}^k) = (\g^{-1}w,u_{\ell}^k) - (w,u_{\ell}^k) = (w,\g u_{\ell}^k) - (w,u_{\ell}^k) =0$), hence its image is equal to $(u_{\ell}^k)^{\perp}$. Hence there exists a vector $w$ satisfying \eqref{v_x-eq}, unique up to a multiple of $u_{\ell}^k$. That is, we can modify the $1$-chain $X_{\ell,c} \otimes w$ by any multiple of the $1$-cycle $X_{\ell} \otimes u_{\ell}^k$ without changing \eqref{v_x-eq}. This amounts to changing $w$ by a multiple of $u_{\ell}^k$. We set $ \beta:= \int_{X_{\ell,c} \otimes w} \omega_{\ell,k}$.
Then by construction
\[
\int_{X_{\ell,c} \otimes (w-\beta u^k)}  \omega_{\ell,k} = 0. \qedhere
\]
\end{proof}

We now give an explicit formula for the vector $w$ in Lemma~\ref{v_x-lemma}.

\begin{proposition}\label{Propv_x}
 Let $v_{2i} \in (u_{\ell}^k)^{\perp}$, $i=-k+1,\dots,k$ be one of the rational weight vectors for $\Hk(V)$ given in Section~\ref{reptheory}. Consider the vector $v=n(r)v_{2i}$ for some real number $r$. Then for the boundary point $c=n(r)z_{\ell}$, the vector $w$ as in Lemma~\ref{v_x-lemma} is given by
\[
w= \sum_{j=i-1}^{k} (-M_{\ell})^{j-i} \begin{pmatrix} k+j \\ j+1-i \end{pmatrix}  \frac{B_{j+1-i}(-\tfrac{r}{M_{\ell}})}{k+i} v_{2j}.
\]
Here $B_j(x)$ is the $j$-th Bernoulli polynomial. Note that for $i=0$, the highest weight vector component of $w$ is given by $(2M_{\ell})^k\tfrac{B_{k+1}(-r/M_{\ell})}{k+1}u_{\ell}^k$.

 \end{proposition}

\begin{proof}
We can assume $\ell=\ell_{\infty}$. We can write $\gamma = \exp(M_{\infty}R)$ so that
\begin{equation}\label{raising}
\gamma^{-1} -Id =  \exp(-M_{\infty}R) -Id = -M_{\infty}R + \tfrac12 M_{\infty}^2 R^2 -  \tfrac16 M_{\infty}^3 R^3 + \dots
\end{equation}
From this it is clear that $w$ can be written in the form
\begin{equation*}
w = \sum_{j = i-1}^{k}\alpha_j  M_{\infty}^{j+i}  v_{2j} = \frac{1}{M_{\infty}} \sum_{j= 0}^{k} \frac{\alpha_{j+i-1}}{j!} (M_{\infty}R)^{j} v_{2(i-1)}
\end{equation*}
for some scalars $\alpha_j$. We consider $p_k(t)=(n(t)u'^k, w)$. Note that $p_k$ is a polynomial of degree $k+1-i$ such that $\int_r^{r+M_{\ell}} p_k(t) dt =0$. We have
\begin{align*}
p_k(t)= (n(t-M_{\infty})u'^k, \g^{-1} w) &= (n(t-M_{\infty})u'^k, w+n(r)v_{2i}) \\
& = p_k(t-M_{\infty}) + (n(t-M_{\infty}-r)u'^k, v_{2i}).
\end{align*}
Since 
\[
(n(t-M_{\infty}-r)u'^k, v_{2i}) = (\tfrac1{(k-i)!} (t-M_{\infty}-r)^{k-i} R^{k-i} u'^k,v_{2i}) = \tfrac{(-1)^ic_k (2k)!}{(k+i)!(k-i)!}(t-M_{\infty}-r)^{k-i}
\]
(see Lemma~\ref{repformulas}), we see
\[
p_k(t) -  p_k(t-M_{\infty}) = \tfrac{(-1)^ic_k (2k)!}{(k+i)!(k-i)!}(t-M_{\infty}-r)^{k-i}
\]
But this is (up to constant, shift, and scaling) the difference equation which is satisfied by the $(k+1-i)$-th Bernoulli polynomial $B_{k+1-i}(t)$. Hence, since $\int_r^{M_{\infty}+r} p_k(t) dt =0$, we obtain
\begin{align*}
p_k(t) & = M_{\infty}^{k-i}\tfrac{(-1)^ic_k (2k)!}{(k+i)!(k+1-i)!}B_{k+1-i}\left( \tfrac{t-r}{M_{\infty}} \right) 
\\&=   M_{\infty}^{k-i} \tfrac{(-1)^{k+i}2^k (k!)^2}{(k+i)!(k+1-i)!} 
 \sum_{j=0}^{k+1-i} \left( \begin{smallmatrix} k+1-i \\ j \end{smallmatrix} \right) B_j\left(-\tfrac{r}{M_{\infty}}\right) \left(\tfrac{t}{M_{\infty}}\right) ^{k+1-i-j}.
\end{align*}

On the other hand, we can easily express $p_k(t)$ explicitly in terms of the coefficients $\alpha_j$ (again by Lemma~\ref{repformulas}). We obtain
\[
\alpha_{j+i-1} = (-1)^{j-1} \begin{pmatrix} k+j+i-1 \\ j \end{pmatrix}  \frac{B_{j}\left(-\tfrac{r}{M_{\infty}}\right)}{k+i}
\]
for $j=0,\dots,k+1-i$. The proposition follows.
\end{proof}

Let $\x \in V$ be a rational vector of positive length such that $\x^{\perp}$ is $\Q$-split. We have shown that for $v \in (u_{\ell_\x}^k)^{\perp}$, we can find $w$ such that $\partial(X_{\ell_\x,c_\x} \otimes {w})=c_\x \otimes  v$ and $ \int_{X_{\ell_\x,c_\x} \otimes w} \omega_{\ell_x,k} = 0$. On the other hand, if in addition $v \in (u_{\ell_\x}'^k)^{\perp}$, then we can apply Lemma~\ref{v_x-lemma} also for the other endpoint of the geodesic $C_\x$ and obtain a $1$-chain $X_{\ell'_\x,c'_\x} \otimes {w'}$ such that $\partial(X_{\ell'_\x,c'_\x} \otimes {w'})=c'_\x \otimes  v$ and $ \int_{X_{\ell'_\x,c'_\x} \otimes w'} \omega_{\ell'_x,k} = 0$.

\begin{definition}[Spectacle cycles]
Let $\x \in V$ be a rational vector of positive length such that $\x^{\perp}$ is $\Q$-split. For $v \in  (u_{\ell_\x}^k)^{\perp}\cap(u_{\ell'_\x}^k)^{\perp}$ we define the spectacle cycle $C^c_\x \otimes s_v$ by 
\[
C^c_\x \otimes v = C_\x \otimes v - X_{\ell_\x,c_\x} \otimes {w} + X_{\ell'_\x,c'_\x} \otimes {w'}.
\]
Then $C^c_\x \otimes v$ defines by construction a {\it closed} cycle in $\overline{X}$. So
\[
[C^c_\x \otimes v] \in H_1(\overline{X}, \widetilde{\calH_{k}(V)}).
\]

\end{definition}

In particular, for $v = \pi_k(\x^k)$, we make the following 

\begin{definition}\label{specdef}
Let $\x \in V$ be a rational vector of positive length. Then the spectacle cycle $C^c_{\x,[k]}$ is given by
\[
C^c_{\x,[k]} =
 \begin{cases}
C_{\x,[k]} - X_{\ell_\x,c_\x} \otimes {w_\x} + X_{\ell'_\x,c'_\x} \otimes {w'_\x} & \text{if $C_x$ is infinite} \\
C_{\x,[k]} & \text{if $C_x$ is closed.} 
\end{cases}
\]
We define the composite cycle $C_{n,[k]}^c$ in the same way as for trivial coefficients. 
\end{definition}

\begin{remark}
{ The image of $C^c_\x \otimes v$ in the relative homology group $H_1(\overline{X},\partial \overline{X}, \widetilde{\calH_{k}(V)})$ is homologous to the original relative cycle $C^c_\x \otimes v$}. 
\end{remark}

\begin{example}
Assume $\G = \G_0(N)$. Let $\x = \sqrt{2} e_2$ so that $C_\x$ is the geodesic joining the cusps $\ell_0$ and $\ell_{\infty}$. Let $k=1$ so that $\Hk(V) = V$. Then 
\[
w_\x = u'-\tfrac1{{2}} \x + \tfrac16 u.
\]
At the other cusp with width $N$ we have 
\[
w'_\x = -\tfrac1{N}u-\tfrac1{{2}} \x  - \tfrac16Nu'.
\]
\end{example}

\section{Modular forms as vector valued differential forms and pairings with modular symbols and spectacle cycles}\label{L-Section}

We consider $f \in M_k(\G) \, (S_k(\G))$ a holomorphic modular (cusp) form for $\G \subseteq \SL_2(\Z)$ of weight $2k+2$. Then 
\[
\eta_f := f(z)dz \otimes (z\frake_1+\frake_2)^{2k} = f(z)dz \otimes (z^2u - \sqrt{2}ze_2 + u')^{k}
\]
defines a closed holomorphic $1$-form on $X$ with values in the local system associated to $\Sym^{2k}(\C^2) \simeq \Hk(V_\C)$. Note that
\[
n(z)u'^k = (z^2u - \sqrt{2}ze_2 + u')^{k}.
\]
It is well-known that this assignment induces the Eichler-Shimura isomorphisms 
\begin{align*}
M_{2k+2}(\G) \oplus \overline{S_{2k+2}(\G)} &\simeq H^{1}(X,\widetilde{\Sym^{2k}(\C^2)}), \\
 S_{2k+2}(\G) \oplus \overline{S_{2k+2}(\G)} &\simeq H_!^{1}(X,\widetilde{\Sym^{2k}(\C^2)}).
\end{align*}
Here $H_!^{1}(X,\widetilde{\Sym^{2k}(\C^2)})$ is the image of the compactly supported cohomology in the absolute cohomology. It is isomorphic to $H^{1}(\widetilde{X},\widetilde{\Sym^{2k}(\C^2)})$, the cohomology of the smooth compactification $\widetilde{X}$.

\subsection{Cohomological periods of cusp forms}

Let $\x = \kzxz{b}{2c}{-2a}{-b} \in V$ be of positive length and consider the associated cycle $C_{\x,[k]}$ with values in $\calH^k(V)$. Since 
\[
(n(z)u'^k, \pi_k(\x^k))= (n(z)u'^k, \x^k) = (n(z)u',\x)^k = (-2)^k (az^2+bz+c)^k,
\]
we immediately see that for $f \in S_{2k+2}(\G)$ the cohomological pairing with the special cycle $C_{\x,[k]}$ is given by
\begin{equation}\label{Shin-period}
\langle {C_{\x,[k]}}, [\eta_f] \rangle = (-2)^k \int_{C_\x} f(z) (az^2+bz+c)^k dz.
\end{equation}
This is of course the classical formula for the weighted period of $f$ over the cycle $C_\x$, see e.g. \cite{Shin}. A little calculation using Lemma~\ref{repformulas} also yields the well-known

\begin{proposition}\label{cuspLvalues}
Let $\x = e_2$ such that $C_\x$ is the imaginery axis. Let $f \in S_{2k+2}(\G)$ be a cusp form. Then for any weight vector $v_{2j} \in \Hk(V)$, $j=-k,\dots,k$, the cohomoligical pairing $\langle [\eta_f], [C_{\x} \otimes {v_{2j}}] \rangle$ is up to a constant equal to a critical value of the L-function of the cusp form $f$:
\[
\langle  [\eta_f],[C_{\x} \otimes v_{2j}] \rangle = c_{k,j} \Lambda(f,k+1-j).
\]
Here $\Lambda(f,s)=\int_0^{\infty} f(iy)y^s \tfrac{dy}{y}$ is the completed Hecke L-function associated to $f$ and
\[
c_{k,j} = \frac{i(-i)^{k-j} 2^{k} (k!)^2}{(k-j)!(k+j)!}.
\]
\end{proposition}

\subsection{Cohomological periods of arbitrary modular forms}

We would like to pair also an arbitrary modular form with a special cycle with coefficients and extend \eqref{Shin-period} and Lemma~\ref{cuspLvalues}. For this it is natural to consider our spectacle cycles $C^c_{\x} \otimes v$ because they define absolute homology classes in $H_1(\overline{X},\widetilde{\calH_k(V)})$. However, for a modular form $f \in M_{2k+2}(\G)$, the differential form $\eta_f$ does not extend to a form on the Borel-Serre compactification $\overline{X}$, and hence defines (only) a class in $H^1(X, \widetilde{\calH_k(V)})$. In particular, the integral of $\eta_f$ over the spectacle cycle does not converge. Therefore the cohomological pairing $\langle [\eta_f], [C^c_{\x} \otimes v] \rangle$ only makes sense using the isomorphisms of the (co)homology groups of $X$ and $\overline{X}$. Hence to obtain an integral formula for $\langle  [\eta_f], [C^c_{\x} \otimes v] \rangle$ one has two principal approaches. On one hand one can modify the form $\eta_f$ to extend to $\overline{X}$ (hence defining a class in $H^1(\overline{X}, \widetilde{\calH_k(V)})$) and then integrate this modified form over the cycle $[C^c_{\x} \otimes v]$ in $\overline{X}$. This is the approach in e.g. \cite{Harder,Kaiser} or in somewhat different context of \cite{Stevens}, where they carry this out for Eisenstein series. We proceed differently by modifying the spectacle cycles to have support on $X$ (thus defining classes in $H_1(X, \widetilde{\calH_k(V)})$). 

It suffices to consider the case when the infinite geodesic $D_\x$ is a vertical line in the upper half plane. Close to the cusp at $\infty$, we truncate $C_\x$ at some (sufficiently large) height $T_1$ and do the same at the other cusp at $T_2$ to obtain the truncated geodesic $C^{T_1,T_2}_\x$. Furthermore, we push in the cap $X_{\infty,c_\x} \otimes {w}$ to this (finite) height $T_1$ to obtain 
$X^{T_1}_{\infty,c_\x} \otimes {w}$. We do the same at the other cusp to obtain $X^{T_2}_{\ell'_\x,c'_\x} \otimes {w'}$. This gives the ``pushed in'' cycle 
\[
C^{c,T_1,T_2}_{\x} \otimes v =C^{T_1,T_2}_{\x} \otimes v - X^{T_1}_{\ell_\x,c_\x} \otimes w + X^{T_2}_{\ell'_\x,c'_\x} \otimes w'.
\]

\begin{lemma}\label{homology1}
The cycles $C^c_{\x} \otimes v$ and $C^{c,T_1,T_2}_{\x} \otimes v$ are homologous in the absolute homology of $\overline{X}$. 
\end{lemma}
\begin{proof}
We assume that the first vertex of the triangulation $u_0$ is the point $c_\x$ at infinity of the upper half plane, and we write $\gamma_T$ for the geodesic ray from $iT$ to $u_0$. The quotient $X^{T_1}_{\infty}$ of the horocircle at height $T_1$ and the Borel boundary circle $X_{\infty}$ bound an annulus
$A^{T_1}$. We orient the annulus so that $\partial A^{T_1} = X_{\infty} - X^{T_1}_{\infty}$.
We break this annulus into three cellular `rectangular' regions by vertical geodesics joining $u_i, i =0,2,1$ to corresponding points $u_i^{T_1}$ on the horocircle at height $T_1$. The first of these geodesic segments will be $\gamma_{T_1}$. We will let $u_0^{T_1} = c_\x^{T_1}$, the
intersection of $C_\x$ with the quotient of the horocircle at height $T_1$. We then
extend the three parallel sections on the simplices $(u_0,u_1),(u_1,u_2),(u_2,u_0)$ by parallel
translation within the region that has the corresponding simplex as a part of its boundary.
We again denote the resulting multivalued section on $A^{T_1}$ by $s_w$. 
The section $s_w$ has a parallel jump along $\gamma_{T_1}$ with value $s_v$.  After
refining the above cellular decomposition  of $A^{T_1}$ to a triangulation we obtain a simplicial 
two-chain with coefficients to be  denoted
$A^{T_1}_{\infty,\gamma_{T_1}}\otimes s_w$ with coefficients such that
\[
 \partial A^{T_1}_{\infty,\gamma_{T_1}} \otimes s_w = \gamma_{T_1} \otimes s_v + X_{\infty, c_\x} \otimes s_w- X^{T_1}_{\infty,c_\x^{T_1}} \otimes s_w.
\]
We repeat the construction at the other end of the infinite geodesic to
obtain the required primitive, roughly the union of the two annuli with coefficients
$A^{T_1}_{\infty, \gamma_{T_1}}\otimes s_w$ and $A^{T_2}_{0,\gamma_{T_2}} \otimes s_{w'}$ where $w'$ is the solution
of the jump equation with jump $v$  at the cusp $0$. 
\end{proof}

As a consequence we obtain a critical lemma which will enable us to painlessly evaluate a limit
in the following Theorem \ref{Lvalue}.

\begin{lemma}
Let $f \in M_{2k+2}(\G)$. Then the cohomological pairing $\langle  [\eta_f], [C^c_{\x} \otimes v] \rangle$ is given by 
\[
\langle [\eta_f], [C^c_{\x} \otimes v] \rangle =\langle [\eta_f], [C^{c,T_1,T_2}_{\x} \otimes v]  \rangle = \int_{C^{c,T_1,T_2}_{\x} \otimes v}\eta_f
\]
for any sufficiently large $T_1,T_2>0$. 
\end{lemma}

\begin{proof}
The cycles $ C^c_{\x} \otimes v$ and its pushed-in incarnation $C^{c,T_1,T_2}_{\x} \otimes v$ are homologous and since the pairing is cohomological it does not depend on $T_1,T_2$.
\end{proof}

We now see that $\langle [\eta_f], [C^c_{\x} \otimes v] \rangle$ gives a cohomological interpretation of the critical values of the L-function of $f$. The point here is to deal with the case of noncuspidal $f$, especially Eisenstein series, extending Proposition~\ref{cuspLvalues}.

\begin{theorem}\label{Lvalue}
Let $f = \sum_{n=0}^{\infty} a_n e^{2\pi i n z/N} \in M_{2k+2}(\G(N))$, not necessarily a cusp form. Let $\x = e_2$ as before and let $v_{2j}$, $j = -k+1,\dots,k-1$ be a rational weight vector.
Then 
\[
\langle  [\eta_f],[ C^c_{\x} \otimes v_{2j}] \rangle = c_{k,j} \Lambda(f,k+1-j).
\]
Here $\Lambda(f,s)$ is the completed Hecke L-function associated to $f$ which (for $Re(s) \gg 0$) is given by $\int_0^{\infty} (f(iy)-a_0) y^s \tfrac{dy}{y}$ and $c_{k,j}$ is the constant defined in Proposition~\ref{cuspLvalues}.
\end{theorem}

\begin{proof}
We let $g(z) = f|_{2k+2} \kzxz{}{-1}{1}{}(z) = z^{-(2k+2)}f(-1/z) = \sum_{n=0}^{\infty}b_n e^{2\pi i nz/N}$. 
We consider the individual components of $\int_{C^{c,T_1,T_2}_{\x} \otimes v_{2j}}\eta_f$. We easily see
\begin{align}
c_{k,j}^{-1}  \int_{C^{T_1,T_2}_{\x} \otimes v_{2j}} \eta_f &=   \int_1^{T_1} f(iy) y^{k+1-j} \tfrac{dy}{y} +
(-1)^{k+1} \int_1^{T_2} g(iy) y^{k+1+j}  \tfrac{dy}{y} \\
 &= \int_1^{T_1} (f(iy)-a_0) y^{k+1-j} \tfrac{dy}{y} + (-1)^{k+1} \int_1^{T_2} (g(iy)-b_0)  y^{k+1+j}  \tfrac{dy}{y} \label{L1}\\
&\quad - \frac{a_0}{k+1-j} - (-1)^{k+1} \frac{b_0}{k+1+j} \label{L2} \\
&\quad + \frac{a_0}{k+1-j}T_1^{k+1} + \frac{b_0}{k+1+j}T_2^{k+1}. \label{L3}
\end{align}
We also have
\begin{align}
c_{k,j}^{-1} \int_{- X^{T_1}_{\ell_\x,c_\x} \otimes {w_{2j}}} \eta_f & = -c_{k,j}^{-1}\int_0^N f(x+iT_1) (n(x+iT_1) u'^k, w_{2j}) dx \\ &= - c_{k,j}^{-1} a_0 \int_0^N (n(x+iT_1) u'^k, w_{2j}) dx + O(e^{-C_1T_1}) \label{L4}
\end{align}
for some constant $C_1$. Similarly,
\begin{equation}
c_{k,j}^{-1}\int_{X^{T_2}_{\ell'_\x,c'_\x} \otimes {w'_{2j}}} \eta_f = c_{k,j}^{-1} b_0 \int_0^N (n(x+iT_2) u'^k,  \kzxz{}{-1}{1}{}w'_{2j}) dx + O(e^{-C_2T_2}). \label{L5}
\end{equation}
Since $\int_{C^{c,T_1,T_2}_{\x} \otimes v_{2j}} \eta_f$ is finite and independent of $T_1,T_2$,  the limit as $T_1,T_2 \to \infty$ must exist. The limit for the terms \eqref{L1} and \eqref{L2} exists and gives the value $s=k+1-j$ of the standard expression in the proof of the analytic continuation of $\La(f,k+1-j)$. The leading terms in \eqref{L4} and \eqref{L5} are polynomials of degree $k+1-j$ in $T_1$ and of degree $k+1+j$ in $T_2$ respectively. By the characterization of $w_{2j}$ and $w'_{2j}$ given in Lemma~\ref{v_x-lemma}. and setting $T_1=T_2=0$ we see that these polynomials have constant terms zero. Since the limits $T_1,T_2 \to \infty$ exist we see that the leading terms of \eqref{L4} and \eqref{L5} are in fact monomials und must cancel \eqref{L3}. (One can see this also by the explicit characterization of $(n(x+iT_1) u'^k, w_{2j})$ in terms of Bernoulli polynomials given in the proof of Proposition~\ref{Propv_x}). In summary, in the limit the terms \eqref{L3}, \eqref{L4}, \eqref{L5} cancel and do not contribute. This proves the theorem.
\end{proof}

\section{Schwartz forms}\label{Schwartz-Section}

In this section all vector spaces and associated groups are defined over $\R$. 

\subsection{Schwartz forms for $D$}

The following is a very special case of the construction of special Schwartz forms with coefficients given in \cite{FMcoeff}, \S5. 

We let $\calS(V)$ be the space of Schwartz functions on $V$. We let $G'= \Mp_2(\R)$, the metaplectic cover of $\SL_2(\R)$ and let $K'$ be the pullback of $\SO(2)$ under the covering map. Note that $K'$ admits a character $\chi_{1/2}$ whose square descends to the character which induces the isomorphism $\SO(2) \simeq \Uni(1)$. Then $G' \times G$ acts on $\calS(V)$ via the Weil representation $\omega$ for the additive character $t \mapsto e^{2\pi i t}$. Note that $G$ acts naturally on $\calS(V)$ by $\omega(g)\varphi(x) = \varphi(g^{-1}x)$.

We first consider the standard Gaussian $\varphi_0 = \varphi_0^V$ on $V$,
\begin{equation*}
\varphi^V_0(\x,z) = e^{-\pi (\x,\x)_z},
\end{equation*}
where $(\x,\x)_z$ is the majorant associated to $z \in D$. At the base point $z_0$ we also write $(\x,\x)_0$ for $(\x,\x)_{z_0}$, and for $\varphi^V_0$ we also just drop the argument $z_0$. Note $\varphi^V_0(\x,z) = \varphi_0(g_z^{-1}\x)$. Here $g_z = n(x)a(\sqrt{y})=\kzxz{1}{x}{}{1} \kzxz{\sqrt{y}}{}{}{\sqrt{y}^{-1}}$ is the standard element which moves the basepoint $i$ in the upper half plane model to $z$. 

We denote the coordinate functions for a vector $\x$ with respect to the basis $e_1,e_2,e_3$ by $x_i$. 
We define the Howe operators $\calD_{j}$ on $\calS(V)$ by
\begin{equation*}
\calD_{j} =   x_{j} - \frac1{2\pi}\frac{\partial}{\partial x_{j}},
\end{equation*}
We define a Schwartz form $\varphi_{1,k}= \varphi_{1,k}^V$ taking values in $\calA^1(D,\widetilde{\Sym^{k}(V)})$, the differential $1$-forms with values in the local system for $\Sym^{k}(V)$. More precisely,
\begin{equation*}
\varphi^V_{1,k} \in
 [\mathcal{S}(V) \otimes \calA^1(D) \otimes  \Sym^{k}(V)]^{G} \simeq
 [\mathcal{S}(V) \otimes \mathfrak{p^{\ast}} \otimes \Sym^{k}(V)]^{K}.
 \end{equation*}
Here $G$ and $K$ act diagonally on all three factors, and the isomorphism is given by evaluation at the basepoint $z_0$ of $D$. Here $\mathfrak{g} =\mathfrak{k}  \oplus \mathfrak{p}$ is the Cartan decomposition of $\mathfrak{g}$, the Lie Algebra of $G$. We identify $\mathfrak{g}$ with $\wwedge{2}{} V$. Then $e_1 \wedge e_3$ and $e_2 \wedge e_3$ is a basis of $\mathfrak{p}$. We write $\omega_1, \omega_2$ for the corresponding dual basis of $\mathfrak{p}^{\ast}$.

At the basepoint, $ \varphi^V_{1,k}$ is explicitly given by 
\begin{equation*}
\varphi^V_{1,k}= \frac1{2^{k+1/2}} \sum_{\alpha=1}^2 \sum_{\beta_1,...,\beta_{k}=1}^{2}(\calD_{\alpha} \circ \calD_{\beta_1} \circ \cdots \circ \calD_{\beta_{k}})(\varphi_{0}) \otimes \omega_{\alpha} \otimes \left(e_{\beta_1} \cdots e_{\beta_{k}} \right).
\end{equation*}
The form $\varphi_{1,k}^V$ is closed (\cite{FMcoeff}, Theorem~5.7):
\begin{equation}
d \varphi_{1,k}^V(\x) =0
\end{equation} 
for all $\x \in V$. Furthermore, $\varphi^V_{1,k}$ has weight $k+\tfrac{3}2$ under the Weil representation of $K'$ (\cite{FMcoeff}, Theorem~5.6). That is,
\begin{equation}
\omega(k') \varphi^V_{1,k} = \chi^{2k+3} _{1/2}(k')\varphi^V_{1,k}.
\end{equation}
We then project onto $\calH_{k}(V)$ in the coefficients to obtain $\varphi^V_{1,[k]}$. That is, 
 \begin{equation*}
\varphi^V_{1,[k]}= \frac1{2^{k+1/2}} \sum_{\alpha=1}^2 \sum_{\beta_1,...,\beta_{k}=1}^{2}(\calD_{\alpha} \circ \calD_{\beta_1} \circ \cdots \circ \calD_{\beta_{k}})(\varphi_{0}) \otimes \omega_{\alpha} \otimes  \pi_{k} \left(e_{\beta_1} \cdots e_{\beta_{k}} \right).
\end{equation*} 
Thus
\begin{equation*}
\varphi^V_{1,[k]} \in
[\mathcal{S}(V) \otimes \mathfrak{p^{\ast}} \otimes \calH_{k}(V)]^{K}.
\end{equation*}
Of course, $\varphi^V_{1,[k]}$ is also closed and has weight $k+3/2$ as well. 
Note 
\begin{equation*}
\calD_{\alpha}^{j} (e^{-\pi x_{\alpha}^2}) = (2\pi)^{-j/2} H_j\left(\sqrt{2\pi}x_{\alpha}\right) e^{-\pi x_{\alpha}^2},
\end{equation*}
where ${H}_j(t) =(-1)^j e^{t^2} \tfrac{d^j}{dt^j}e^{-t^2}$ is the $j$-th Hermite polynomial. Hence the Schwartz function component of $\varphi_{1,k}$ consists of products of Hermite polynomials times the Gaussian $\varphi_0$. We write $
\widetilde{H}_j(x_{\alpha}) =  (2\pi)^{-j/2} H_j\left(\sqrt{2\pi}x_{\alpha}\right)$. Then explicitly we have
\begin{align*}
\varphi^V_{1,k}(\x,z) & =  \frac1{2^{k+1/2}} \sum_{j=1}^{k+1} \widetilde{H}_j((g_z^{-1}\x)_1) \widetilde{H}_{k+1-j}((g_z^{-1}\x)_2) \varphi_0(g_z^{-1}\x) \frac{dy}{y} \otimes g_z(e_1^{j-1}e_2^{k+1-j}) \\
& -
\frac1{2^{k+1/2}} \sum_{j=0}^{k} \widetilde{H}_j((g_z^{-1}\x)_1) \widetilde{H}_{k+1-j}((g_z^{-1}\x)_2) \varphi_0(g_z^{-1}\x) \frac{dx}{y} \otimes g_z(e_1^{j}e_2^{k-j}). 
\end{align*}

\subsection{Schwartz forms at the boundary}

In this subsection we discuss certain Schwartz forms at the boundary components of the Borel-Serre enlargement of $D$. We consider the boundary component $D_{\ell} = N_{\ell} \simeq \R$ associated to the cusp $\ell$. The (rational) isotropic vectors $u_{\ell}$ and $u'_{\ell} =u_{\ell'}=\sigma_\ell u'$ define a positive definite subspace $W_{\ell} = \ell^{\perp} \cap \ell'^{\perp}$. This gives rise to a Witt splitting of $V$:
\[
V = \ell \oplus W_{\ell} \oplus \ell'.
\]
For $\ell_{\infty}$ we have $W = \R e_2$, and we use this to identify $W$ with $\R$. Hence $(\w,\w)=\w^2$. We define a Schwartz form 
\[
\varphi^{W_{\ell}}_{j} \in  [\mathcal{S}(W_{\ell}) \otimes \Sym^j(W_{\ell})],
\]
on $W_{\ell}$ which for $\ell_{\infty}$ is given by 
\[
\varphi_{j}^W(\w) =  -\frac1{2^{j+1/2}}  \widetilde{H}_{j+1}(\w) e^{-\pi \w^2} \otimes e_2^{j+1},
\]
and similarly at the other cusps. It is easy to see (\cite{FMcoeff}, Theorem~5.6) that $\varphi_j^{W_{\ell}}$ has weight $j+3/2$.

The Schwartz function $\varphi^{W_{\ell}}_{k+1}$ gives rise to a differential $1$-form $\varphi^{N_{\ell}}_{1,[k]}$ on $W_\ell$ on the boundary component $D_{\ell}$ of $D$ with values in the $\calH_{k}(V)$. At $\infty$ it is given by 
\[
\varphi_{1,[k]}^N(\w,x) =  -\frac1{2^{k+1/2}}  \widetilde{H}_{k+1}(\w) e^{-\pi \w^2} \otimes {dx} \otimes n(x)\pi_k(e_2^{k}).
\]
Here $\w \in W_\R$ and $x \in D_{\infty} \simeq \R$. We therefore have
\[
\varphi^{N_{\ell}}_{1,[k]} \in  [\mathcal{S}(W_{\ell}) \otimes \calA^1(D_{\ell}) \otimes \calH_{k}(V)]^{N_{\ell}}
\]
Here $N_{\ell}$ acts diagonally, where the action on the first factor is trivial. This construction is a very special case of our general construction given in \cite{FMres}, which assigns to a Schwartz form for the smaller orthogonal space $W$ a form on a boundary component of the Borel-Serre enlargement of the symmetric space associated to the larger space $V$. In the present situation something special occurs. For $k>0$, the form $\varphi^{N_{\ell}}_{1,[k]}$ is an {\it exact} differential form on $D_{\ell}$. At $\infty$, a primitive is given by 
\[
\phi_{[k]}^N(\w,x) :=  -\frac1{2^{3k/2+1/2}k}  \widetilde{H}_{k+1}(\w) e^{-\pi \w^2} \otimes  1 \otimes n(x)v_{-2} \in  [\mathcal{S}(W) \otimes \calA^0(D_{\infty}) \otimes \calH_{k}(V)]^{N}.
\]
Here $v_{-2}$ is the weight $-2$ vector in the weight decomposition of $\calH_{k}(V)$ such that $\tfrac1{k}R v_{-2} = v_0=2^{k/2}\pi_k(e_2^{k})$. We easily see
\begin{proposition}\label{primitive}
Assume $k>0$. The form $\varphi_{1,[k]}^{N_{\ell}}(\x,x)$ is an exact differential form on $D_{\ell}$ with a primitive $\phi_{[k]}^{N_{\ell}}$. Thus  
\[
d \phi_{[k]}^{N_{\ell}} = \varphi_{1,[k]}^{N_{\ell}}.
\]
\end{proposition}

\begin{proof}
It is enough to check this for $\ell_{\infty}$ at the base point of $D_{\infty}$. Since $N$ acts trivial on $\mathcal{S}(W)$, it suffices to solve the equation $Rv= \pi_k(e_2^k)$ for the infinitesimal generator $R$ of $N$. But $\tfrac1{k}2^{-k/2}v_{-2}$ is by definition a solution.
\end{proof}

\begin{remark}
The primitive is (of course) not unique. We could add any multiple of the highest weight vector $u^k$ to $v_{-2}$. 
\end{remark} 

\begin{remark}
For $k=0$ the form $\varphi_{1,[k]}^{N_{\ell}}$ is {\it not} exact. 
\end{remark}

\subsection{Schwartz forms for the hyperbolic line}\label{hyperbolicline}

We consider a real quadratic space $U$ of signature $(1,1)$. Hence $U = \ell \oplus \ell'$ for two isotropic lines $\ell = \R u$ and $\ell' = \R u$ with $(u,u')=-1$. We obtain an orthogonal basis $\eps_1 = (u-u')/\sqrt{2}$, $\eps_2=(u+u')/\sqrt{2}$. We can realize the symmetric space $D_{U}$ associated to $U$ as 
\[
D_{U}=\{\x \in U; \; (\x,\x) = -1, (\x, \eps_2)<0\}
\]
with base point $\eps_2$. We have $\SO_0(U) \simeq \R_+$, the connected component of the identity, and we identify $D_U$ with $\R_+$ in this way. The isomorphsim $\R_+ \to D_U$ is explicitly given by 
\[
t \mapsto \x(t) := (tu+t^{-1}u')/\sqrt{2}.
\]

The Schwartz forms $\varphi^{U}_{1,k}$ constructed in \cite{FMcoeff} for this signature satisfy\begin{equation*}
\varphi^{U}_{1,k} \in
 [\mathcal{S}(U) \otimes \calA^1(D_{U}) \otimes  \Sym^{k}(U)]^{A},
\end{equation*}
and are given by 
\[
\varphi^U_{1,k}(\x,t) =
\frac1{2^{k+1/2}}  \widetilde{H}_{k+1}((a(\sqrt{t})^{-1}\x)_1)\varphi_0^U(a(\sqrt{t})^{-1}\x) \otimes 
\frac{dt}{t} \otimes
a(\sqrt{t})  \eps_1^{k}.
\]
Here $\varphi^U_0$ is the Gaussian for $U$. Note that $\varphi^{U}_{1,k}$ has weight $k+1$ under the Weil representation of $K'$ (\cite{FMcoeff}, Theorem~5.6) and defines a closed $1$-form on $D_{U}$ (\cite{FMcoeff}, Theorem~5.7). Projecting the coefficients to $\calH_k(U)$, we obtain the forms $\varphi^{U}_{1,[k]}$. Of course neither $\Sym^{k}(U)$ nor $\calH_k(U)$ are irreducible representations of $\SO_0(U)$. In fact, $\calH_k(U)$ is two-dimensional spanned by $u^k$ and $u'^k$. 

Now view $U$ as a subspace of $V$ and $D_U$ as a subsymmetric space of $D$. Specifically, if $U^{\perp}$ is spanned by a vector $\w \in V$ of positive length, then $D_U = D_{\w}$ in the notation of Section~3.1. We easily derive from the definitions

\begin{lemma}\label{varphi-splitting}
The form $\varphi^V_{1,k}$ on D is functorial with respect to the restriction to the subsymmetric  space $D_{U}$. We have 
\[
r_{D_{U}}\varphi^V_{1,k}(\x+\w) = \sum_{j=0}^k  \varphi^{U}_{1,j}(\x)  \cdot \varphi^{U^{\perp}}_{k-j}(\w).
\]
Here $\x \in U$ and $\w \in U^{\perp}$, and the product on the left hand side arises from the natural product structure on the Schwartz spaces and on the symmetric algebra $\Sym^{\bullet}(V)$. 
\end{lemma}

\section{Theta series and integrals for the hyperbolic line}\label{1-1-Section}

In \cite{FMcoeff} we developed a general theory of generating series of cycles with {\it non-trivial coefficients} inside locally symmetric spaces associated to orthogonal groups of arbitrary signature. In this section, we develop rather completely the easiest case of an extension of these results to include boundary contributions. Namely, we consider the $\Q$-split case of signature $(1,1)$. This case is not included in \cite{FMcoeff}. The results of this section should rather directly generalize to signature $(1,q)$ with nontrivial coefficients. 

We let $U$ be a rational split space of signature $(1,1)$ spanned by two isotropic vectors $u,u'$ with $(u,u')=-1$. We set $\eps_1= (u-u')\sqrt{2}$ and $\eps_2 = (u+u')\sqrt{2}$ as before.
For the associated symmetric space $D_U$ we will also write $X_U$ (if we think of it as a locally symmetric space of infinite volume). For an even lattice $L_U$ and $h_U \in L_U^{\#}$, we define the theta series 
\[
\theta_{\calL_U}(\tau,t,\varphi^U_{1,[k]}) = v^{-k/2} \sum_{\x \in \calL_U} \varphi^U_{1,[k]}(\sqrt{v}\x,t) e^{\pi i(\x,\x)u},
\]
where $\calL_U = L_U + h_U$. This defines a closed $1$-form on $X_U$, which in the $\tau$-variable transforms like a non-holomorphic modular form of weight $k+1$ for $\G(M')$, where $M'$ is the level of $L_U$. Assume (for simplicity)
\[
\calL_U = (\Z M_1 + h_1) u \oplus (\Z M_2 + h_2) u',
\]
with $0\leq h_i < M_i$ so that $M'= M_1M_2$.

\begin{proposition}\label{partialPoissonU}
With the notation above, we have
\begin{multline*}
\theta_{\calL_U}(\tau,\varphi^U_{1,[k]}) \\
=  (-i)^{k+1} 2^{-k/2}v^{-k-1}\left( \frac{t}{M_1} \right)^{k+2}  \sum_{\substack{ m \equiv h_2M_1 \; (M')\\ n \in \Z} } e^{2\pi i nh_1 /M_1} \left(m\bar{\tau} +n
\right)^{k+1} e^{ -\pi \frac{t^2}{vM_1^2} |m\bar{\tau} +n|^2}\\  \otimes \frac{dt}{t} \otimes (a(\sqrt{t})\eps_1)^k.
\end{multline*}
In particular, $\theta_{\calL_U}(\tau,\varphi^U_{1,[k]})$ is a rapidly decreasing $1$-form on $X_U$.

\end{proposition}

\begin{proof}

The formula follows by applying Poisson summation over the sum $(\Z M_1 + h_1) u$. It boils down to the fact that the Fourier transform of $H_j(\sqrt{\pi} x) e^{-\pi x^2}$
is given by $(-2i \sqrt{\pi}  x)^j e^{-\pi x^2}$ (which
can be easily seen by \cite{Lebedev}, (4.11.4)).

The rapid decay as $t\to \infty$ is now obvious. The decay at the other end $t=0$ follows from switching the cusps. That is, using Poisson summation for the sum over multiples of $u'$.
\end{proof}

From the proposition, we see that $\theta_{\calL_U}(\tau,\varphi^U_{1,[k]})$ defines a class in the compactly supported cohomology of $X_U$:
\[
[\theta_{\calL_U}(\tau,\varphi^U_{1,[k]}) ] \in H^1_c(X_U, \widetilde{\calH_k(U_\C)}).
\]
We can therefore pair the theta series with classes $X_U \otimes w$ in the relative homology $H_1(X_U, \partial X_U,\widetilde{\calH_k(U)})$. Recall that $\calH_k(U)$ is spanned by $u^k$ and $u'^k$. In the following we will only consider $X_U \otimes u'^k$. The pairing with $X_U \otimes u^k$ is analogous.

\begin{theorem}\label{SW}
We associate to $\calL_U$ the level $M'$ Eisenstein series of weight $k+1$ by 
\[
G_{k+1}(\tau,\calL_U,u'):= \lim_{s \to 0} 
 \sideset{}{'}\sum_{\substack{ m \equiv h_2M_1 \; (M')\\ n \in \Z} } e^{2\pi i nh_1/M_1} \left(m{\tau} +n \right)^{-k-1} |m{\tau} +n|^{-2s}.
\]
Here $\sideset{}{'}\sum$ means that we only sum over pairs $(m,n) \ne (0,0)$. (Of course, we only need Hecke summation for $k=0,1$). Then 
\[
\langle \theta_{\calL_U}(\tau,\varphi^U_{1,[k]}),X_U \otimes u'^k  \rangle =
\int_{X_U \otimes u'^k} \theta_{\calL_U}(\tau,z,\varphi^U_{1,[k]}) = - M_1^{k} k! \left( \frac{i}{2\pi} \right)^{k+1}  G_{k+1}(\tau,\calL_U,u').
\]

\end{theorem}

\begin{proof}
We use Proposition~\ref{partialPoissonU}. For $k\geq 2$, termwise integration is valid and easily yields the result. For $k=0,1$, one needs to include a term $t^s$ with $Re(s)$ large into the integral before interchanging integration and summation.
\end{proof}

\begin{remark}
Theorem~\ref{SW} is a very simple special case of the extended Siegel-Weil formula of Kudla-Rallis \cite{KR} in the isotropic case. In our particular situation, the choice of the Schwartz function implies that no regularization of the theta integral is necessary.
\end{remark}

We now give an interpretation of the cohomological pairing $\langle  \theta_{\calL_U}(\tau,\varphi^U_{1,[k]}),X_U \otimes u'^k \rangle$ as a generating series of intersection numbers. 

A rational vector $\x \in U$ of positive length defines a point $C_\x^U$ (and hence a $0$-cycle) by the condition
\[
(\x,\x(t))=0.
\]
We set 
\[
\eps(\x) = \sgn(\x,\eps_1).
\]
Then we define for $n>0$
\[
C_{n,[k]}^U= \sum_{\substack{\x \in \calL_U \\ q(\x)=n}} \eps(\x) C^U_{\x} \otimes \pi_k(\x^k) \in H_0(X_U, \calH_k(U)).
\]
Note here that since $U$ is split there are only finitely many vectors $\x$ in $\calL_U$ of length $n$. 

For $n=0$, we define $C^U_{0,[k]}$ as follows. We can compactify $X_U$ by two "endpoints" $X_{U,\ell}$ (at $t=\infty$) and $X_{U,\ell'}$ (at $t=0$). We set
\[
C^U_{0,[k]}= - \delta_{0,h_2} M_1^{k} \frac{B_{k+1}\left(\frac{h_1}{M_1}\right)}{k+1}  \left(X_{U,\ell} \otimes u^k\right)
-\delta_{0,h_1} M_2^{k}\frac{B_{k+1}\left(\frac{h_2}{M_2}\right)}{k+1}  \left(X_{U,\ell'} \otimes u'^k \right).
\]
Here $\delta_{i,j}$ is the Kronecker delta.

\begin{theorem}\label{11-theorem}

Let $w \in \calH_{k}(U)$ and consider the relative $1$-cycle $X_U \otimes w$. For $k \ne 1$, we have 
\begin{align*}
\int_{X_U \otimes w} \theta_{\calL_U}(\tau,z,\varphi^U_{1,[k]}) &= \sum_{n \geq0} \langle C^U_{n,[k]},  X_U \otimes w \rangle e^{2\pi i n \tau} \\
&=
- \delta_{0,h_2} M_1^{k} \frac{B_{k+1}\left(\frac{h_1}{M_1}\right)}{k+1} (w,u^k)
- \delta_{0,h_1} M_2^{k}\frac{B_{k+1}\left(\frac{h_2}{M_2}\right)}{k+1}  (w,u'^k) \\ &\quad  + \sum_{n >0} \sum_{ \substack{\x \in \calL_U \\ (\x,\x)>0 }}  \eps(\x) (w,\x^k) e^{\pi i (\x,\x)\tau}
\end{align*}
is a holomorphic modular form of weight $k+1$ for $\G(M')$. That is, the cohomology class $[\theta(\tau,z,\varphi^U_{1,[k]})] \in H_c^1(X_U,\widetilde{\calH_{k}(U_\C)})$ defines a holomorphic modular form of weight $k+1$ for $\G(M')$ and is equal to the generating series of the $0$-cycles
\[
\sum_{n \geq 0} \PD[C^U_{n,[k]}] e^{2\pi i n \tau}.
\]
For $k=1$, the integral $\int_{X_U \otimes w} \theta_{\calL_U}(\tau,z,\varphi^U_{1,[1]})$ contains an additional term 
\begin{align*}
 \delta_{0,h_1} \frac{(w,u)}{4M_2\pi v} + \delta_{0,h_2} \frac{(w,u')}{4M_1\pi v}.
\end{align*}

\end{theorem}

\begin{proof}
We only do $w=u'^k$. The case of coefficient $u^k$ is analogous. 
Let $\x \in\calL_U$ such that $q(\x)>0$. We need to calculate $
\int_{D_U} (\varphi_{1,[k]}(\x,\tau),u'^k)$, where $\varphi_{1,[k]}(\x,\tau) = v^{-k/2} \varphi^U_{1,[k]}(\sqrt{v}\x,t) e^{\pi i(\x,\x)u}$. We can assume $\x = m \eps_1$ for some nonzero $m \in \R$. We write $\varphi^0_{1,[k]}(\x,v) = \varphi_{1,[k]}(\x,\tau) e^{-\pi i (\x,\x) \tau}$.
Switching variables to $t = e^r$, we obtain
\begin{align*}
\int_{D_U} (\varphi^0_{1,[k]}&(\x,v),u'^k) \\ &= \frac{(-1)^kv^{-\tfrac{k}{2}}(2\pi)^{-\tfrac{k+1}{2}} }{2^{(3k+1)/2}}\int_{-\infty}^{\infty} H_{k+1}\left(\sqrt{2\pi v}m \cosh(r)\right) e^{-2\pi m^2 v \sinh^2(r)}e^{kr} dr.
\end{align*}
We denote the integrand by $\Phi_k(r)$. Using the recurrence relation $H_j(y)=2yH_{j-1}(y)-H'_{j-1}(y)=2yH_{j-1}(y)-2(j-1)H_{j-2}(y)$ for the Hermite polynomials, a little calculation yields
\[
\Phi_k(r) =  2\sqrt{2\pi v}m\Phi_{k-1}(r) - 2\frac{\partial}{\partial r} \left( \Phi_{k-2}(r)e^{2r}\right).
 \]
Applying this recursion $k$ times we obtain
\begin{align*}
\int_{D_U} (\varphi^0_{1,[k]}(\x,v),u'^k)&= (\x,u')^k \pi^{-1/2} \int_{-\infty}^{\infty} \sqrt{2\pi v}m \cosh(r) e^{-2\pi m^2 v \sinh^2(r)} dr \\
&=  \sgn(m) (\x,u')^k \pi^{-1/2} \int_{-\infty}^{\infty} e^{-r^2}dr =  \sgn(m) (\x,u')^k.
\end{align*}
For the negative Fourier coefficients of the integral, we consider $\x \in\calL_U$ such that $q(\x)<0$. We can assume $\x = m \eps_2$ for some nonzero $m \in \R$. Then a similar recursion, reduces $\int_{D_U} (\varphi_{1,[k]}(\x,\tau),u'^k)$ to the case $k=0$, which is directly seen to vanish. 

For the constant coefficient we could of course just refer to the standard calculation for the constant term of the Eisenstein series $G_{k+1}(\tau,\calL_U,u)$. We give a more geometric approach. The constant coefficient arises from the isotropic vectors in $\calL_U$. Hence it is given by 
\[
\int_{X_U} \left[\delta_{0,h_2}\sum_{n \in M_1\Z+h_1}  (\varphi^0_{1,[k]}(nu,v),u'^k) + \delta_{0,h_1}\sideset{}{'}\sum_{m \in M_2\Z+h_2}  (\varphi^0_{1,[k]}(mu',v),u'^k)\right],
\]
where $\sideset{}{'}\sum$ indicates that the term $m =0$ is omitted if $h_2=0$. For the first summation one applies Poisson summation and exactly obtains the $m=0$ contribution to the constant coefficient in the definition of the Eisenstein series $G_{k+1}(\tau,\calL_U,u')$:
\[
- M_1^{k} k! \left( \frac{i}{2\pi}\right)^{k+1} \sum_{\substack{  n \in \Z} } e^{2\pi i nh_1/M_1} n ^{-k-1}=
(-1)^{k+1} M_1^{k} \frac{B_{k+1}\left(\frac{h_1}{M_1}\right)}{k+1},
\]
see for example \cite {Apostol}, chapter 12. The other term is equal to 
\begin{align*}
\lim_{s \to 0} \frac{(-1)^kv^{-\tfrac{k}{2}}(2\pi)^{-\tfrac{k+1}{2}} }{2^{(3k+1)/2}} \int_0^{\infty}
\sideset{}{'}\sum_{m \in M_2\Z+h_2} H_{k+1}(-\sqrt{\pi v} mt)e^{-\pi v m^2t^2} t^{k+s} \frac{dt}{t}.
\end{align*}
We introduced the complex variable $s$ so that we can now interchange summation and integration (which would not be possible for $k=1$). We obtain

\[
\begin{split}
 \frac{-v^{-\tfrac{k}{2}}(2\pi)^{-\tfrac{k+1}{2}} }{2^{(3k+1)/2}}(\sqrt{\pi v}M_2)^{-k-s} \left( H\left((k+s,\tfrac{h_2}{M_2}\right)+ (-1)^{k+1}H\left(k+s,1-\tfrac{h_2}{M_2}\right) \right)  \\ \times \int_{0}^{\infty} H_{k+1}(t)t^{k+s-1}e^{-t^2} dt.
\end{split}
\]
Here $H(s,x) = \sum_{n=1}^{\infty} (n+x)^{-s}$ is the Hurwitz zeta function. Now for $k\geq1$ and $s=0$ we have $\int_{0}^{\infty} H_{k+1}(t)t^{k-1}e^{-t^2} dt = \tfrac12\int_{-\infty}^{\infty} H_{k+1}(t)t^{k-1}e^{-t^2} dt =0$ by the orthogonality of the Hermite polynomials. This gives vanishing for $k>1$, while for $k=1$, we easily compute $\int_{0}^{\infty} H_{2}(t)t^{s}e^{-t^2} dt = s \G((s+1)/2)$. We then obtain $-1/(4M_2\pi v)$ as $s\to 0$.  
\end{proof}

\begin{remark}
The theorem also holds for $k=0$. This case was initially considered by Kudla \cite{KAnnalen}, Theorem~3.2. In that case one obtains Hecke's Eisenstein series of weight $1$. Kudla considered in \cite{KAnnalen,KShintani} more generally $0$-cycles for $\SO(n,1)$ which give rise to holomorphic Siegel modular forms of degree $n$. 
\end{remark}

\begin{example}
We consider the level $1$ situation with $\calL_U=\Z u \oplus \Z u'$. Then 
\[
\langle\theta_{\calL_U}(\tau,\varphi^U_{1,[k]}), X_U \otimes u'^k  \rangle = -\frac{B_{k+1}}{k+1} E_{k+1}(\tau) = -\frac{B_{k+1}}{k+1} + 2 \sum_{x,y>0} x^k e^{2\pi i xy\tau}
\]
Here $E_{k+1}(\tau)$ is the standard Eisenstein series of weight $k+1$ for $\SL_2(\Z)$ given for $k>1$ by $\tfrac12\sum_{\g \in {\G}_{\infty} \back {\SL_2(\Z)}} j(\g,\tau)^{-k-1}$.
\end{example}

\section{The generating series of the spectacle cycles}\label{Main-Section}

In this section we state our main result of this paper.

\subsection{Generating series of modular symbols}

We first recall the classical result of Shintani in our setting.

We define the theta series associated to $\varphi_{1,[k]}^V$ and the (coset of the) lattice $\calL_V=\calL=L+h$ in the usual way by 
\[
\theta_{\calL}(\tau,z,\varphi^V_{1,[k]}) = v^{-k/2}\sum_{\x \in \calL} \varphi^V_{1,[k]}(\sqrt{v}\x,z) e^{\pi i (\x,\x)u}.
\]
Then the principal result of \cite{FMcoeff} (in much greater generality) realizes the cohomology class of this theta series as a holomorphic modular form and generating series of special cycles. In the present case, we recover the Shintani lift \cite{Shin}.

\begin{theorem}\cite{Shin,FMcoeff}\label{Shintani}
The cohomology class $[\theta(\tau,z,\varphi^V_{1,[k]})] \in H^1(X,\widetilde{\calH_{k}(V_\C)})$ defines a holomorphic cusp form of weight $k+3/2$ for $\G(M)$ and is equal to the generating series of the (Poincar\'e duals of the) modular symbols 
\[
\sum_{n \in \Q_+} [C_{n,[k]}] e^{2\pi i n \tau}.
\]
Here $M$ is the level of lattice $L$. That is, for any compactly supported or rapidly decreasing closed $1$-form $\eta$ on $X$ with values in 
$\calH_{k}(V_\C)$ representing a class in $H^1_c(X,\widetilde{\calH_{k}(V_\C)})$ we have
\[
\int_X \eta \wedge \theta(\tau,z,\varphi^V_{1,[k]}) =  \sum_{n \in \Q_+} \left(\int_{C_{n,[k]}} \eta \right) e^{2\pi i n \tau} \in S_{k+3/2}(\G(M)).
\]
Equivalently, for any absolute $1$-cycle $C$ in $X$ with coefficients representing a class in $H_1(X,\widetilde{\calH_{k}(V)})$, we have 
\[
\int_C \theta(\tau,z,\varphi^V_{1,[k]}) =
\sum_{n \in \Q_+}  ({C_{n,[k]}}\bullet C)e^{2\pi i n \tau} \in S_{k+3/2}(\G(M)).
\]
Here $\bullet$ denotes the intersection product of cycles. 
\end{theorem}

\begin{remark}
Shintani formulates his result not in terms of cohomological pairings but rather in terms of weighted periods of holomorphic cusp forms. Of course he also uses a theta lift to obtain this result. However, he employs a different, scalar-valued, theta kernel $\theta(\tau,z,\varphi_S)$ which is integrated against a holomorphic cusp form $f$. Shintani's kernel function at the base point $z=i$ is given by
\[
\varphi_S(\x) = (x_1+ix_2)^{k+1} \varphi_0(\x).
\]
For such input, the kernels are closely related, namely one has 
\[
\eta_f \wedge \theta(\tau,z,\varphi^V_{1,[k]}) =  (-1)^{k} 2^{(k+1)/2}\theta(\tau,z,\varphi_S) f(z) y^{k+1} d\mu(z),
\]
where $d\mu(z) = \tfrac{dxdy}{y^2}$ is the invariant measure on $D$. This can be seen by a direct calculation of $ dz \otimes n(z)u'^k \wedge \varphi_{1,[k]}(\x,z)$. It boils down to the Hermite identity
\[
\sum_{j=0}^k \begin{pmatrix} k+1 \\ j \end{pmatrix} (-i)^{j} \widetilde{H}_{k+1-j}(x_1)\widetilde{H}_{j}(-x_2) =(x_1+ix_2)^{k+1}.
\]
 \end{remark}

\begin{remark}
Since Shintani only considers the lift of cusp forms, he actually obtains a priori a slightly weaker result because $H^1_c(X,\widetilde{\calH_{k}(V_\C)})$ not only consists of classes arising from cusp forms but also of the image of $H^0(\partial \overline{X},\widetilde{\calH_{k}(V_\C)})$ in $H^1_c(X,\widetilde{\calH_{k}(V_\C)})$ under the long exact cohomology sequence. However, this makes no difference for $k>0$, since in that case it is easy to see that the lift of classes arising from $H^0(\partial \overline{X},\widetilde{\calH_{k}(V_\C)})$ vanishes. For $k=0$, one obtains via the restriction formula $\theta(\tau,z,\varphi^V_{1,0})$ (see Theorem~\ref{restriction} below) for the lift of $H^0(\partial \overline{X},\C)$ unary theta series of weight $3/2$. 
\end{remark}

\subsection{Restriction of the theta series and the construction of a class in the mapping cone}\label{restriction99}

For the coset $\calL_V=\calL = L +h$ of the lattice $L$ in $V$, we can write
\[
\calL_V\cap \ell^{\perp} = \coprod_j \left(L \cap \ell+ h_{\ell,j} \right)
\oplus \left(L \cap W_{\ell}+ h_{W_\ell,j}\right)
\]
with vectors $h_{\ell,j} \in (L\cap \ell)^{\#}$ and $h_{W_\ell,j} \in (L\cap W_\ell)^{\#}$ (if $\calL_V\cap \ell^{\perp} \ne \emptyset$). Then we set
\[
\hat{{\calL}}_{W_\ell} = \frac{1}{  \det L_\ell}   
\sum_j
\left(L_{W_\ell} +
h_{W_\ell,j}\right),
\]
where $L_{W_\ell} = L \cap W_\ell$. Here $\det L{_\ell} :=M_1$ where $L_{\ell} = L \cap \ell = M_1\Z u_{\ell}$. Here we view the sum as a sum of characteristic functions of sets on $W$. We then define the positive theta series associated to $\varphi^{N_{\ell}}_{1,[k]}$ and $\hat{{\calL}}_{W_\ell}$ by 
\begin{align*}
\theta_{\hat{\calL}_{W_\ell}}(\tau,\varphi^{N_{\ell}}_{1,[k]}) &= v^{-(k+1)/2}  \sum_{\w \in W} \hat{{\calL}}_{W_\ell}(\w)
\varphi^{N_{\ell}}_{1,[k]}(\sqrt{v}\w) e^{\pi i (\w,\w)u} \\ &= \frac{v^{-(k+1)/2}}{  \det L_\ell}\sum_{j}  \sum_{\w \in 
L_{W_\ell} +h_{W_\ell,j}}\varphi^{N_{\ell}}_{1,[k]}(\sqrt{v}\w) e^{\pi i (\w,\w)u}.
\end{align*}
Then $\theta_{\hat{\calL}_{W_\ell}}(\tau,\varphi^{N_{\ell}}_{1,[k]})$ transforms like a non-holomorphic modular form of weight $k+3/2$. In the same way we define $\theta_{\hat{\calL}_{W_\ell}}(\tau,\phi^{N_{\ell}}_{[k]})$.

\begin{theorem}\label{restriction}
The differential $1$-form $(\theta_{\calL_V}(\tau,z,\varphi^V_{1,[k]})$ on $X$ extends to a form on the Borel-Serre compactification $\overline{X}$. More precisely, we let $\iota_{\ell}: X_{\ell} \hookrightarrow \overline{X}$ be the natural inclusion of the boundary face $X_{\ell}$. Then for the restriction we have 
 \[
\iota^{\ast}_{\ell} \theta_{\calL_V}(\tau,\varphi^V_{1,[k]})= \theta_{\hat{\calL}_{W_\ell}}(\tau,\varphi^{N_{\ell}}_{1,[k]}).
\]
 Moreover, $\iota^{\ast}_{\ell}\theta_{\calL_V}(\tau,z,\varphi^V_{1,[k]})$ is an exact form on $X_{\ell}$, and we have 
\[
\iota^{\ast}_{\ell}\theta_{\calL_V}(\tau,\varphi^V_{1,[k]}) = d \left( \theta_{\hat{\calL}_{W_\ell}}(\tau,\phi^{N_{\ell}}_{[k]}) \right).
\]
\end{theorem} 

\begin{proof}
The extension to $\overline{X}$ and the restriction formula are (in much greater generality) the main themes of \cite{FMres}. To obtain such a result we employ partial Poisson summation on $L \cap \ell$, that is, on the $u$-summation. The exactness at the boundary face $X_{\ell}$ follows from Proposition~\ref{primitive}.
\end{proof}

From the appendix we can therefore conclude

\begin{corollary}
The pair $\left(\theta_{\calL_V}(\tau,\varphi^V_{1,[k]}), \sum_{[\ell]}  \theta_{\hat{\calL}_{W_\ell}}(\tau,\phi^{N_{\ell}}_{[k]}) \right)$ defines a cohomology class of the mapping cone associated to $\iota: \partial \overline{X} \to \overline{X}$. Here the sum extends over all $\G$-equivalence classes of rational isotropic lines.
\end{corollary}

\subsection{The main result}

For a rational isotropic line $\ell$ in $V$, we define
\[
C_{\ell,[k]} = 
\begin{cases}
-M_{\ell}^k\frac{B_{k+1}\left(\frac{h_{\ell}}{M_{\ell}}\right)}{k+1} (X_{\ell} \otimes u_{\ell}^k)
 & \text{if $\calL \cap \ell = (M_{\ell}\Z +h_{\ell}) u_{\ell}$} \\
0 & \text{if $\calL \cap \ell = \emptyset$}.
\end{cases}
\]
We then set 
\[
C_{0,[k]}^c = \sum_{[\ell]} C_{\ell,[k]} \in H_1(\overline{X},\widetilde{\calH_k(V)}),
\]
where the sum extends over all $\G$-equivalence classes $[\ell]$ of rational isotropic lines in $V$.

The main result of the paper is

\begin{theorem}\label{Main}
The mapping cone element $\left[\theta_{\calL_V}(\tau,\varphi^V_{1,[k]}), \sum_{[\ell]}  \theta_{\hat{\calL}_{W_\ell}}(\tau,\phi^{N_{\ell}}_{[k]}) \right]$ representing a class in $H^1_c({X},\widetilde{\calH_{k}(V_\C)})$ defines a holomorphic modular form of weight $k+3/2$ for $\G(M)$ and is equal to the (Poincar\'e duals of the) generating series of the spectacle cycles with coefficients
\[
\sum_{n \geq0} [C^c_{n,[k]}] e^{2\pi i n \tau}.
\]
That is, for any closed $1$-form $\eta$ in $\overline{X}$ with values in 
$\calH_{k}(V_\C)$ representing a class in $H^1(\overline{X},\widetilde{\calH_{k}(V_\C)})\simeq H^1(X,\widetilde{\calH_{k}(V_\C)})$ we have
\[
\left\langle \eta, \left[\theta_{\calL_V}(\tau,\varphi^V_{1,[k]}), \sum_{[\ell]}  \theta_{\hat{\calL}_{W_\ell}}(\tau,\phi^{N_{\ell}}_{[k]}) \right] \right\rangle =
\sum_{n \geq0} \left(\int_{C^c_{n,[k]}} \eta \right) e^{2\pi i n \tau} \in M_{k+3/2}(\G(M)).
\]
Equivalently, for any relative $1$-cycle $C$ in $\overline{X}$ with coefficients representing a class in $H_1(\overline{X},\partial \overline{X},\widetilde{\calH_{k}(V)})$, we have 
\[
\left\langle \left[\theta_{\calL_V}(\tau,\varphi^V_{1,[k]}), \sum_{[\ell]}  \theta_{\hat{\calL}_{W_\ell}}(\tau,\phi^{N_{\ell}}_{[k]}) \right],C \right\rangle =
\sum_{n \geq0} ({C^c_{n,[k]}} \bullet C) e^{2\pi i n \tau} \in M_{k+3/2}(\G(M)).
\]
\end{theorem}

\begin{proof}
By Theorem~\ref{Shintani} the homology version of the assertion holds for the image of the full homology  $H_1(\overline{X},\widetilde{\calH_{k}(V)})$ inside the relative homology $H_1(\overline{X},\partial \overline{X},\widetilde{\calH_{k}(V)})$. Therefore it suffices to show the theorem for representatives of the quotient  \linebreak $H_1(\overline{X},\partial \overline{X},\widetilde{\calH_{k}(V)})/H_1(\overline{X},\widetilde{\calH_{k}(V)})$. This means it suffices to consider preimages under the (surjective) boundary map
\[
\partial: H_1(\overline{X},\partial \overline{X},\widetilde{\calH_{k}(V)}) \to H_1(\partial \overline{X},\widetilde{\calH_{k}(V)}) = \bigoplus_{[\ell]} H_0(X_\ell,\widetilde{\calH_k(V)}).
\]
Now $H_0(X_\ell,\calH_k(V))$ is one-dimensional and spanned by $c \otimes u_{\ell}'^k$ for any point $c \in X_\ell$. We consider the Witt splitting $V = \ell \oplus W_{\ell} \oplus \ell'$ and take a nonzero $\y\in W_\ell$. Then 
\[
\partial [C_\y \otimes u_{\ell}'^k] = [c_\y \otimes u_{\ell}'^k] - [c'_\y \otimes u_{\ell}'^k] =[c_\y \otimes u_{\ell}'^k],
\]
since $(c'_\y \otimes u_{\ell}'^k)$ is trivial in $H_0(X_{\ell'},\widetilde{\calH_k(V)})$ because $u_{\ell}'^k$ is a highest weight vector for the cusp $\ell'$. Hence it suffices to compute the lift for $C_\y \otimes u_{\ell}'^k$. We will carry this out in the next section. 
\end{proof}

\begin{remark}
There are other approaches to this result (for which we did not check details). One can consider $\left\langle \eta, \left[\theta_{\calL_V}(\tau,\varphi^V_{1,[k]}), \sum_{[\ell]}  \theta_{\hat{\calL}_{W_\ell}}(\tau,\phi^{N_{\ell}}_{[k]}) \right] \right\rangle$ for a closed $1$-form $\eta$ on $\overline{X}$ which in a neighborhood of $X_{\ell}$ is equal to the pullback of $\omega_{\ell,k}$, the standard generator of $H^1(X_{\ell}, \widetilde{\Hk(V_{\C})})$. For example, the holomorphicity of the lift can be shown directly via an argument very similar to the one given in \cite{KM90,FMcoeff}. Alternatively, using Lemma~\ref{integralformula} one can also consider the lift of a holomorphic $1$-form $\eta_f$ associated to a modular form $f$. Since the lift of an Eisenstein series is again an Eisenstein series (see Section~\ref{Eisenstein-Section}) an explicit calculation of Fourier coefficients should be also feasible. 

We choose the present approach (which is a bit in the spirit of \cite{HZ}), since it gives us the opportunity to discuss in detail the very pretty example of the lift of a modular symbol. 

\end{remark}

\section{Lift of modular symbols}\label{Modular-Section}

In this section, we prove Theorem~\ref{Main} by considering the lift for a non-compact cycle of the form $C_{\y} \otimes u_{\ell}'^k$. 

More precisely, starting from an isotropic line $\ell= \Q u$ we find an $\G$-inequivalent isotropic line $\ell' = \Q u'$ with $(u,u')=-1$ (the torsion-free congruence subgroup $\G$ has at least two cusps). We then obtain a rational Witt splitting $V = \ell \oplus W \oplus \ell'$, with $W= \Q\y$ so that $u,\y,u'$ is an oriented basis of $V$. Hence $D_{\y}=D_W$  joins the two {\it inequivalent} cusps $\ell'$ and $\ell$.

\begin{lemma}\label{embedLemma}
The infinite geodesic $C_\y$ connecting two distinct cusps embeds into $\overline{X}$.
\end{lemma}

\begin{proof}
Let $\calF$ be a convex fundamental domain for $\G$ (note that the Dirichlet domain is convex).  The cusps $\ell$ and $\ell'$ correspond to vertices of $\calF$. These vertices may be joined in $\calF$ by a unique infinite geodesic (because $\calF$ is convex) which maps one-to-one onto $C_\y$.
\end{proof}

Then Theorem~\ref{Main} will follow from

\begin{theorem}\label{TT1}
\begin{equation}\label{33}
\left\langle  \left[\theta_{\calL_V}(\tau,\varphi^V_{1,[k]}), \sum_{[\ell]}  \theta_{\hat{\calL}_{W_{\ell}}}(\tau,\phi^{N_{\ell}}_{[k]}) \right],(C_\y \otimes u'^k) \right\rangle = \sum_{n \geq0} ({C^c_{n,[k]}} \bullet (C_\y \otimes u'^k)) e^{2\pi i n \tau}.
\end{equation}
\end{theorem}

The proof will occupy the rest of the section. We will compute both sides of \eqref{33} explicitly. For the left hand side we have by Lemma~\ref{integralformula}
\begin{align}
&\left\langle\left[\theta_{\calL_V}(\tau,\varphi^V_{1,[k]}), \sum_{[\ell]}  \theta_{\hat{\calL}_{W_\ell}}(\tau,\phi^{N_{\ell}}_{[k]}) \right],  (C_{\y} \otimes u'^k)\right\rangle \notag \label{11}\\
& = \int_{C_{\y}}\left( \theta_{\calL_V}(\tau,\varphi^V_{1,[k]}),u'^k\right) \\
& \quad - \left( \theta_{\hat{\calL}_{W_{\ell}}}(\tau,c_\ell,\phi^{N_{\ell}}_{[k]}),u'^k\right)  + \left(\theta_{\hat{\calL}_{W_{\ell'}}}(\tau,c_{\ell'},\phi^{N_{\ell'}}_{[k]}), u'^k\right). \label{22}
\end{align}
For simplicity we assume
\[
\calL_V = (M_1\Z + h_1)u \oplus L_W + h_{W} \oplus  (M_2\Z + h_2)u',
\]
where $L_W = L \cap W, h_W \in L_W^\#$ and $0\leq h_i <M_i$. (The general case goes along the same lines but requires more notation). We set $\calL_{W} =  L_W + h_{W}$ and $U = W^{\perp}$.

\begin{proposition}\label{TT2}
For the integral \eqref{11} we have 
\begin{align*}
& \int_{C_{\y}}\left( \theta_{\calL_V}(\tau,\varphi^V_{1,[k]}),u'^k\right)\\ & = 
\sum_{\substack{ \x \in U, \w \in W \\ \x+\w \in \calL_V \\ (\x,\x)>0}} \eps(\x) (\x,u')^k e^{\pi i(\x+\w,\x+\w)\tau} + 
(-1)^{k+1} \delta_{0,h_2} M_1^{k} \frac{B_{k+1}\left(\frac{h_1}{M_1}\right)}{k+1} 
       \sum_{\w \in \calL_W} e^{\pi i (\w,\w)\tau}.
\end{align*}
For $k=1$, we have an additional term 
\[
\frac{-\delta_{0,h_1}}{4M_2\pi v} \sum_{\w \in \calL_W} e^{\pi i (\w,\w)\tau}.
\]
Here $\eps(\x)= \eps(\x, \y) = \pm 1$ depending on whether $\x,\y$ defines a properly oriented basis of the tangent space for the point $z\in D$ determined by $\{\x,\y\}^{\perp}$. This coincides with the definition of $\eps(\x)$ for $\x \in U$ given in Section~\ref{1-1-Section}.

\end{proposition}

\begin{proof}
By Lemma~\ref{varphi-splitting} we immediately obtain
\[
 \int_{C_{W}}\left( \theta_{\calL_V}(\tau,\varphi^V_{1,[k]}),u'^k\right) = \left( \int_{C_{W}}\left( \theta_{\calL_U}(\tau,\varphi^U_{1,[k]}),u'^k\right) \right) \theta_{\calL_W}(\tau,\varphi^W_0),
\]
where  $\calL_{U} = (M_1\Z + h_1)u \oplus  (M_2\Z + h_2)u'$. The proposition now follows from Theorem~\ref{11-theorem}.
\end{proof}

\begin{lemma}\label{TT3}
For $k>1$, we have for the integral \eqref{11} 
\[
\left( \theta_{\hat{\calL}_{W_{\ell}}}(\tau,c_\ell,\phi^{N}_{[k]}),u'^k\right)=\left(\theta_{\hat{\calL}_{W_{\ell'}}}(\tau,c_{\ell'},\phi^{N_{\ell_0}}_{[k]}), u'^k\right)=0.
\]
For $k=1$, we have
\[
\left( \theta_{\hat{\calL}_{W_{\ell}}}(\tau,c_\ell,\phi^{N}_{[1]}),u'\right) =0,
\]
and 
\[
\left(\theta_{\hat{\calL}_{W_{\ell'}}}(\tau,c_{\ell'},\phi^{N_{\ell'}}_{[1]}), u'\right) =  \frac{\delta_{0,h_1} }{M_2}  \sum_{\w \in \calL_W}\left(-(\w,\w) + \frac1{4\pi v}\right) e^{\pi i (\w,\w) \tau}.
\]
\end{lemma}

\begin{proof}
The coefficients for the boundary theta series at $X_\ell$ and $X_{\ell'}$ are the weight vectors $v_{-2}$ and a multiple of $v_{2}$ respectively. Then the claim for $k>1$ is obvious, since $u'^k$ is perpendicular to both. If $k=1$, the restriction of $\theta_{\calL_V}(\tau,\varphi^V_{1,[k]})$ to the cusp $\ell'$ is non zero if and only $\calL_V \cap (\ell')^{\perp} \ne \emptyset$, that is $h_1=0$. Then
\begin{align*}
\left(\theta_{\hat{\calL}_{W_{\ell'}}}(\tau,c_{\ell'},\phi^{N_{\ell'}}_{[1]}), u'\right) 
= \frac{\delta_{0,h_1} }{M_2} \sum_{\w \in \calL_W} \left(-(\w,\w) + \frac1{4\pi v}\right) e^{\pi i (\w,\w) \tau}. 
\end{align*} 
This follows easily from $ {H}_{2}(t) = 4t^2-2$.
\end{proof}

The following results compute the intersection number $C^c_{n,[k]} \bullet (C_\y \otimes u'^k)$ (as defined by algebraic topology). Recall that by Lemma~\ref{embedLemma} the geodesic $C_\y$ embeds into $X$. We first note

\begin{lemma}
Assume that one of the components $C_v$ of $C_n$ coincides with $C_\y$ (i.e., $v$ is a multiple of $\y$). Then the intersection
$C^c_{\y}\bullet C_{\y}$ consists only of the intersections of the caps of $C^{c,T_1,T_2}_{\y}$ with $C_\y$.
\end{lemma}

\begin{proof}
We first push-in $C^c_{\y}$ at the cusps (as described in Lemma~\ref{homology1}) to obtain $C^{c,T_1,T_2}_{\y}$. Then we can apply a small normal deformation to $C_\y$ obtaining a cycle $C'_\y$ which is disjoint from $C_\y$. Here we need that $C_\y$ has no self-intersections. We then see that $C^c_{\y}\bullet C_{\y}$ consists only of the intersections of $C_\y$ with the caps of $C^{c,T_1,T_2}_{\y}$. 
\end{proof}

Note that two modular symbols $C_{\x_1}$ and $C_{\x_2}$ either coincide or intersect in a finite number of points in the interior of the Borel-Serre compactification $\overline{X}$. These intersections are generally transverse (and in this case we may compute this intersection number by counting intersection points with signs). However, it is possible that $C_\y$ intersects the component $C_v$ of $C_n$ at an $m$-fold multiple point $p$ of $C_v$ (note that $C_v$ may have self-intersections). 
We claim the contribution of $p$ to the global intersection number is the sum of these $n$ multiplicities
(in the sense that if we add up the results over all $p$ we get the global intersection number).
This may be seen in two ways.  First we may triangulate  each of $C_\y$ and $C_v$ and use
the simplicial intersection number of transverse simplicial cycles.  More geometrically we can replace
$C_\y$ by a path which agrees with $C_\y$ except in a small neighborhood of $p$ where we replace
a segment of $C_\y$ by a small semicircle that ``hops over $p$ ''.  The deformed chain is
clearly homologous to $C_\y$ and intersects $C_v$ transversally in $m$ points close to $p$.
We leave to the reader to check that the multiplicities of these $m$ intersections agree with
the ones below. By slight abuse of notation we still refer to these intersections as transverse.

In conclusion, to compute $C^c_{n,[k]} \bullet (C_\y \otimes u'^k) $ we need to compute the transversal intersection numbers in $X$ (as modified above if $C_v$ has multiple points) and the intersections of $C_W \otimes u'^k$ with the boundary caps of $C^c_{n,[k]}$.

We first treat the intersection numbers in the interior of $X$. Recall that in subsection \ref{intermult} we
defined a function  $\eps(\x,\y)$ of pairs $\x,\y$ that span a positive two-plane with values in $\{ \pm 1\}$ in terms of the cross-product
 $\x \times \y$ in Minkowski three-space and showed this function was the intersection multiplicity of of $D_\x$ and $D_\y$. 

\begin{proposition}\label{TT4}
The transversal intersection number of $C^c_{n,[k]}$ with $C_\y \otimes u'^k$ in the interior of $\overline{X}$ is given by 
\[
\sum_{\substack{ \x \in U, \w \in W \\ \x+\w \in \calL_V \\  (\x,\x)>0 \\ q(\x+\w)=n}} \eps(\x, \y) (\x,u')^k
\]

\end{proposition}

\begin{proof}
We first show that such vectors $\x,\w$ indeed parameterize all transversal intersections. A transversal intersection point $p \in X $ of $C_\y$ and $C_n$ arises from an intersection of $D_\y$ with a $D_v$ at a {\it unique} point $z$ in $D_\y \in D$. (Unique since $C_\y$ embeds into $X$). Here $v \in \calL$ with $q(v)=n$. Now $v \notin W$, for otherwise we would have $D_v=D_\y$. Hence we can write $v=\x+\w$ with $\x \in U=W^{\perp}$ and $\w\in W$ with $\x \ne 0$. If $\x$ did not have positive length, then $\y$ and $v$ would not span a positive definite space of dimension $2$ and hence would not determine a unique point in $D$. Thus $\x$ and $\w$ are as above. If $p$ is a simple point, then the vector $v$ is unique. However, if $p$ is a multiple point of order $m$ of $C_v$, then each of the branches of $C_v$  corresponds to a geodesic $D_{v_i}$ meeting $D_\y$ at $z \in D$. The set of $D_{v_i}'s$ lie in a single $\G$-orbit so all the $v_i$'s satisfy $q(v_i) = n$. Then we decompose the $v_i$'s as above. 

Conversely, given $\x$ and $\w$ as above we put $v = \x + \w$ and we obtain  a transversal intersection point $z \in D$ of $D_v$ and  $D_\y$.  Thus we have established a one-to-one correspondence between the interior transversal intersections
of $C_n$ and $C_\y$ and the index set of the above sum.

As noted above we computed the intersection multiplicity $\eps(v,\y)$ of $D_v$ and $D_{\y}$ in subsection \ref{intermult}.  To complete the proof we have only to note
$$\eps(v,\y)= \eps(\x,\y)$$
since we have an equality of cross-products $v \times \y = \x \times \y$. 
\end{proof}

We finally turn to the boundary intersections. We denote the components of $C_{n,[k]}^c$ at $X_{\ell}$ and $X_{\ell'}$ by $C_{n,[k]}^{\ell}$ and $C_{n,[k]}^{\ell'}$ respectively. Clearly only these will contribute to the intersection with $C_\y \otimes u'^k$.

\begin{proposition}\label{TT5}

We define a constant $c$ by $c=2$ if $2h_W \in L_W $ and $c=1$ otherwise. Then

\begin{itemize}

\item[(i)]

The intersection of $C_{n,[k]}^{\ell}$ with $C_W \otimes u'^k$ is given by 
\[
c (-1)^{k+1} \delta_{0,h_2} M_1^{k} \frac{B_{k+1}\left(\frac{h_1}{M_1}\right)}{k+1},
\]

\item[(ii)]

The intersection of $C_{n,[k]}^{\ell'}$ with $C_W \otimes u'^k$ is zero if $k>1$. For $k=1$, we have $-2 c n\frac{\delta_{0,h_1} }{M_2}$.

\end{itemize}

\end{proposition}

\begin{proof}

We can assume that $C_\y$ connects the cusps $0$ and $i\infty$, i.e., $\y$ is a positive multiple of $e_2$ and $\ell = \ell_{\infty}$ etc. The vectors $\x \in \calL_n$ such that $C_\x$ intersects $X_{\ell}$ are $\G$-equivalent to vectors which are perpendicular to $u$. So if $h_2 \ne 0$, there are no such vectors. Otherwise we can assume  $\x=\pm m\sqrt{2}e_2 + (M_1 j+h_1)u \in \calL$ with $j \in \Z$ and $\sqrt{n}=m \in \Q_+$. We assume "$+$" for the moment. These are exactly the geodesics which terminate in $X_{\ell}$ (taking the orientation of $C_\x$ into account). For the generator $n(M_{\ell})$ of the stabilizer of the cusp $\G_{\ell}$ we have $n(M_{\ell})(m\sqrt{2}e_2 +h_1u) = (m\sqrt{2}e_2 +(-2mM_{\ell}+h_1)u$. Hence we have $2mM_{\ell}/M_1$ $\G$-inequivalent vectors of this form (with the "$+$" condition). Namely,
\[
\x = n\left(\tfrac{-jM_1}{2m}- \tfrac{h_1}{2m}\right) (m\sqrt{2}e_2) \qquad \qquad (j=0,\dots, 2mM_{\ell}/M_1-1).
\]
 According to Proposition~\ref{Propv_x} the $u^k$ component for the cap-coefficient is given by 
 \[
 -(2mM_{\ell})^k\tfrac{B_{k+1}\left(\tfrac{jM_1}{2mM_{\ell}}+ \tfrac{h_1}{2mM_{\ell}}\right)}{k+1}u^k.
 \]
The negative sign arises since we have to take $-w_\x$ in the definition of the cap, see Definition~\ref{specdef}. Taking the inner product with $u'^k$, the coefficient of $C_W$, and summing over $j$ yields  
\[
(-1)^{k+1} M_1^{k} \frac{B_{k+1}\left(\frac{h_1}{M_1}\right)}{k+1}
\]
by the multiplication property of the Bernoulli polynomial. The vectors with the "$-$" condition give exactly the geodesics which originate from $\ell$. (Note that these vectors are not necessarily $\G$-inequivalent to the ones satisfying the "$+$" condition. If equivalent, the geodesic both originates from and returns to $X_{\ell}$). A similar analysis yields the same answer. But both cases can only occur simultaneously if $2h_W \in L_W$. This shows (i).
 
For (ii), we first note that for $k>1$, the cap vector $w_{\x}$ for $C_{\x}$ at the cusp $\ell'$ only involves weight vectors of weight at most $2$. So the pairing with $u'^k$ will be zero. For $k=1$ we apply a similar analysis as for (i). We first note that we have 
$2mM_{\ell'}/M_2$ $\G$-inequivalent vectors $\x$ of the form $m\sqrt{2}e_2+ jM_2u'$ giving rise to cycles which originate from $X_{\ell'}$ (if $h_1 \ne 0$). Using Proposition~\ref{Propv_x} the corresponding caps for these vectors involve all $-\tfrac{m}{M_{\ell'}}u$. Pairing with $u'$ hence gives $-\frac{2m^2}{M_2}$ (taking the incidence number $-1$ into account). The vectors with the "$-$" condition yield the same.
\end{proof}

Now Theorem~\ref{TT1} follows from comparing the combined Fourier coefficients of the modular forms computed in Proposition~\ref{TT2} and Lemma~\ref{TT3} with the intersections given in Proposition~\ref{TT4} and Proposition~\ref{TT5}.

\section{Lift of Eisenstein series}\label{Eisenstein-Section}

In this section we discuss the lift of Eisenstein cohomology classes. For $z \in D$, we define the Eisenstein series for the cusp $\infty$ by 
\[
\calE(z,k) =  \frac12 \sum_{\g \in \G_{\infty} \back \G} \g^{\ast}(dx \otimes u'^k).
\]
It is well known (eg \cite{Stevens}, \S 6) that for $k>0$ the series $\calE(z,k)$ converges absolutely and defines a closed differential $1$-form on the Borel-Serre compactification $\overline{X}$. Its restriction to $X_{\infty}$ is equal to $dx \otimes u'^k$ and zero at the other cusps. Finally, $\calE(z,k)$ is cohomologous to the differential form $\eta_{E_{2k+2}}=E_{2k+2} \otimes dz \otimes n(z)u'^k$ defined by the usual holomorphic Eisenstein series $E_{2k+2}$ for $\G$ at the cusp $\infty$. Note that the holomorphic $1$-form $\eta_{E_{2k+2}}$ does not extend to $\overline{X}$.

For simplicity we restrict to the case $\G=\SL_2(\Z)$. Furthermore, it is quite convenient to use Borcherds vector-valued modular forms/theta series setting. To compute the lift $\calE(z,k)$ we adapt the argument given in \cite{BFCrelle}, section ~7 to our situation.

We let $\Mp_2(\R)$ be the two-fold cover of $\SL_2(\R)$ realized by the two choices of holomorphic square roots of $\tau \mapsto j(g,\tau) = c\tau + d$, where $g = \left( \begin{smallmatrix} a&b \\ c&d \end{smallmatrix} \right) \in \SL_2(\R)$. Hence elements of $\Mp_2(\R)$ are of the form $(g,\phi)$ with $\phi^2= j(g,\tau)$. Then given an even lattice $L$ there is the Weil representation $\rho_L$ of the inverse image $\G'$ of $\SL_2(\Z)$ in $\Mp_2(\R)$, acting on the group algebra $\C[L^{\#}/L]$ (see \cite{Bo1}). We denote the standard basis elements of $\C[L^{\#}/L]$ by $\frake_{h}$, where $h \in {L^{\#}/L}$.

In our situation we consider the lattice 
\[
L=\left\{\zxz{b}{c}{a}{-b};\quad a,b,c\in \Z\right\}.
\]
We have $L^\#/L\cong \Z/2\Z$, the level of $L$ is $4$, and $\Gamma=\Sl_2(\Z)$ takes $L$ to itself and acts trivially on $L^\#/L$. We let $\mathfrak{e}_0, \mathfrak{e}_1$ be the standard basis of  $\C[L^\#/L]$ corresponding to the cosets $h= \left( \begin{smallmatrix} h_1 & 0 \\ 0& -h_1 \end{smallmatrix} \right)$ with $h_1 = 0$ and $h_1 = 1/2$, respectively.
We let $K =K_W= \Z\kzxz{1}{}{}{-1}$ be the $1$-dimensional lattice in the positive definite subspace $W$ in $V$. We frequently identify $K$ (resp. $W$) with $\Z$ ( resp. $\Q$) so that $(b,b')= 2bb'$. We naturally have $L^{\#}/L \simeq K^{\#} / K$ and $\rho_L \simeq \rho_K$. Finally we easily see $\hat{L}_{W}=K_W$, see Subsection~\ref{restriction99}. 

We then define a vector valued theta series by
\[
\left[\Theta_{L_V}(\tau,\varphi^V_{1,[k]}),  \Theta_{K_{W}}(\tau,\phi^{N}_{[k]}) \right]  =
\sum_{h\in L^{\#}/L} \left[\theta_{L_V+h}(\tau,\varphi^V_{1,[k]}),  \theta_{K_W+h}(\tau,\phi^{N}_{[k]}) \right]
 \frake_{h},
\]
which transforms like a modular form of weight $k+3/2$ with respect to the representation $\rho_L$. That is, for $(\g',\phi) \in \G'$,
\[
\left[\Theta_{L_V}(\g'\tau,\varphi^V_{1,[k]}),  \Theta_{K_W}(\g'\tau,\phi^{N}_{[k]}) \right] 
= \phi^{2k+3}(\tau) \rho_L(\g',\phi) \left[\Theta_{L_V}(\tau,\varphi^V_{1,[k]}),  \Theta_{K_W}(\tau,\phi^{N}_{[k]}) \right].
\]
Note that the theta series vanishes identically unless $k$ is odd. We want to compute the cohomological pairing
\[
\left\langle \calE(z,k), \left[\Theta_{L_V}(\tau,\varphi^V_{1,[k]}),  \Theta_{K_{W}}(\tau,\phi^{N}_{[k]}) \right] \right\rangle.
\]
We also define a vector valued Eisenstein series $\calE_{k+3/2,K}(\tau)$ of half-integral weight $k+3/2$ for the representation $\rho_{K}$ by
\[
\calE_{k+3/2,K}(\tau) = \frac12 \sum_{\g' \in \G'_{\infty} \back \G'} \phi(\tau)^{-2k-3} \rho^{-1}_{K}(\g') \mathfrak{e}_0.
\]
Here $\g'=(\g,\phi)\in \G'$ and $\G_{\infty}'$ is the inverse image of $\G_{\infty} = \{ \kzxz{1}{n}{0}{1} \}$  inside $\G'$.

\begin{remark}
Let $k$ be odd. We can view the theta series and the Eisenstein series naturally as scalar-valued forms of weight $k+3/2$ for $\G_0(4)$ satisfying the Kohnen plus space condition, which means that the $n$-th Fourier coefficient vanishes unless $n \equiv 0,1 \pmod{4}$. Namely, it is not too hard to see that the sum of the two components of the vector-valued form (evaluated at $4\tau$) $\left(\calE_{k+3/2,K}(4\tau) \right)_{0} + \left( \calE_{k+3/2,K}(4\tau) \right)_{1} $ is a Cohen Eisenstein series for $\G_0(4)$ \cite{Cohen}. The same procedure for the theta series gives $\left[\theta_{L'_V}(\tau,\varphi^V_{1,[k]}),  \theta_{K_{W}}(\tau,\phi^{N}_{[k]}) \right]$ for the lattice $L'  =\left\{\kzxz{b}{2c}{2a}{-b};\, a,b,c\in \Z\right\}$.
\end{remark}

\begin{theorem}\label{Eisensteinlift}
Let $k$ be odd. Then with the notation as above, we have
\[
\left\langle \calE(z,k), \left[\Theta_{L_V}(\tau,\varphi^V_{1,[k]}),  \Theta_{K_{W}}(\tau,\phi^{N}_{[k]}) \right] \right\rangle = \frac12 \frac{B_{k+1}}{k+1} \calE_{k+3/2,K}(\tau).
\]
\end{theorem}

\begin{remark}
The constant coefficient of the $\frake_0$-component of $\calE_{k+3/2,K}(\tau)$ is $2$. Hence the constant coefficient of the lift of $\calE(z,k)$ is $\tfrac{B_{k+1}}{k+1}$. This corresponds to the geometric interpretation of the period of $\calE(z,k)$ over boundary cycle $C^c_{0,[k]}$ given in Theorem~\ref{Main}.

\end{remark}

\begin{proof}

We use Lemma~\ref{integralformula} to compute the pairing. First note that the integral over the boundary $X_{\infty}$ does not contribute since the pairing in the coefficients of $\calE(z,k)$ and $\Theta_{K_{W}}(\tau,\phi^{N}_{[k]})$ vanishes:  $(n(x)u'k, n(x)v_{-2}) = 0$ ($u'^k$ has weight $-2k$). It remains to compute 
\[
\int_X \calE(z,k) \wedge \Theta_{L_V}(\tau,\varphi^V_{1,[k]}).
\]
We unfold the Eisenstein summation in the usual way and obtain for the $\frake_h$ component

\begin{multline*}
  \int_{\G_{\infty} \back D}  dx \otimes n(x)u'^k \wedge \Theta_{L+h}(\tau,\varphi^V_{1,[k]}) \\
= 2^{-k-3/2}v^{-k/2} \sum_{j=1}^{k+1}
\int_{\G_{\infty} \back D}  \sum_{\x \in L+h} \widetilde{H}_j(\sqrt{v}(g_z^{-1}\x)_1) \widetilde{H}_{k+1-j}(\sqrt{v}(g_z^{-1}\x)_2)   \varphi_0(\x,z,\tau)  \\ \times (n(x) u'^k, g_z(e_1^{j-1}e_2^{k+1-j})  )\frac{dxdy}{y},
\end{multline*}
where $\varphi_0(\x,z,\tau) =  \varphi_0(\sqrt{v}g_z^{-1}\x) e^{2\pi i (\x,\x)u)}$. But now $(n(x) u'^k, g_z(e_1^{j-1}e_2^{k+1-j}) =0$ unless $j=k+1$ in which case we obtain $(-1)^k 2^{-k/2} y^k$. Hence so far
\begin{multline}\label{firstapprox}
\left\langle \calE(z,k), \left[\Theta_{L+h}(\tau,\varphi^V_{1,[k]}),  \Theta_{K+h}(\tau,\phi^{N}_{[k]}) \right] \right\rangle \\
= (-1)^k 2^{-(3k+3)/2}v^{-k/2} 
\int_{\G_{\infty} \back D} 
 \sum_{\x \in L+h} \widetilde{H}_{k+1}(\sqrt{v}(g_z^{-1}\x)_1)
 \varphi_0(\x,z,\tau) y^{k-1} dxdy.
\end{multline}
We write $\x = \kzxz{b}{c}{-a}{-b}$. Then 
\[
\sqrt{v}(g_z^{-1}\x)_1 = \frac{\sqrt{v}}{\sqrt{2}y}(c+a(x^2-y^2) +2bx)
\]
and 
\[
\varphi_0(\x,z,\tau) = e^{-\pi \frac{v}{y^2}(c+a(x^2-y^2)+2bx)^2} e\left(-ac\bar{\tau}\right) e\left(2ia^2x^2v\right) e(b^2\tau).
\]
Here $e(t) = e^{2\pi i t}$. We apply Poisson summation on the summation on $c\in \Z$. This goes similarly as in \cite{BFCrelle}, section ~7 (which in turn is a special case of the considerations in \cite{Bo1}, section~5). After some tedious manipulations and using that the Fourier transform of $H_{k+1}(\sqrt{\pi} t) e^{-\pi t^2}$ is given by $(-2i\sqrt{\pi}\alpha)^{k+1} e^{-\pi \alpha^2}$ we obtain for \eqref{firstapprox} 
\begin{multline}\label{secondapprox}
-i^{k+1} 2^{-k-1}v^{-k-1}  \sum_{h\in K^{\#}/K} \biggl( \int_{\G_{\infty} \back D} \sum_{\substack{a,\alpha \in \Z \\ b \in K+h}}  (a\bar{\tau} +\alpha)^{k+1} e^{-\pi |a\bar{\tau}+\alpha|^2 y^2/v} \\
\times e\left( \tau(b+ax)^2+ 2 \alpha x(b+ \tfrac{ax}{2}) \right) y^{2k+1} dxdy \biggr) \frake_h.
\end{multline}
We define the unary theta series
\[
\Theta_{K}(\tau, \alpha,\beta) = \sum_{h \in  K^{\#} / K}
\sum_{b \in K + h} e\left( {\tau} (b+\beta)^2\right) e\left( -2\alpha(b +\beta/2)  \right) \mathfrak{e}_h,
\]
as in \cite{Bo1}, section~4. 
For \eqref{secondapprox} we then get
\begin{multline}\label{thirdapprox}
-i^{k+1} 2^{-k-1}v^{-k-1}  \int_{\G_{\infty} \back D} \sum_{n=1}^{\infty} n^{k+1} \sum_{\substack{c,d \in \Z \\ \gcd(c,d)=1}} (c\bar{\tau} +d)^{k+1} e^{-\pi n^2 |c\bar{\tau}+d|^2 y^2/v} \\
\times \Theta_{K}(\tau,-dnx ,cnx)
y^{2k+1} dxdy. 
\end{multline}
Now we complete each coprime pair $(c,d)$ to an element $\g' =  \left( \kabcd , \sqrt{c\tau+d}\right)\in \Gamma'$. By \cite{Bo1}, Theorem~4.1 we find
\[
 \Theta_{K}(\tau, -dnx,cnx) = \left(c{\tau} +d \right)^{-1/2} \rho_{K}^{-1} \left( \g' \right)\Theta_{K}(\g'\tau, -nx,0).
\]
Using this and $\int_0^1 \Theta_{K}(\g'\tau, -nx,0)dx = \frake_0$, we then easily obtain for \eqref{thirdapprox}
\begin{align*}
&-i^{k+1} 2^{-k-1}v^{-k-1} \sum_{n=1}^{\infty} n^{k+1} \frac12
\sum_{\g' \in \G'_{\infty} \back \G'}  \left(c{\tau} +d \right)^{-1/2}(c\bar{\tau} +d)^{k+1}  
 \\& \quad \times 
 \int_0^{\infty} e^{-\pi n^2 |c\bar{\tau}+d|^2 y^2/v} y^{2k+1} dy 
 \left( \rho_{K}^{-1} \left( \g' \right) \int_0^1 \Theta_{K}(\g'\tau, -nx,0)dx \right) \\
 &= -i^{k+1} 2^{-k-1} \zeta(k+1) \pi^{-k-1} \G(k+1)  \calE_{k+3/2,K}(\tau). 
\end{align*}
The theorem follows. 
\end{proof}

\section{The extension of the main theorem to the orbifold case}\label{lastsection}

In this section we will prove that if $\G'$ is a normal subgroup of the congruence subgroup $\G$ with finite index and the main theorem holds for $\G'$ then it holds for $\G$ even if $\G$ has elements of finite order, say $\G = \SL_2(\Z)$. To this end we develop in \cite{FMlocal} a theory of simplicial homology with local coefficients for orbifolds.  

In the following we abreviate the theta series element $[\theta(\varphi_{1,[k]}^V), \sum_{[\ell]} \theta(\phi_{[k]}^{N_{\ell}})]$ in the cohomology of the mapping cone complex by $\theta(\varphi,\phi,\G)$ emphasizing the level $\G$ (and omitting all the other data). We want to prove 
\begin{equation}
[\theta(\varphi,\phi,\G)] = \sum_{n=0}^{\infty} \PD(C_n^c(\G))~q^n \tag{$\ast_{\G}$},
\end{equation}
where $q = e^{2\pi i \tau}$. 

\begin{theorem}
Suppose that for some normal subgroup $\G'$ of $\G$ of finite index we have
\begin{equation}
[\theta(\varphi,\phi, \G')] = \sum_{n=0}^{\infty} \PD(C_n^c(\G'))~q^n.  
\tag{$\ast_{\G'}$}
\end{equation}
Then $(\ast_{\G})$ holds as well for a proper definition of the cycles $C_n^c(\G)$ and the (co)homology groups for $X$, see below.
\end{theorem}

The proof will depend heavily on the constructions and results from \cite{FMlocal}.

\subsection{Simplicial homology with local coefficients}
\subsubsection{Local coefficient systems over a simplicial complex}

We will follow now the presentation of \cite{JM}, page 69. However in the presentation of \cite{JM}, loc.cit., homology and cohomology were treated at the same time. This requires the arrows in the local system to be bijective (see Remark \ref{notbijective} below). We let $M$ be a triangulated space. We will abuse notation and use $M$ to denote both the simplicial complex and its underlying space. We refine the triangulation of $M$ by taking the barycentric subdivision $sd~M$. We order the vertices of each simplex $\sigma$ of $sd~M$ so that the barycenter of a larger dimensional simplex precedes the barycenter of a smaller dimensional one. Note that $\Phi$ will preserve this partial order. The point is that the boundary operators $\partial_p$ depend only on this {\it partial} order.

Because of the above partial order the $1$-skeleton is a directed graph (quiver) (where each edge is directed from its smaller vertex towards its larger vertex)and the vertices in each simplex are totally ordered. We can consider a directed graph as a category where the objects are the vertices and the morphisms are the edges.

\begin{definition} A local system $\calE$ on $M$ is a covariant functor from the directed graph to the category of finite-dimensional vector spaces (so a representation of the quiver) satisfying the following ``zero curvature condition''.

Suppose $(v_0,v_1,v_2)$ is an ordered two simplex of $M$. Then we have
\begin{equation} \tag{  zero curvature }
 T((v_0,v_2)) = T((v_1,v_2)) \circ T((v_0,v_1))
\end{equation}
Here if  $(x,y)$ is a directed edge of $M$ then $T((x,y)): \calE(x) \to \calE (y)$ denotes the associated linear transformation.
\end{definition}

\begin{remark}\label{composition}
The reader will observe that in order to prove that $\partial_{p-1} \circ \partial_p =0 $ the zero curvature condition is needed.
\end{remark}

The standard example of a local system comes from a flat vector bundle $E$ over $M$. In this case $\calE (x)$ is the fiber $E_x$ over $x$ and $T((x,y)):E_x \to E_y$ is given by parallel translation. In this case $T((x,y))$ is invertible for all $(x,y)$.  
The following remark is crtical.

\begin{remark}\label{notbijective}
The morphisms corresponding to the edges need not be either injective or surjective (unlike the case of flat bundles). In the case we need here some edge morphisms will be the projections of the generic fiber onto its coinvariants under a (finite) group action.

\begin{remark}
Note that we have changed the notation for the local system associated to the flat bundle $E$ from $\widetilde{E}$ to $\calE$. The reason for doing this is that in this paper for the orbifold case (e.g. for the full modular curve) the local system we consider is not the local system associated to a flat bundle, it is the quotient of such a local system by a finite group, see below. For our later research (and for that of others) it is important to separate the two notions, for example it will be necessary to consider local systems with values in the category of  abelian groups (not necessarily free) - that is, the functor from the directed graph consisting of the one skeleton taking values in the category of abelian groups (not necessarily free).
\end{remark}

Suppose $M$ and $N$ are simplicial complexes and $f: M\to N$ is a simplicial map. Suppose $\calE$ is a local system over $M$ and $\calF$ is a local system over $N$. Then a morphism $\widetilde{f}:\calE \to \calF$ over $f$ consists of an assignment 
for each vertex $x \in M$  of a linear map $\widetilde{f}(x):\calE (x) \to \calF (f(x))$ such that
for any directed edge $(x,y)$  we have a commutative diagram
\[
\begin{CD}
\calE(x)                    @>T((x,y))>> \calF(y) \\
@V\widetilde{f}(x)VV     @VV\widetilde{f}(y))V\\
\calE(f(x))                            @>T((f(x),f(y))>>  \calF(f(y)).
\end{CD} 
\]
\end{remark}

\subsubsection{Simplicial chains with local coefficients}
Now we define the complex of ordered simplicial chains with coefficients in $\calE$ following \cite{JM}, also see \cite{FMcoeff}.
We define a $p$-simplex $s$ with coefficients in $\calE$ to be a pair $s,v$ where $s =(x_0,\cdots,x_p)$ is an ordered $p$-simplex in $M$
and $v$ is an element in the vector space $\calE (x_0)$.  We will denote ths above simplex by $s \otimes v$ and the group of simplicial
$p$-chains with coefficients in $\calE$ by $C_p(M,\calE)$. We define the boundary operator $\partial_p$ on $p$-simplices with coefficients by
\begin{equation}\label{boundary}
\partial_p ((x_0,\cdots,x_p) \otimes v) = (x_1,\cdots,x_p) \otimes T((x_0,x_1))(v) + \sum_{i=1}^p (-1)^i (x_0, \cdots,\hat{x_i},\cdots,x_p) \otimes v.
\end{equation}
The following lemma is proved using the usual proof together with Remark \ref{composition}.
\begin{lemma}
 \[
 \partial_{p-1} \circ \partial_p =0.
 \]
\end{lemma}

\subsubsection{ The quotient of a flat bundle by a finite group action}
Now assume that we are given a flat bundle $E$ over $M$ and that $\Phi$ is a finite group acting simplicially on $M$ and also on a flat bundle $p:E \to M$ such that the action on $E$ covers the action on $M$ and preserves the connection. We let $\calE$ be the associated local system. Let $N$ be the quotient of $M$ by $\Phi$. In our paper \cite{FMlocal} we have constructed a new local system $\calE_{\Phi}$ over $N$ which is the quotient of $\calE$ by $\Phi$. Roughly speaking (see \cite{FMlocal}
for the actual construction which is independent of a choice of $x$ below) the vector space attached to a vertex $y$ of $N$ is the space of coinvariants of the generic fiber (the fiber of the original flat bundle) by the isotropy subgroup of a vertex $x$ of $M$ lying over $y$.  Now suppose
$(y_1,y_2)$ is an ordered edge of $N$ and $(x_1,x_2)$ is an ordered edge of $M$ lying over 
$(y_1,y_2)$.  We claim (see below) that for the case of quotients of the upper half plane we may choose the orderings of simplices so that for any ordered edge as above we have an inclusion of isotropy subgroups
$$\Phi_{y_1} \subset \Phi{y_2}$$
and hence parallel translation along $(x_1,x_2)$ from $E_{x_1}$ to $E_{x_2}$ induces the required morphism
$$T((y_1,y_2)): \calE(y_1) \to \calE(y_2).$$  It is very important in what follows that 
{\it the morphism $T((y_1,y_2))$ can have a nonzero kernel}. 

The map $\pi:M\to N$ induces a morphism $\widetilde{\pi}: \calE \to \calE_{\Phi}$ over $\pi$. The morphism $\widetilde{\pi}$ induces a morphism to be denoted $\pi_*$ from the chain complex of simplical chains $C_{\bullet}(M,\calE)$ to the chain complex of simplicial chains $C_{\bullet}(N, \calE_{\Phi})$. In \cite{FMlocal} we prove

\begin{proposition}\label{homology}
The morphism of chain complexes $\pi_{\ast}: C_{\bullet}(M,\calE) \to C_{\bullet}(N, \calE_{\Phi})$ induces an isomorphism of chain complexes $\pi_{\ast}: C_{\bullet}(M, \calE)_{\Phi} \to C_{\bullet}(N, \calE_{\Phi})$. Here $C_{\bullet}(M, \calE)_{\Phi}$ denotes the coinvariants.  
\end{proposition}

\subsubsection*{Special cycles with local coefficients in orbifold quotients of the upper half plane}

Now that we have developed the requisite theory of local coefficient systems we return to the modular curve and its finite covers. We have defined $\G, \G'$ and let $\Phi$ the quotient of $\G$ by $\G'$. We assume $\Phi$ has order $m$. We obtain the regular branched covering $\pi:X' \to X$. We assume that we have triangulated $X'$ so that $\Phi$ acts by simplicial automorphisms. Furthermore, {\it we
assume that all fixed points of $\Phi$ acting on $X'$ are vertices of $X'$}. We then refine the triangulation of $X'$ by taking the barycentric subdivision $sd~X'$.  The vertices $x$ of a simplex
in the barycentric subdivion of $X'$ correspond to a simplex $s$ in the orginal triangulation
so $x = x_s$ and an two vertices $x_s,x_{s'}$ form an unordered edge if and only if either $s \subset s'$ or $s' \subset s$. 
We order the vertices of edges (and a fortiori of two simplices) so that a bigger simplex corresponds to a smaller vertex in the order i.e. 
$$ x_{s'} < x_s   \iff s \subset s'.$$
Thus by construction if $(x_1,x_2)$ is an ordered edge then the isotropy of $x_1$ is trivial and
the above claim is proved.

 As before, we let $E_{2k} = \calH_k(V)$ be the irreducible representation of $\SL_2$ of highest weight $2k$. We obtain a flat bundle $E=E_{2k}$ on $X'$ given by $E= \h \times _{\G'} E_{2k}$. The bundle $E$ gives rise to the associated locally constant local system $\calE$ on $X'$ and hence by the theory of the previous subsection to the local system $\calE_{\Phi}$ on $X$.

Proposition~\ref{homology} then gives

\begin{corollary}\label{actualhomology}
The homology of the complex $C_{\bullet}(X, \calE)_{\Phi}$ is isomorphic to the group homology
$H_{\bullet}(\G ,E_{2k})$. 
\end{corollary}
\begin{proof}
Since taking coinvariants by a finite group is an exact functor on the category of finite-  dimensional vector spaces it commutes with homology and we have
\[
H_{\bullet}(X,\calE_{\Phi}) \cong H_{\bullet}(X',\calE)_{\Phi}.
\]
But since $X'$ is a space of type $K(\G',1)$ we have $H_{\bullet}(\G' ,E_{2k}) \cong H_{\bullet}(X',\calE)$. But $H_{\bullet}(\G ,E_{2k})\cong H_{\bullet}(\G' ,E_{2k})_{\Phi}$.
\end{proof}

We have already defined special cycles $ C'_{\x}$ in $X'$. They are either closed geodesics (possibly with self-intersections) or infinite geodesics joining two cusps - i.e. modular symbols. 
We will use the term ``special cycle'' to refer to either.
In what follows for the rest of this subsection it will be important to keep track of the coefficient 
$v$ and to write $C'_{\bf x} \otimes v$.

In the case that $ C'_{\bf x}$ is a closed geodesic we use only the coefficents $\pi_k(\x^{k})$. 
Suppose now that $ C'_\x$ is a modular symbol. Then the subgroup $G_\x \cong \R$  is the  
real points of a maximal torus $T$ defined over $\Q$. As before the torus $T$ gives rise to a weight deomposition of $E_{2k}$ defined over $\Q$
\[
E_{2k} = \bigoplus _{\ell= -2k}^{2k} E_{2k}(\ell),
\]
where $E_{2k}(\ell)$ denotes the $1$-dimensional weight space with weight $\ell$. 
We will now define the  special cycles $C_{\x} \otimes v$. We will separate into two cases, the case where  $C'_{\x}$ is noncompact ( modular symbols)
and the case where $C'_{\x}$ is compact (closed geodesics with transverse self-intersections).
The first case breaks up into two subcases, the ``unfolded'' case where the image of $C'_{\x}$ is
an infinite geodesic joining joining two distinct cusps and the ``folded'' case where of $C'_{\x}$
is a geodesic ray from an orbifold point with the label two (the image of a point with isotropy
$\Z/2$) to a point in the boundary $\partial X'$. 
The second case again breaks up  into two subcases,  the ``unfolded '' case where the image of 
$C'_{\x}$ is a closed geodesic with transverse intersections and the ``folded'' case where
the image of $C'_\x$ is the geodesic segment $\overline{xy}$ joining two orbifold singular points $x$ and $y$ each with label $2$. 

\bigskip

\paragraph{\bf Folded modular symbols}
Suppose the (rational) geodesic  $D_{\bf x}$ joins two cusps $c_1$ and $c_2$. Since $D_{\x}$
is oriented (by $\x$ as explained above) we may distinguish between the two cusps. We assume that
the oriented geodesic is directed from $c_1$ to $c_2$. 
In case $D_{\bf x}$
maps  into the quotient with transverse multiple points then it is a usual modular symbol.  We will now analyse the new phenomenon cause by the existence of elements of finite order in $\G$.  The problem is there may now
be elements in $\G$ that carry $D_{\bf x}$ into itself. As we will see below, the subgroup of 
$\G$ stabilizing $D_{\bf x}$ is either trivial or has order two.  In the first  case we  will get a
usual modular symbol (joining two distinct cusps), in the second case we will get   
 a ``folded modular symbol''. We now analyse this second case. Suppose  then is an element $\iota \in \G$ of order two that carries $D_\x$ into itself and reverses that cusps $c_1$ and $c_2$ and hence reverses the orientation of $D_{\bf x}$.  Clearly $\iota$ has a (unique) fixed point $z_0$ on $D_{\bf x}$ and rotates the tangent space $T_{z_0}(D)$ by the angle $\pi$ around this fixed point. 
Hence, the subgroup  $\{1,\iota \}$ of $\G$ is a subgroup of $\G$ that stabilizes the set $D_{\bf x}$ and hence is 
the full subgroup of $\G$ that carries $D_\x$ into itself.  Clearly the image of $D_\x$ in $X$
is the geodesic ray $C^+_\x$ joining $z_0$ to the point in the Borel-Serre boundary corresponding to
the geodesic ray from $z_0$ to $c_2$. 

\medskip

To analyse the relative cycles with coefficients corresponding to  this second case it suffices to analyse the case of the folded $y$-axis for the
case in which $\G = PSL(2,\Z)$.  We will defer the detailed analysis as a simplicial cycle with
local coefficients to the remark following and explain the key point in terms of the ray  $C^+_\x=(i,\infty)$ informally first. 
The  image $C_{\bf x}$ of
$D_{\bf x}$ in the modular curve is ``folded'', as a set it is a closed half line but as a simplicial chain each simplex has coefficient zero so we get the zero $1$-chain.
However if we have coefficient $v$ we have
\[
 C_{\x} \otimes v =  (0,i) \otimes v + (i,\infty) \otimes v.
 \]
But 
\[
 (0,i) \otimes v  \sim \iota( (0,i) \otimes v )= (\infty,i) \otimes \iota(v) = -(i,\infty) \otimes \iota(v).
 \]
Combining the two equations above we have
\[
C_{\x} \otimes v = (i,\infty) \otimes (v - \iota(v)).
\]
In general we let $D^+_{\bf x}$ denote the positively directed half-line from $z_0$ to the cusp at its end. The half-line $D^+_{\bf x}$ projects
one-to-one to a half line in $X$ and we have 
\[
C_{\bf x}\otimes v = C^+_{\bf x} \otimes (v - \iota(v)).
\]
In case the coefficients are trivial we interpret $v - \iota(v)$ as zero.  Unless the weight $\ell$ of $v$ is zero the coefficient $v - \iota(v)$ will be nonzero
but will not be a weight vector.  The point is that $\iota E_{2k}(\ell) = E_{2k}(-\ell)$.

\begin{definition}
We will call a relative cycle of the form $C^+_{\bf x} \otimes (v - \iota(v))$ a {\it folded} modular symbol (with coefficients) and denote it $C_{\bf x}\otimes v$. 
\end{definition}

In the following remark we refine the above discussion and give a careful description of the
above ``folded modular symbol'' as a  simplicial one-chain with coefficients which
is a cycle relative to the Borel-Serre boundary. 

\begin{remark}\label{foldedms}
The folded modular symbols are cycles relative to the Borel-Serre boundary of $X$ with coefficients in the system $\calE_{\Phi}$. We know since they are images of relative cycles in $X$ under the chain map $\pi_{\ast}$  that they have to be 
relative cycles but we will verify now this directly for the special case of the $y$-axis (the general case is no harder, just replace $i$ by $z_0$).  
In the Borel-Serre compactification the half-line $(i,\infty )$ has a boundary point $c$ and we have a closed (geodesic) interval joining $i$ and $c$
that we denote $[i,c]$. We choose an interior point $z$ to $[i,c]$ and form the two ordered one simplices $(z,i)$ and $(z,c)$, the ``barycentric subdivision'' of $[i,c]$. Recall the first vertex of any one-simplex in $sd X$ is the barycenter of the original edge ($(i,c)$) so
$z$ has to come first in each of the two above one simplices.
We then define (here $\x= e_2 = \kzxz{1}{0}{0}{-1} \in V$)
\[
 C_{\x}^+ \otimes (v -\iota(v)) := (z,i) \otimes (v -\iota(v)) + (z,c) \otimes (\iota(v) -v)).
 \]
Accordingly we have 
\begin{multline*}
\partial \left((z,i) \otimes (v -\iota(v)) + (z,c) \otimes (\iota(v)-v) \right)\\
 =  (i) \otimes T((z,i))(v -\iota(v)) - (z) \otimes (v - \iota(v)) \\
+  (c)\otimes T((z,c))( \iota(v)-v) -(z) \otimes (\iota(v)-v).
\end{multline*}
But $v-\iota(v)$ is zero in the $\Phi_i$ -covariants of the generic fiber and $T((z,c)) = I$ hence

\[
T((z_0,i))(v -\iota(v)) =0
\]
and
\[
\partial \left((z,i) \otimes (v -\iota(v)) + (z,c) \otimes (\iota(v)-v)\right) =  (c) \otimes (v - \iota(v)).
\]
Hence we have proved by a direct calculation in the complex $C_{\bullet}(X, \calE_{\Phi})$ that $C_{\bf x} \otimes (v -\iota(v))$ is a relative cycle.
\end{remark}

\begin{remark}
In case $v =\pi_k(\x^k)$ we have 

$$ C_{\x}^+ \otimes (v -\iota(v)) = 2 C_{\x}^+ \otimes v.$$
\end{remark}

Recalling that in all cases we have defined
$$C_\x \otimes v := C_{\x}^+ \otimes (v -\iota(v))$$
we have an equality of cycles with coefficients
\begin{equation}\label{imagems}
\pi \cdot C'_\x \otimes v = C_\x \otimes v.
\end{equation}

\paragraph{{\bf The stabilizer of the modular symbol $ C_{\x}^{\prime} \otimes v$}}

We will now need an analysis of the subgroup of $\Phi$ that stabilizes the cycle $ C_{\x}^{\prime} \otimes v $.  We may assume (by our choice of $\G'$) that  $C_{\x}^{\prime}$ is an infinite geodesic joining
two cusps. Thus the stabilizer of the {\it oriented} geodesic $D_{\x}$ in $\SL(2,\Q)$ is a rational torus $\mathbb{T}$ which is {\it split} over $\Q$.

 The stabilizer of the underlying unoriented geodesic is the normalizer of the torus $N(\mathbb{T})$ and we have an extension $\mathbb{T} \to N(\mathbb{T}) \to \Z/2$, which we may split by assigning the element $\iota = \iota_{\x}$ of order $4$ in $N(\mathbb{T})$ which for the case of the $y$-axis is the matrix $\kzxz{0}{-1}{1}{0}$. We note that $\iota$ exchanges the two cusps which are the ends of the geodesic. We will assume that $\iota \in \G$ otherwise $\iota = I$ in the lemma below.  We abuse notation and let $\iota$ denote the element of $\Phi$ represented by $\iota$. 

We first observe
\begin{lemma}
\[
\eta \cdot  C_{\x} \otimes v =  C_{\x} \otimes v \Longrightarrow   \ \text{either} \ \eta = \iota \ \text{or} \ \eta = I.
\]
\end{lemma}

\begin{proof}
$\eta \cdot  C_{\x} \otimes v =  C_{\x} \otimes v$ implies that we have an equality of submanifolds $\eta(C_{\x}^{\prime}) = C_{\x}^{\prime}$. But this implies that $\eta$ has a representative $\tilde{\eta}$ satisfying $\tilde{\eta}(D_\x) = D_{\x}$. This in turn implies that $\tilde{\eta} \in N(\mathbb{T})$.  But $N(\mathbb{T}) \cap \SL(2,\Z) = \{ I, \iota \}$ since $\mathbb{T}$ is split over $\Q$. 
\end{proof}

We then have

\begin{proposition}\label{msstabilizer}
The stabilizer of the modular symbol $ C_{\x}^{\prime} \otimes v$ is trivial unless $v$ has weight zero for $\mathbb{T}$. If $v$ has weight zero then the stabilizer is $\{I,\iota\} \cong \Z/2$ if $k$ is odd. Otherwise, it is trivial.  
\end{proposition}

\begin{proof}

By the lemma either the isotropy group of the cycle is trivial or $\Z/2$.  Let us examine when the later holds that is $\iota$ is in the isotropy subgroup. 
Now we have 
\[
\iota \cdot  C_{\x} \otimes v = \tilde{\pi} \cdot (\iota \cdot  D_{\x}) \otimes \iota(v).
\]
Here we think of $D_\x$ as an oriented  subcomplex of $D$. But $\iota \cdot  D_{\x} = -  D_{\x} $ whence we have $\iota(v) = -v$. Clearly this occurs if and only if $v$ is the zero weight vector $\pi_k(\x^k)$ and $k$ is odd.
\end{proof}

We now deduce a corolllary of the proposition.
\begin{corollary}
For the case in which $C_{\x}$ is not compact the stabilizer of $\x$ in $\PSL(2,\Z)$ is trivial.
\end{corollary}
\begin{proof}
Clearly if $\gamma(\x) = \x$ then $\gamma(D_{\x}) = D_{\x}$. But we have seen that this latter equation holds if and only if $\gamma \in \{I,\iota\}$.
But we claim $\iota(\x) = - \x$. Indeed $\iota$ preserves the orientation of $D$ and reverses the orientation of $D_{\x}$. Hence it must also reverse the orientation of 
the normal direction $\x$ to $D_{\x}$.
\end{proof}

\paragraph{\bf Folded geodesics and closed geodesics  with transverse self-intersection} 
We will now assume that $C'_{\x}$ is compact. 
Now the torus $\mathbb{T} \cap \G $ is defined over $\Q$ {\it but not split}.  There are two cases
depending whether of not the intersection $N(\mathbb{T}) \cap \Gamma $ is an infinite cyclic group 
or an infinite dihedral group i.e. whether or not it contains an element $\iota$ of order two. We will
call the first case the ``unfolded case'' and the second the ``folded case''. 
The intersection will always contain an infinite cyclic group which is the subgroup
that fixes both ends of the infinite geodesic $D_{\x}$.  We will let $\gamma_\x$ be the
generator of this cyclic group which preserves the orientation of $D_{\x}$ (since we have oriented $D_{\x}$ this generator is well-defined. Note that since we have assumed $\G'$ is torsion free we have
$N(\mathbb{T}) \cap \Gamma' $ is infinite cyclic generated by a power of $\gamma_x$. Thus in both
cases the cycle $C'_\x$ will be a closed geodesic with at worst transverse self-intersections. The vector
$v =\pi_k(\x^k)$ will be invariant under $\gamma_x$ and we obtain a cycle with
coefficients $C'_{\x} \otimes v$ in $X'$ in both cases.  We let  $\eta_\x$
denote the element of $\Phi$ induced by $\gamma_\x$. 

\begin{definition} \label{coveringdegree}
For both cases we define the number $d_{\x}$ to be the order of the cyclic subgroup of $\Phi$
generated by $\eta_\x$ above. 
\end{definition} 

In the unfolded case the space $C_\x$ which is the direct image of $D_{\x}$ under the branched covering
$D \to X$ will be a closed geodesic with possible transverse self-intersections and we have
an induced covering $\pi:C'_\x \to C_\x$ of degree $d_\x$.  In this case the vector  $v =\pi_k(\x^k)$
induces a parallel section over $C_\x$ and we may define the cycle with coefficients
$C_\x \otimes v$ as usual. Note that we have

$$ \pi \cdot C'_\x \otimes v = d_\x C_\x \otimes v.$$

\paragraph{\bf Folded geodesic cycles with coefficients}
In the ``folded case '' we can no longer give the straight-forward definition 
of the {\it cycle} $C_\x \otimes v$ because the image of $D_{\x}$ in $X'$ is now a geodesic
segment.  It will take considerable effort to arrive at the correct definition of $C_\x \otimes v$.

By a folded 
geodesic we mean a geodesic segment $\overline{y_1y_2}$ joining two orbifold singular points 
$y_1$ and $y_2$ each  with label $2$.  This means that any inverse image $x_1$ of $y_1$ resp. $x_2$
of $y_2$ in $D$ is fixed by an element $\iota_1$, resp. $\iota_2$ of $\G$ of order $2$.  Let
$\alpha(x_1,x_2)$ be the infinite geodesic in $D$ joining $x_1$ and $x_2$.  Then the subgroup $\G(1,2)$
of $\G$ generated by $\iota_1$ and $\iota_2$ is an infinite dihedral group that acts on
the geodesic $\alpha(x_1,x_2)$ and with image the geodesic segment $\overline{y_1y_2}$.  This is
because a fundamental domain for $\G(1,2)$ acting on $\alpha(1,2)$ is the geodesic segment
$\overline{x_1x_2}$. We have  
\begin{lemma}
The orbifold $X$ contains a folded geodesic if and only if it has two distinct orbifold points with label $2$.
\end{lemma}
\begin{proof}
The condition is obviously necessary. But if $X$ has two orbifold points we simply join
them in $X$  by a shortest geodesic segment so the condition is sufficent as well.
\end{proof}

\begin{example}
The famous ``theta group'' consisting of matrices \[ \begin{pmatrix} a & b \\
                                                                     c & d \\
\end{pmatrix}\] with $ac$ and $bd$ even has two distinct orbifold points with label $2$ namely $i$ and $i+1$ and hence contains the folded geodesic $\overline{i i+1}$.    It is interesting that the modular curve itself does not contain a folded geodesic.  
\end{example}

Let  $\overline{y_1y_2}$ be a folded geodesic and $v =\pi_k(\x^k)$.  We will now define
a simplicial one chain with coefficients $\overline{y_1y_2} \otimes v$. 
 We let $z$ be the midpoint of  the geodesic
segment $\overline{y_1y_2}$.  We define the one chain with coefficients $\overline{y_1y_2} \otimes v$ by
$$ \overline{y_1y_2} \otimes v = (z,y_1) \otimes v - (z,y_2) \otimes v.$$

\begin{lemma}
The chain $\overline{y_1y_2} \otimes v$ is a cycle if and only if $k$ is odd.
\end{lemma}
\begin{proof}
Note
first that although the weight space splitting 
of $E_{2k}$ is not invariant under $\mathbb{T}$ the zero weight space is carried into itself
(and fixed)  and the Weyl group $N(\mathbb{T})/\mathbb{T} \cong \Z/2$ acts on the zero weight space 
by $(-1)^k$.  Thus the vector $v$ is
zero in the coinvariants of both $\iota_1$ and $\iota_2$ if $k$ is odd and nonzero in both spaces
of coinvariants if $k$ is even (we note that since $\iota_1$ and $\iota$ are conjugate by an
element of $\mathbb{T}(\R)$ and $\mathbb{T}(\R)$ fixes $v$ they both have to act the same way
on $v$). 
\end{proof}

Next we note the relation of  the chain $\overline{y_1y_2} \otimes v$ with 
the
direct image $\pi \cdot C'_{\x} \otimes v$. We leave the proof to the reader
\begin{lemma}
\hfill

\begin{enumerate}
\item $\pi \cdot C'_{\x} \otimes v = 2 d_\x \ \overline{y_1y_2} \otimes v$ if  $k$ is odd.
\item $\pi \cdot C'_{\x} \otimes v = 0\  \overline{y_1y_2} \otimes v$ if  $k$ is odd.
\end{enumerate}
\end{lemma}

We now have
\begin{definition}
$$C_\x \otimes v := \begin{cases} & 2 \   \overline{y_1y_2} \otimes v  \ \text{if $k$ is odd}\\
& 0 \ \text{if $k$ is even}.
\end{cases} $$
\end{definition}

Thus with the above definition of $C_\x \otimes v$ in the folded case we have {\it both
the folded and unfolded cases}

\begin{equation}\label{imagecs}
\pi \cdot C'_\x \otimes v = d_\x C_\x \otimes v.
\end{equation}

We now introduce some terminology which we will use in the rest of the paper.

\begin{remark}
We will often abuse terminology and refer to $C_\x$ as a ``folded geodesic'' in the second case.
In a certain sense $C_\x$  {\it is} the folded geodesic, it goes back and forth across
the segment $\overline{y_1y_2}$.
\end{remark}

\paragraph{{\bf The stabilizer subgroups associated to closed and folded geodesics}}

We will summarize the facts we will need later about the stabilizers of $C'_\x \otimes v$ for
the case that $C'_\x$ is a closed geodesic in the next proposition.  We will leave its proof to the reader.  
 
\begin{proposition}\label{scstabilizer}
Let $\Phi_{C'_{\x} \otimes v}$ be the stabilizer of the cycle $C'_{\x} \otimes v$ in $\Phi$. Then we 
have  
\begin{enumerate}
\item Either $\Phi_{C'_{\x} \otimes v}$ is the finite cyclic group  $\Z/ d_\x \Z$, the deck group of the covering $C'_{\x} \to C_{\x}$ or
\item  a finite dihedral group which is the extension
by $\Z/2$ of the above finite cyclic group.
\end{enumerate}
In the second case we let $\iota \in \G $ be an element whose image modulo 
$\G'$ is  nontrivial element  of the $\Z/2$. 
In the first case we  have
$$\Phi_{C'_{\x} \otimes v} = \Phi_{\x}$$
and in the second we have
$$\Phi_{C'_{\x} \otimes v}\supset \Phi_{\x} \cong \Z/ d_\x \Z \  \text{and} \ \Phi_{C'_{\x} \otimes v}/ \Phi_{\x} 
\cong \Z /2.$$

In both cases we have 
$$d_\x = |\Phi_{\x}|.$$

\end{proposition}

\bigskip

\paragraph{\bf Horocircles $X_{\ell}^{\prime}$}
We now analyse the case of the horocircle cycles with coefficients $X_{\ell}^{\prime} \otimes u_{\ell}^k$ where $\ell = \Q u_{\ell}$ is an isotropic rational line in $V$.

\begin{proposition}\label{hstabilizer}
The stablizer of the cycle $X_{\ell}^{\prime} \otimes u_{\ell}^k$ in $\Phi$ is a finite cyclic group $\Z/ d \Z$, the deck group of the covering 
$X_{\ell}^{\prime}\to X_{\ell}$ and coincides with the image of the stablizer $\G_{\ell}$ of $\ell$ in $\G$ 
modulo $\G'$. This latter stabilizer $\G_{\ell}$ equals $N_{\ell} \cap \G $,  is infinite cyclic and coincides with the fundamental group of $X_{\ell}$.
\end{proposition}
We leave the proof to the reader. We will often replace $d$ by $d_{\ell}$ to emphasize the dependence on $\ell$.

\subsubsection{The direct image formula, the homology transfer and the Gysin homomorphism}
We first need to compute $\pi \cdot  C'_{\x} \otimes v $. The next  proposition follows from the analysis immediately above.
Recall that if $C'_\x$ is compact then $d_\x$ is the order of the covering $C'_\x \to C_\x$
(and this is equal to the order of the isotropy group $\Phi_\x$). 
We will define $d_{\x}$ to be $1$ if $C'_{\x}$ is not compact. 
We have from equations \eqref{imagems} and \eqref{imagecs}

\begin{proposition}\label{directimage}
\hfill

\begin{enumerate}
 \item Suppose $C_{\x}$ is not compact, then we have
\[
\pi\cdot C'_{\x} \otimes v = C_{\x} \otimes v.
\]

\item  If $C_{\x}$ is compact then
 \[
 \pi\cdot C'_{\x} \otimes v = d_{\x} C_{\x} \otimes v.
 \]
\item \[ \pi \cdot X'_{\ell} \otimes u^k = d_{\ell} X_{\ell} \otimes u^k 
\]
\end{enumerate}

\end{proposition}

 We define the reducible cycles $ \tr^{\G}_{\G'}(C_{\x} \otimes v)$, the transfer of the cycle 
$C_{\x} \otimes v$  and $ \tr^{\G}_{\G'}(C_{\ell})$, the transfer of the cycle $C_{\ell} \otimes v$ by 

\begin{definition}\label{uncappedhomologytransfer}
\[
 \tr^{\G}_{\G'}(C_{\x} \otimes v) = \frac{1}{d_{\x}}\sum_{\eta \in  \Phi} \eta \cdot  
C_{\x}^{\prime} \otimes v
\]
and 
\[
 \tr^{\G}_{\G'}(X_{\ell} \otimes v) = \frac{1}{d_{\ell}}\sum_{\eta \in  \Phi} \eta \cdot  X_{\ell}^{\prime} \otimes v.
\]

\end{definition}

\begin{remark} In this remark we  justify using the symbol $ \tr^{\G}_{\G'}(C_{\x} \otimes v)$ for the above sum. $ \tr^{\G}_{\G'}(C_{\x} \otimes v)$ is the sum in the oriented simplicial chain complex for $X'$ of the oriented simplicial cycles $\eta \cdot  C_{\x}^{\prime} \otimes v$.  We have
\[
\pi \cdot \eta \cdot  C_{\x}^{\prime} \otimes v =  C_{\x} \otimes v.
\]
So all the oriented simplicial cycles $\eta \cdot  C_{\x}^{\prime} \otimes v$ are in the inverse image of the simplicial cycle $ C_{\x} \otimes v$ and moreover comprise the full inverse image (with multiplicity). We then add these cycles in the inverse image to obtain $ \tr^{\G}_{\G'}(C_{\x} \otimes v)$. Thus $\tr^{\G}_{\G'}$ is the operation from collections of simplicial chains to the group of simplicial chains described in \cite{Brown},  page 82, paragraph (E), (except Brown used cellular chains and assumes a free action). 
\end{remark}

 We have
\begin{lemma}\label{pushforward}
\hfill

\begin{enumerate}
\item$\pi \cdot  \tr^{\G}_{\G'}(C_{\x} \otimes v) =  m  C_{\x} \otimes v $.
\item $\pi \cdot  \tr^{\G}_{\G'}(X_{\ell} \otimes v) =  m  X_{\ell} \otimes v $.
\end{enumerate}
\end{lemma}
\begin{proof}
We will prove only the first formula.
\begin{align*}
\pi \cdot  \tr^{\G}_{\G'}(C_{\x} \otimes v) = &\pi \cdot \frac{1}{d_{\x}}\sum_{\eta \in\Phi} \eta \cdot  C_{\x}^{\prime} \otimes v =  \frac{1}{d_{\x}} \sum_{\eta \in\Phi}(\pi \circ \eta) \cdot   C_{\x}^{\prime} \otimes v \\
= &\frac{1}{d_{\x}} \sum_{\eta \in\Phi} \pi \cdot   C_{\x}^{\prime} \otimes v = \frac{1}{d_{\x}} \sum_{\eta \in\Phi} d_{\x}  C_{\x} \otimes v =  m   C_{\x} \otimes v. \qedhere
\end{align*}
\end{proof}
The following proposition is in \cite{FMlocal} but it is so important we give it again here.

\begin{proposition}\label{dualofinverseimage}
With the above notation we have
\begin{enumerate}
\item $\PD ( \tr^{\G}_{\G'}(C_{\x} \otimes v)) = \pi^{\ast} \PD(C_{\x} \otimes v) \quad  \text{or equivalently} \quad \frac{1}{d_{\x}} \sum_{\eta \in  \Phi}  \PD(\eta \cdot C_{\x}^{\prime} \otimes v) =  \pi^{\ast} \PD(C_{\x} \otimes v)$.

\item $\PD ( \tr^{\G}_{\G'}(X_{\ell} \otimes v)) = \pi^{\ast} \PD(X_{\ell} \otimes v) \ \ \text{or equivalently} \ \frac{1}{d_{\ell}} \sum_{\eta \in  \Phi}  \PD(\eta \cdot X_{\ell}^{\prime} \otimes v) =  \pi^{\ast} \PD(X_{\ell} \otimes v)$.
\end{enumerate}
\end{proposition}
\begin{proof}
We will prove the second formula in (1). Note first the formula for any cycle $C$ in an oriented compact manifold (possibly with boundary)
\begin{equation} \label{upperandlowerstars}
PD(\eta \cdot C)= (\eta^{-1})^{\ast} PD ( C )
\end{equation}
We may write the left-hand side as $\pi^{\ast} \psi$ for some form $\psi$ on $X$.
Let $\alpha$ be a closed form on $X$ of the same degree as the dimension of $C_{\x}$. Then using \eqref{upperandlowerstars} we have
\begin{align*}
\int_X \alpha \wedge \psi & = \frac{1}{m}  \int_{X'} \pi^{\ast}\alpha \wedge \pi^{\ast} \psi =  \frac{1}{m} \frac{1}{d_{\x}} \sum_{\eta \in  \Phi} \int_{X'} \pi^{\ast}\alpha \wedge 
( \eta^{-1})^{\ast} \PD( C_{\x}^{\prime} \otimes v)\\
& = \frac{1}{m} \frac{1}{d_{\x}} \sum_{\eta \in \Phi}\int_{X'} \eta^{\ast} \pi^{\ast}\alpha \wedge  \PD (C_{\x}^{\prime} \otimes v)= \frac{1}{m}\frac{1}{d_{\x}} \sum_{\eta \in  \Phi}\int_{X'} \pi^{\ast} \alpha \wedge  \PD (C_{\x}^{\prime} \otimes v)\\
& = \frac{1}{m}(m) \frac{1}{d_{\x}}\int_{C'_{\x} \otimes v} \pi^{\ast} \alpha = 
\frac{1}{m}(m) \frac{1}{d_{\x}}\int_{
\pi_{\ast}(C'_{\x} \otimes v)} \alpha. 
\end{align*}
But by Equations \eqref{imagems} and \eqref{imagecs} we have
$$ \pi_{\ast}(C'_{\x} \otimes v) = d_\x C_\x \otimes v $$
and consequently 
$$\int_X \alpha \wedge \psi = \frac{1}{m}(m) \frac{1}{d_{\x}} d_\x \int_{C_\x \otimes v} \alpha
 = \int_{C_\x \otimes v} \alpha. $$
Hence $\psi = PD(C_\x \otimes v)$ and since $\pi^{\ast}$ is injective the proposition is proved.

\end{proof}

\begin{remark}\label{transfer} We may rewrite the formula in the Proposition as
\[
 \tr^{\G}_{\G'}(C_{\x} \otimes v) = \PD^{-1} (\pi^{\ast} \PD(C_{\x} \otimes v)).
 \]
The right-hand side is by definition the Gysin homomorphism usually denoted $\pi^{!}$. Thus we are proving that for the finite cover case the left-hand side computes this Gysin homomophism.
\end{remark}

Accordingly we have proved the following equation that will be of critical importance to us.
\begin{equation}\label{eqhomologytransfer}
\pi^{!}( C_{\x} \otimes v) = \tr^{\G}_{\G'}(C_{\x} \otimes v).
\end{equation}

The following lemma is in \cite{FMlocal}.

\begin{lemma}\label{upstairsdownstairs}
Let $A$ and $B$ be $1$-cycles with local coefficients in $X$. Then
 \[
 \pi^{!}A\bullet \pi^{!}B = m A \bullet B.
 \]
\end{lemma}

\subsection{The case of capped special cycles}
In this section we  extend the transfer and its relation to the Gysin homomorphism to our capped cycles. We have previously defined the capped (spectacle) cycles $C_{\x}^c \otimes v$ and $( C_{\x}^{\prime}) ^c \otimes v$. In order to simplify notation in the case that $C_{\x} \otimes v$ is 
 already compact we will interpret the symbol $C_{\x}^c \otimes v$ as $C_{\x} \otimes v$.
From now on we will abbreviate $C'_{\x} \otimes v$ and $C_{\x} \otimes v$ to $C'_\x$ and $C_\x$
and the same for their capped analogues.

We define the capped reducible cycle
\begin{definition}\label{cappedhomologytransfer}
\[
(\tr^{\G}_{\G'} C_{\x} )^c = \frac{1}{d_{\x}}\sum_{\eta \in \Phi} \eta \cdot  ( C_{\x}^{\prime})^c .
\]
\end{definition}

We will see that to extend our previous work for the uncapped cycles we need only two properties  of the capped cycles {\it with their coefficients} expressing how capping commutes with deck transformations and covering projections acting on cycles. We state these in the following lemma.
Both are true is because the normalized caps with their coefficients are naturally determined by the infinite geodesic with its coefficient according to {\it geodesic} $\to$ {\it boundary point} $\to$ {\it Borel-Serre boundary component} and the fact that the coefficient on the geodesic {\it uniquely determines the coefficient on the normalized cap}.

\begin{lemma}
\hfill

\begin{enumerate}
\item $(\pi \cdot  C_{\x}^{\prime})^c = \pi \cdot ( C_{\x}^{\prime}) ^c$. \\
\item $(\eta \cdot  C_{\eta(\x)}^{\prime} )^c = \eta \cdot  (C_{\x}^{\prime})^c  $.\\
\item $(\tr^{\G}_{\G'}C_{\x})^c =\tr^{\G}_{\G'}(C_{\x}^c) $.
\end{enumerate}
\end{lemma}

We also observe the following 
\[
\eta \cdot   (C_{\x}^{\prime})^c =  (C_{\eta(\x)}^{\prime}) ^c.
\]

For any closed form $\alpha$ extending over the Borel-Serre boundary of $X$, the pull-back $\pi^{\ast}\alpha:=\widetilde{\alpha}$ extends over the boundary of $X'$ and by definition we have the analogue of Proposition \ref{directimage}

\begin{equation}\label{pushforward1}
 \int_{ (C_{\x}^{\prime})^c} \widetilde{\alpha} = d_{\x} \int_{ C_{\x}  ^c} \alpha.
\end{equation}
We again have
 
\begin{lemma}\label{cappushforward}
\[
  \pi \cdot \pi^{!}( C_{\x} ^c) = \pi \cdot  \tr^{\G}_{\G'}(C_{\x}^c) = m  C_{\x} ^c .
\]
\end{lemma}

As before we have as a consequence  
\begin{equation}\label{fullpushforward}
\int_{ \tr^{\G}_{\G'}(C_{\x}^c)} \widetilde{\alpha} = m \int_{ C_{\x} ^c} \alpha,
\end{equation}
and the capped version of the critical  Proposition \ref{dualofinverseimage}

\begin{lemma} \label{dualofcappedinverseimage}
\[
 \pi^{!}(C_{\x}^c) = \tr^{\G}_{\G'}(C_{\x}^c)
\]
\end{lemma}

Recall also that from Proposition \ref{dualofinverseimage} we have
\begin{equation}\label{transferofhorocycle}
\pi^{!}(X_{\ell}) = \tr^{\G}_{\G'}(X_{\ell})
\end{equation}

\subsection{Generalizing our Main Theorem to noncompact orbifolds}
We will now prove that if the  main theorem is true for neat torsion free congruence subgroups then it extends to all arithmetic subgroups.

Recall that the capped decomposable cycle $C_n^c,n>0,$ is given by 
\[
C_n^c=C_n^c(\G) = \sum_{\x \in \mathcal{S}_n(\G)} C_\x^c ,
\]
where $\mathcal{S}_n(\G)=\mathcal{S}_n(\G,\mathcal{L})$ denote a set of $\G$ orbit representatives of vectors $\x \in \mathcal{L}$ satisfying $q(\x)=n$.   
We also recall that the cycle $C_0$ is given by
\[
C_0=C_0(\G) = \sum_{\ell \in \mathcal{S}_0(\G)} c_{\ell}~(X_{\ell} \otimes u_{\ell}^k),
\]
where $\mathcal{S}_0(\G)=\mathcal{S}_n(\G,\mathcal{L})$ denotes the set of $\G$ orbit representatives of isotropic
rational lines which meet  $\mathcal{L}$ and the  $c_{\ell}$ are rational numbers depending only on the line $\ell$, the lattice and the congruence
condition $h$. Here we have included the coefficient $u^k$.  From now on we will omit it. We make the analogous definitions for $\G'$.
 
Recalling the (capped) homology transfer from Definition \ref{cappedhomologytransfer} and Lemma 
\ref{dualofcappedinverseimage}
 that computes it for the case in hand we have
\[
\pi^{!}C^c_n(\G) = \sum_{\x \in \mathcal{S}_n(\G)}   \tr^{\G}_{\G'}(C_\x^c) 
\]

\subsubsection{The transfer formula} 

Recall \cite{FCompo}, Lemma 3.6, that if $n$ is a square then $C_\x$ is noncompact for all $\x$ with $q(\x) =n$ and if $n$ is not a square then $C_\x$ is compact for all $\x$ with $q(\x) =n$.

We now return to proving the extension of the main theorem. As we will see that extension is a formal
consequence of what we have proved so far and the following (two) critical lemmas. We have separated the cases,
$n>0$ and $n=0$. The proofs are almost identical but the result(s) is so important we give both cases. First we treat the case for $n>0$. 

\begin{lemma}\label{vectorsandcycles}
Note we have a map $p_n:\mathcal{S}_n(\G') \to \mathcal{S}_n(\G)$. This map is the quotient map by $\Phi$ which accordingly acts transitively on the fibers of $p_n$. Then
\[
 \tr^{\G}_{\G'}(C_{\x}^c) = \sum_{\y \in p_n^{-1}(\x)}  (C_\y')^c.
\]
\end{lemma}

\begin{proof}
We know, by the corollary to Proposition \ref{msstabilizer}, that $\Phi$ acts simply transitively on the fibers of $p_n$ in the case $n$ is a square ($C_{\x}$ is not compact so the isotropy of $\x$ is trivial). In case $n$ is not a square we know 
by Proposition \ref{scstabilizer}, that the isotropy of $\x$ is a nontrivial cyclic group of order $d_{\x}$, the order of the covering $C'_{\x} \to C_{\x}$. 
Now we have
\[
 \tr^{\G}_{\G'}(C_{\x}^c) = 
\frac{1}{d_{\x}} \sum_{\eta \in \Phi} \eta \cdot {C'_{\bf x}}^c 
=\frac{1}{d_{\x}} \sum_{\eta \in \Phi} ({C'_{\eta(\bf x)}})^c.
 \]
But since $\Phi$ acts transitively on $p_n^{-1}(\x)$ and again
{\it because $d_\x$ is the order of $\Phi_\x$} by Proposition \ref{scstabilizer}, we find that the set of translates $\Phi \cdot \x$ repeats each vector $d_\x$ times in $p_n^{-1}(\x)$  and  we have 
\[
\sum_{\eta \in \Phi} ({C'_{\eta(\bf x)}})^c = d_{\x} \sum_{\y \in p_n^{-1}(\x)}  (C_\y')^c.
\] \qedhere
\end{proof}

We now treat the case $n=0$. We have a map $p_0:\mathcal{S}_0(\G') \to \mathcal{S}_0(\G)$.   This map is again the quotient map by $\Phi$ which accordingly acts transitively on the fiber of $p_0$.

\begin{lemma}\label{linesandcycles}

\[
 \tr^{\G}_{\G'}(X_{\ell}) = \sum_{\ell' \in p_0^{-1}(\ell)}  X'_{\ell'}.
\]
\end{lemma}

\begin{proof}
We know $\Phi$ acts simply transitively on the fibers of $p_0$.  .Also we know the isotropy is a nontrivial cyclic group of order $d_{\ell}$, the order of the covering $X'_{\ell} \to X_{\ell}$. 
Now we have
\[
 \tr^{\G}_{\G'}(X'_{\ell}) = 
\frac{1}{d_{\ell}} \sum_{\eta \in \Phi} \eta \cdot X'_{\ell} 
=\frac{1}{d_{\ell}} \sum_{\eta \in \Phi} (X'_{\eta{\ell}}).
 \]
But since $\Phi$ acts transitively on $p_n^{-1}(\ell)$ and $\Phi\cdot \ell$ repeats each vector in $p_0^{-1}(\ell)$ the number of times equal to its stabilizer $d_{\ell}$ we have 
\[
\sum_{\eta \in \Phi} X'_{\eta(\ell)}= d_{\ell} \sum_{\ell' \in p_0^{-1}(\ell)}  X'_{\ell'}. \qedhere
\] 
\end{proof}

As an immediate consequence we have the key result, the transfer formula, relating the homology transfer of a decomposable cycle at one level to the corresponding decomposable cycle at another level.  Since this Proposition is a formal consequence of Lemmas \ref{dualofcappedinverseimage} and \ref{vectorsandcycles}
both of which we have proved carefully for both the cases $n>0$ and $n=0$ we may safely treat only the case $n>0$.
We should emphasize that the constants $c_{\ell}$ in the definition of $C_0(\G)$ do not depend on the level
but only on the rational line $\ell$, the lattice $L$ and the congruence condition $h$. 

\begin{proposition}\label{composite} 
Let $C_n^c(\G'),n \geq 0,$ be the decomposable cycle defined above for the group $\G'$. 
Then
\[
 \pi^{!}(C^c_n(\G)) =   C_n^c(\G').
 \]
\end{proposition}

\begin{proof}
We have by  Lemma \ref{dualofcappedinverseimage}
$$\pi^{!}(C_{\x}^c) = \tr^{\G}_{\G'}(C_{\x}^c)$$
and by Lemma \ref{vectorsandcycles}
\[ 
 \tr^{\G}_{\G'}(C_{\x}^c) = \sum_{\y \in p_n^{-1}(\x)}  (C_\y')^c.
\]
Hence
\[
\pi^{!}C_n(\G)^c = \sum_{\x \in S_n(\G)}\pi^{!}(C_{\x}^c)  = \sum_{\x \in S_n(\G)}  \sum_{\y \in p_n^{-1}(\x)}  (C_\y')^c = \sum_{\mathbf{z} \in \mathcal{S}_n(\G')} C_{\mathbf{z}}^c .
\]
The proposition follows.
\end{proof}

\subsubsection{The extension of the main theorem to orbifold quotients}

We now assume we have proved the main theorem for $\G'$, ie. ($\ast_{\G'}$), 
\begin{equation}
\theta(\varphi,\G') = \sum_{n=0}^{\infty} PD(C_n^c(\G'))~q^n 
\tag{$\ast_{\G'}$}
\end{equation}
for $X'$ and prove
\begin{equation}
\theta(\varphi,\G) = \sum_{n=0}^{\infty} PD(C_n^c(\G))~q^n  
\tag{$\ast_{\G}$}
\end{equation}
Since the map $\pi^{\ast}$ on first cohomology induced by $\pi$ is injective it suffices to prove 
\[
 \pi^{\ast} \theta(\varphi,\phi\G) = \sum_{n=0}^{\infty} \pi^{\ast} PD(C_n^c(\G))~q^n.
 \]
But 
\[
\pi^{\ast} \theta(\varphi,\phi,\G) = \theta(\varphi,\phi,\G').
\]
Hence the pull-back of the left-hand side of $\ast_{\G}$ is the left-hand side of $\ast_{\G'}$.
We will now show (term-by-term) that the pull-back of the right -hand side of $\ast_{\G}$ is the right-hand side of $\ast_{\G'}$.
But by definition
$$\pi^{\ast}  PD(C_n^c(\G)) = PD (\pi^{!} C_{n}^c(\G))$$
and by Proposition \ref{composite}
\[
 \pi^{!}(C^c_n(\G)) =   C_n^c(\G').
 \]
Hence we have shown that the pull-back of the right -hand side of $\ast_{\G}$ is the right-hand side of $\ast_{\G'}$
and we have obtained the required extension of our main theorem.

\subsection{The geometry of the compactified modular curve}
It is our goal in this section to give a heuristic proof that the entire homology of the
compactified modular curve with coefficients in $\calE _{2k}$ is captured by the homology
of the one-dimension simplicial complex with coefficients in the restricted local system consisting of the barycentric subdivison of the arc of the unit circle 
joining the images $i$ and $\rho$. The resulting complex has two one-simplices and three vertices.

As a topological space the Borel-Serre compactification $\overline{X}$ is a closed two-disk but of course
it has more structure, it is a hyperbolic orbifold with  two singular points, the images of $i$ and $\rho$
which we will again denote $i$ and $\rho$. 
We now  give a description that captures more of this hyperbolic geometry. 
Recall that the fundamental domain has boundary consisting of three geodesic arcs,  the two infinite
arcs from $\rho$ and $- \overline{\rho}$ to infinity and the finite arc joining $\rho$ to $ - \overline{\rho}$.
After making the identifications to pass to the quotient the two infinite arcs are identified to from
a ray $r_{\rho}$ and the finite
arc gets folded at $i$ to the geodesic segment $\overline{\rho i}$.  After compactifying by adding a circle $X_{\infty}$ at infinity and drawing the vertical
geodesic ray $r_i$ from $i$ to its limit point $c_i$ on $X_{\infty}$ the resulting space looks like an
unzipped change purse. The opened zipper at the top of the purse is the circle at infinity  $X_{\infty}$, the sides of the purse
are the two rays $r_i$ and $r_{\rho}$ and the (folded) bottom of the purse is the geodesic arc 
$\overline{\rho i}$ joining the two orbifold points $i$ and $\rho$. In order to triangulate $\overline{X}$
we first take the midpoint $z$ of the arc $\overline{\rho i}$ and draw two infinite rays $r_z^{\pm}$ from
$z$ to $X_{\infty}$, thereby dividing $\overline{X}$ into four rectangles. Drawing appropriate diagonals of these
rectangles we get a triangulation with eight two-simplices.   In particular, the arc $\overline{\rho i}$ gets subdivided into two oriented simplices
$(z,i)$ and $(z,\rho)$ and from this we obtain the cycles described in Remark \ref{foldedms}. We let
$C_{\bullet}(\overline{\rho i},\calE)$ denote this subcomplex.  We will prove the following theorem in \cite{FMlocal}.

\begin{theorem}
The inclusion $C_{\bullet} (\overline{\rho i},\calE) \to C_{\bullet} (\overline{X}, \calE)$ is a quasi-isomorphism.
\end{theorem}
 
As promised we will conclude this section with a  short  intuitive argument motivating this theorem. 
Note that over the complement of $\overline{\rho i}$ ( a topological cylinder) the local system $\calE$
is locally free, a flat bundle determined by the action of the translation subgroup associated to the
cusp at infinity.  The arc $\overline{\rho i}$ is a deformation retract of the space $\overline{X}$ 
and from the observation immediated above it is ``intuitively clear'' that the entire local system $\calE$ retracts
onto its restriction over $\overline{\rho i}$. The reader will also verify that the homology of  
$C_{\bullet} (\overline{\rho i},\calE)$ is the homology of $\SL(2,\Z)$ with values in $E_{2k}$.

\appendix
\section{Cohomology groups associated to smooth manifolds with boundary}\label{mappingcone}

In this appendix, we discuss how one explicitly realize the isomorphism between the compactly supported cohomology of a (general) smooth manifold $\overline{X}$ with boundary and the cohomology of the mapping cone of the inclusion of the inclusion of the boundary $\partial{\overline{X}}$ in $\overline{X}$. Of course, this is in principal well-known, but for us it is critical to obtain explicit formulae for the Kronecker pairings in this setting in terms of integrals of forms over $X$ and $\overline{X}$. These will be needed for computations in this and future papers.

In this paper, $\overline{X}$ is the Borel-Serre compactification of an arithmetic quotient of the upper-half plane. In later papers, it will be the Borel-Serre compactification of a $\Q$-rank $1$ arithmetically defined locally symmetric space.

\subsection{The cohomology with compact supports and the relative cohomology}

In this section, $\overline{X}$ will be a smooth manifold with boundary $\partial \overline{X}$ and $E$
will be a flat vector bundle over $\overline{X}$.  We will consider the de Rham complexes 
$A^{\bullet}(\overline{X}, E), A^{\bullet}(\partial \overline{X},E)$ and the relative de Rham complex
$A^{\bullet}(\overline{X},\partial \overline{X}, E)$.  Henceforth when we refer to cohomology groups - we will mean cohomology groups {\it with coefficients}. 

Let $i: \partial \overline{X} \to \overline{X}$ be the inclusion and $i^{\ast} : A^{\bullet}(\overline{X}, E) \to A^{\bullet}(\partial \overline{X},E)$ be the restriction map. We note that $i^{\ast}$ fits into a short exact sequence
$$
A^{\bullet}(\overline{X},\partial \overline{X}, E) \to A^{\bullet}(\overline{X}, E) \to A^{\bullet}(\partial \overline{X},E).
$$

Let $k: A_c^{\bullet}(X,E) \to  A^{\bullet}(\overline{X}, \partial \overline{X},E)$ be the inclusion of the complex of
compactly supported forms on the open manifold $X= \overline{X} - \partial \overline{X}$ into the complex of forms on $\overline{X}$ whose restrictions
to the boundary $\partial \overline{X}$ vanish. We will need a standard result from topology, the following
result is stated in the proof of Theorem 3.43, p. 254 of \cite{Hatcher} with trivial coefficients. The basic argument using a collar neighborhood works equally well with nontrivial coefficients. The details
are left to the reader.

\begin{proposition}\label{Hatcher}
$k$ is a quasi-isomorphism, that is, $k$ induces an
isomorphism
$$
H_c^{\bullet}(X,E) \cong H^{\bullet}(\overline{X}, \partial \overline{X},E).
$$
\end{proposition}

In particular, the de Rham cohomology of the open manifold $X$ with coefficients in $E$ is dually paired with the compactly supported cohomology of $X$ with coefficients in $E^{\ast}$  of complementary degree by the integration pairing.

\subsection{From the cohomology with compact supports to the cohomology of the mapping cone}

In what follows we wish to represent the compactly-supported cohomology of $X$ with coefficients in $E$ by the cohomology of the mapping cone $C^{\bullet}$ of $i^*$, see eg \cite{Weibel}, p. 19. However we will change the sign
of the differential on $C^{\bullet}$ and shift the grading down by one. Thus we have
$$
C^i =\{ (a,b), a \in A^i (\overline{X},E), b \in A^{i-1}(\partial \overline{X},E)\}
$$
with
$$
d(a,b) = (da, i^*a - db).
$$
If $(a,b)$ is a cocycle in $C^{\bullet}$ we will use $[[a,b]]$ to denote its cohomology class. Hence we obtain

\begin{lemma}
A cocycle $(a,b)$ in $C^{\bullet}$ of degree $i$ consists of a closed form $a$ on $\overline{X}$ of degree $i$ and a form $b$ on $\partial\overline{X}$ of degree $i-1$ such that
\[
db  =  i^*a.
\]
\end{lemma}

We have a short exact sequence of cochain complexes
\begin{equation}\label{short}
A^{\bullet}(\partial \overline{X},E)[1] \to C^{\bullet} \to A^{\bullet}( \overline{X},E).
\end{equation}
Here the first map is the map $b \to (0,- b)$  and the second is projection on the first factor.
The first map is a map of complexes because the differential on $A^{\bullet}(\partial \overline{X},E)[1] $ is the
negative of the differential on $A^{\bullet}(\partial \overline{X},E)$, see \cite{Weibel}, pg. 10.
Here we  use the standard notation, see \cite{Weibel}, p. 9, $ A^{i}(\partial \overline{X},E)[1] = A^{i-1}(\partial \overline{X},E).$

We observe that if $c \in A^{\bullet}(\overline{X}, \partial \overline{X},E)$ then
$(c,0) \in C^{\bullet}$ and the induced map
$j:  A^{\bullet}(\overline{X}, \partial \overline{X},E) \to C^{\bullet}$ is a cochain map.

\begin{lemma}
The map $j:  A^{\bullet}(\overline{X}, \partial \overline{X},E) \to C^{\bullet}$ is a quasi-isomorphism.
\end{lemma}
\begin{proof}
From the short exact sequence Equation \ref{short}, see also \cite{Weibel} p. 19, 1.5.2, the cohomology of the cochain complex $C^{\bullet}$ fits into a long exact
sequence 
$$ \to H^i(C^{\bullet}) \to H^i(\overline{X},E) \to H^i(\partial \overline{X},E) \to \cdots .$$
But the  relative cohomology fits into a similar long exact sequence in the same place.
We map the exact sequence for relative cohomology to the exact sequence for the mapping cone
by mapping the class of $c \in H^i(\overline{X}, \partial \overline{X},E)$ to the class of $j(c)$.
We use the identity maps on the other terms. We claim all the resulting squares
are commutative. This is obvious for the squares not involving the coboundary map $d^{\bullet- 1}: H^{\bullet-1}(\partial \overline{X},E) \to 
H^{\bullet}(\overline{X},\partial \overline{X},E)$. For these latter squares we are required to prove an equality of cohomology classes
$$j\circ d^{i-1}([b]) = [[0,-b]].$$
Here we assume $b$ is a closed $i-1$ from on $\partial \overline{X}$ with cohomology class $[b]$.
Let $\overline{b}$ be any extension of $b$ to $\overline{X}$ then 
$$j\circ d^{i-1}([b]) = [[d \overline{b},0]].$$
But in $C^i$ we have
$$d(\overline{b},0) = (d \overline{b}, b)$$
whence $(d \overline{b},0)$ is cohomologous to $(0,- b)$ in $C^i$.
The claim is proved and the lemma follows from the five lemma.
\end{proof}

We have seen that the inclusion $k$ of the compactly supported forms $A^{\bullet}_c(X,E)$ into  $ A^{\bullet}(\overline{X}, \partial \overline{X},E)$
is a quasi-isomorphism. Composing $j$ with $k$ we obtain an inclusion 
$\ell: A_c^{\bullet}(X,E) \to C^{\bullet}$. From the above results we obtain
\begin{lemma} \label{quasiiso}
The inclusion $\ell:  A_c^{\bullet}(X,E) \to C^{\bullet}$ is a quasi-isomorphism.
\end{lemma}

\subsection{From the cohomology of the mapping cone to the cohomology with compact supports}

In this subsection we will construct a cochain map from the mapping cone of $i^*$ to
the cohomology of $X$ with compact supports (viewed as a cochain complex with zero differential) that induces the isomorphism on cohomology
inverse to the isomorphism induced by $\ell$. We assume that we have chosen a neighborhood
$V$ of the boundary and a product decomposition
$ V \cong \partial \overline{X} \times (\epsilon,0]$. We let $t' \in (\epsilon,0]$ be the normal coordinate to
$\partial \overline{X}$.
\begin{remark} 
For the case of this paper we
will use the geodesic flow projection to give
the normal coordinate $t'$ to $\partial \overline{X}$.
In more detail, we descend the map $D_{\ell}  \times [0,\infty) \to D$
induced by the extension of 
$$
(x,t') \to n_{\ell}(x) a((t')^{-1/2}) z_0
$$
to $t' =0$
to give coordinates in a product neighborhood $V$ of $\partial \overline{X}$.
\end{remark}
Hence we have arranged that the subset of $V$ defined by the equation $t'=0$ is $\partial \overline{X}$.
We let $\pi:V \to \partial \overline{X}$ be the projection. If $b$ is an $E$-valued form on $\partial \overline{X}$
we define $\tilde{b}$ on $V$ by
$$\tilde{b} = \pi^* b.$$
Let $f$ be a smooth function of the geodesic flow coordinate $t'$ which is $1$ near $t'=0$ and zero for $t' \geq \epsilon'$
for some small positive $\epsilon' \leq \epsilon$. 
We may regard $f$ as a function on a product neighborhood $U$ of $\partial \overline{X}$ by making it constant on
the $\partial \overline{X}$ factor.  We extend $f$ to all of $ \overline{X}$ by making it zero off of $U$.
Let $(a,b)$ be a cocycle in $C^i$. We need an explicit formula for a compactly supported form $\alpha$ so that $\ell(\alpha)$
is in the same  cohomology class as $(a,b)$. In fact, we now construct a map $F$
from the  mapping cone $C^{\bullet}$ to $H_c^{\bullet}(X,E)$ that induces the isomorphism $\overline{F}$ on cohomology that is inverse to 
that induced by $\ell$.

By Proposition \ref{Hatcher} we have

\begin{lemma}\label{alpha}
There exists a compactly supported closed form $\alpha $ and a form $\mu$
which vanishes on $\partial \overline{X}$ such that
$$a - d(f \tilde{b}) = \alpha + d\mu.$$
\end{lemma}
 
Suppose the degree of $(a,b)$ is $i$.
We define the cohomology class $[a,b]$ in 
the compactly supported  cohomology $H^i_c(X,E)$ to be the class of $\alpha$.
We wish to define $F: C^{\bullet} \to H_c^{\bullet}(X,E)$ by
$$F(a,b) = [a,b].$$

\begin{lemma}
$F$ is well-defined.
\end{lemma}
\begin{proof}
Given two decompositions as above,
$$a - d (f \tilde{b}) = \alpha_1 + d\mu_1 \qquad  \text{and} \qquad a- d (f \tilde{b}) = \alpha_2 + d\mu_2,$$
we obtain 
$$\alpha_1 - \alpha_2 = d(\mu_1 -\mu_2).$$
Hence the compactly supported form $\alpha_1 - \alpha_2$ is cohomology to zero in the relative complex.
But by \cite{Hatcher} the inclusion of complexes is
injective on cohomology whence $\alpha_1 - \alpha_2$ is the coboundary of a compactly supported
form. Thus the class of $\alpha$ is well-defined. 
\end{proof}

We next prove 

\begin{lemma} 
\begin{enumerate}
\item The map $F$ from $(a,b)$ to the class of $\alpha$ depends only on the cohomology class $[[a,b]]$
of $(a,b)$.
\item The induced map $\overline{F}:H^{\bullet}(C) \to H^{\bullet}_c(X,E)$ given by
$$\overline{F}([[a,b]]) = [a,b]$$
is an isomorphism.
\end{enumerate}
\end{lemma}

\begin{proof}
Suppose $(a,b)$ is the  coboundary of $(a',b')$ whence we may write 
$$ a = d a' \ \text{and} \ b = i^* a' - db'.$$
Then
$$F((a,b)) = da' - d(f (\widetilde{i^*a'} - d\widetilde{b'})) = d( a' - f(\widetilde{i^* a'} - d \widetilde{b'})).$$
Thus   $\nu = a' - f\widetilde{i^* a'} +f d \widetilde{b'}$ is a primitive for $F((a,b))$. 
Unfortunately this primitive does not vanish on the boundary in general. We now construct a new primitive
that does vanish on the boundary.
It is immediate that $ a' - f(\widetilde{i^* a'})$ vanishes on $\partial \overline{X}$. As for 
the third term note that we may obtain a new primitive for $F((a,b))$ by subtracting the exact (hence closed) form
$d(f b')$ from $\nu$.
But $$\nu - d(fb') = a' - f(\widetilde{i^* a'}) - df \wedge \widetilde{b'}.$$
Now the third term vanishes on $\partial \overline{X}$ since $df$ does.

Finally, we observe that  for a closed compactly supported form $c$ on $X$ we have  $F\circ \ell(c)$ is the class of $c$ whence $\overline{F}$ is surjective and consequently is an isomorphism.
\end{proof}

\subsection{Integral formulas for Kronecker pairings}

Let $(a,b)$ be a cocycle in the mapping cone.
Let $\eta$ be a closed form on $\overline{X}$ of degree complementary to that of $a$ (or equivalently of $(a,b)$) and $C$ be a relative cycle in $\overline{X}$ of degree equal to that of $a$. We will need integral formulas for the Kronecker pairings
\[
\langle  [\eta],[a,b] \rangle \qquad \text{and} \qquad 
\langle [a,b],C \rangle
\]
We need to be a bit careful about the non-trivial coefficients case since we must pair a group with coefficients in $E$ to a corresponding group with coefficients in the dual $E^{\ast}$. In the cases we will study here $E$ has a parallel nondegenerate symmetric bilinear form so we will not have to change to $E^{\ast}$.

From the considerations in the previous subsections we easily obtain

\begin{lemma}\label{integralformula}
Let $\eta$ be a closed form on $\overline{X}$ of degree complementary to that of $(a,b)$. Let $\alpha$ be as in Lemma \ref{alpha}. Let $C$ be a relative cycle in $\overline{X}$ of degree equal to that of $a$. 
Then
\begin{align*}
\langle [\eta],[a, b] \rangle &= \int_{\overline{X}}\eta \wedge  \alpha 
= \int_{\overline{X}} \eta \wedge a - \int_{\partial \overline{X}}  i^*\eta \wedge b \\
\langle [a,b],C \rangle &= \int_C \alpha = \int_{C}a - \int_{\partial C} b.
\end{align*}
Here the $\eta \wedge \alpha$ is a scalar valued top degree differential form obtain by pairing in the coefficients. 

Furthermore, if $\eta$ is a closed form on ${X}$ (which might not extend to $\overline{X}$), we have
\[
\langle [\eta],[a, b] \rangle = \int_{X}\eta \wedge (a - d (f \tilde{b})).
\]
\end{lemma}

\end{document}